%% file: todd-coxeter.tex
\documentclass{article}

\usepackage[utf8]{inputenc}
\usepackage[cm]{fullpage}

\usepackage{graphicx, amsmath, amsfonts, amsthm, amssymb, latexsym,
  pifont, xspace, multirow}

\usepackage{enumitem}
\usepackage{tocloft}

\usepackage{thmtools}
\usepackage[colorlinks, hypertexnames=false]{hyperref}
\hypersetup{linkcolor=blue, urlcolor=blue, citecolor=red}
\usepackage[capitalise]{cleveref}

\usepackage{tablefootnote}
\usepackage{caption}
\usepackage{subcaption}
\usepackage{pdflscape}

\usepackage{calc}
\usepackage{xcolor}

\usepackage{tikz}
\usepackage{tikz-cd}
\usetikzlibrary{decorations, arrows}
\usepackage[backend=bibtex,giveninits=true,maxnames=6]{biblatex}
\usepackage{chngcntr}
\counterwithin{figure}{section}
\counterwithin{table}{section}

\pgfdeclaredecoration{sl}{initial}{
  \state{initial}[width=\pgfdecoratedpathlength-1sp]{
    \pgfmoveto{\pgfpointorigin}
  }
  \state{final}{
    \pgflineto{\pgfpointorigin}
  }
}
\pgfarrowsdeclarecombine{aa}{aa}{angle 90}{angle 90}{angle 90}{angle 90}

\definecolor{color0}{RGB}{127, 0, 255}
\definecolor{color1}{RGB}{128, 128, 128}
\definecolor{color2}{RGB}{255, 0, 127}
\definecolor{color3}{RGB}{255, 0, 255}
\definecolor{color4}{RGB}{0, 0, 255}
\definecolor{color5}{RGB}{0, 255, 0}
\definecolor{color6}{RGB}{255, 128, 0}
\definecolor{color7}{RGB}{255, 0, 0}
\definecolor{color8}{RGB}{0, 128, 255}
\definecolor{color9}{RGB}{128, 255, 0}

\tikzstyle{vertex}=[rectangle,
draw,
fill=white,
opacity=1,
text opacity=1,
inner sep=3pt,
thick,
align=center,
minimum size=1.6em]
\tikzstyle{mylabel}=[draw=none, opacity=0]

\tikzstyle{solid black}=[->, draw=black, thick, -angle 90]
\tikzstyle{dashed black}=[->, draw=black, thick, dashed, -angle 90]

\tikzstyle{a}=[-angle 90, draw=color0, thick]
\tikzstyle{b}=[-angle 90, draw=color1, thick]
\tikzstyle{c}=[-angle 90, draw=color2, thick]

\tikzset{arrow/.style={-angle 90, shorten >=2pt, shorten <=2pt}}
\tikzset{parallel_arrow_1/.style={-angle 90,
      decoration={sl,raise=1mm},decorate, shorten >=2pt, shorten <=2pt}}
\tikzset{parallel_arrow_2/.style={-angle 90,
      decoration={sl,raise=-1mm},decorate, shorten >=2pt, shorten <=2pt}}
\tikzset{aarrow/.style={-aa, shorten >=2pt, shorten <=2pt, densely dashed}}
\tikzset{parallel_aarrow_1/.style={-aa,
      decoration={sl,raise=1mm},decorate, shorten >=2pt, shorten <=2pt, densely dashed}}
\tikzset{parallel_aarrow_2/.style={-aa,
      decoration={sl,raise=-1mm},decorate, shorten >=2pt, shorten <=2pt, densely dashed}}

\newcommand{\set}[2]{\ensuremath{\{\: #1 \: :\: #2 \:\}}}
\newcommand{\genset}[1]{\ensuremath{\langle\: #1 \:\rangle}}
\renewcommand{\to}{\longrightarrow}

\newcommand{\N}{\mathbb{N}}
\renewcommand{\P}{\mathcal{P}}

\newtheorem{thm}{Theorem}[section]
\newtheorem{prop}[thm]{Proposition}
\newtheorem{cor}[thm]{Corollary}
\newtheorem{lemma}[thm]{Lemma}
\newtheorem{defi}[thm]{Definition}
\newenvironment{de}{\begin{defi} \rm}{\end{defi}}
\newtheorem{exam}[thm]{Example}
\newenvironment{ex}{\begin{exam} \rm}{\end{exam}}
\newtheorem{remark}[thm]{Remark}

\numberwithin{equation}{section}
\newtheorem*{lemma-no-num}{Lemma}

\newcommand{\defn}[1]{\textbf{\textit{#1}}}
\newcommand{\libsemigroups}{\textsf{libsemigroups}~\cite{Mitchell2021ab}\xspace}

\title{The Todd-Coxeter Algorithm for Semigroups and Monoids}
\date{\today}
\author{T. D. H. Coleman, J. D. Mitchell, F. L. Smith, and M. Tsalakou}

\begin{document}
\maketitle

\begin{abstract}
  In this paper we provide an account of the Todd-Coxeter algorithm
  for computing congruences on semigroups and monoids. We also give a novel
  description of an analogue for semigroups of the so-called Felsch strategy
  from the Todd-Coxeter algorithm for groups.
\end{abstract}

\tableofcontents

\section{Introduction}

In this article we describe the Todd-Coxeter algorithm for congruence
enumeration in semigroups and monoids. The essential purpose of this algorithm
is to compute the action of a finitely presented semigroup or monoid on the
equivalence classes of a left, right, or two-sided congruence. Most existing
implementations (see, for example, ACE~\cite{Havas1999aa}, GAP~\cite{GAP4}, and
MAF~\cite{Williams2016aa}) and expository accounts (see, for example,~\cite[p.
351]{Rotman1995aa},~\cite[Sections 4.5, 4.6 and Chapter 5]{Sims1994aa},
and~\cite[Chapter 5]{B.-Eick2004aa}) of the Todd-Coxeter algorithm relate
to the enumeration of the cosets of a subgroup of a finitely presented
group; or, more precisely, to the production of a permutation
representation of the action of the group on the cosets. This was
extended to linear representations by Linton~\cite{Linton1991aa,
Linton1993aa}. The purpose of this article is to provide an expository,
but more or less complete, account of the Todd-Coxeter algorithm for
semigroups and monoids.

The Todd-Coxeter algorithm is not a single procedure, but rather an infinite
collection of different but related procedures. In the literature for finitely
presented groups, examples of procedures in this collection are referred to as
\defn{coset enumerations}; see, for example,~\cite[Sections 4.5, 4.6 and
Chapter 5]{Sims1994aa}. Coset enumerations are also not algorithms, at least by
some definitions, in that they might consist of infinitely many steps, and they
do not always terminate. In fact, a coset enumeration terminates if and only if
the subgroup whose cosets are being enumerated has finite index; for further
details see~\cref{section-validity}. This is not to say, however, that the
number of steps, or the run time, of a coset enumeration can be predicted in
advance; it is relatively straightforward to find examples of finite
presentations for the trivial group where the number of steps in a coset
enumeration is arbitrarily high. For instance, the group presentation $\langle
a, b\ |\ ab ^ n, a ^ n b ^ {n + 1}, b ^ n ab ^ {-1}a ^ {-1}\rangle$ defines the
trivial group, and is likely that at least $n$ steps are required in any coset
enumeration for this presentation; see~\cite{Jura1978aa}.  Although this might
seem rather negative, nothing more can really be expected. For example, if a
coset enumeration for the trivial subgroup of a finitely presented group $G$
does successfully terminate, then the output can be used to solve the word
problem in $G$.  It is well-known by the theorem of
Novikov~\cite{Novikov1955aa} and Boone~\cite{Boone1958} that the word problem
for finitely presented groups is undecidable and so no procedure that solves
the word problem, including any coset enumeration, can be expected to terminate
in every case. Analogous statements about the undecidability of the word
problem hold for finitely presented semigroups and monoids; see~\cite[Chapter
12]{Rotman1995aa}.

Congruences are to semigroups and monoids what cosets are to groups, and so we
will refer to \defn{congruence enumeration} for finitely presented semigroups
and monoids as the analogue of coset enumeration for groups. We will not
consider the special case of enumerating the cosets of a subgroup of a group
separately from the more general case of enumerating classes of a congruence of
a semigroup.

The first computer implementation of the Todd-Coxeter algorithm is attributed
to Haselgrove in 1953 by Leech~\cite{Leech1963aa}.  Several authors (for example, Neub\"user in \cite{Neubuser1982aa} and Walker in \cite{Walker1992aa}) comment
that this may be the first computer implementation of an algorithm for groups,
possibly representing the starting point of the field of computational group
theory.  Neumann~\cite{Neumann1968aa} adapted the algorithm
in~\cite{Todd1936aa} to semigroups, and Jura~\cite{Jura1978aa} proved that
Neumann's adaptation was valid.  One congruence enumeration strategy
is described in Ru\v{s}kuc~\cite[Chapter 12]{Ruskuc1995aa} as well as variants
for computing Rees congruences and minimal ideals.  Stephen~\cite[Chapter
  4]{Stephen1987aa} also describes a variant of Todd-Coxeter, which can be used
to solve the word problem by constructing only part of the action of a finitely
presented semigroup or monoid on itself by right multiplication. Stephen's
procedure~\cite[Section 4.1]{Stephen1987aa} is similar to one of the two main
strategies for congruence enumeration described in \cref{section-hlt}; see also
\cref{section-variants}.

Versions of the Todd-Coxeter algorithm for semigroups were
implemented in FORTRAN by Robertson and \"Unl\"u~\cite{robertson-unlu}; in C by
Walker~\cite{Walker1992aa, Walker1992ab}; and in GAP by Pfeiffer~\cite{GAP4}.
The C++ library \libsemigroups (by three of the present authors) contains a
flexible optimized implementation of the different versions of the Todd-Coxeter
algorithm described in this document for semigroups and monoids.

The paper is organised as follows: some mathematical prerequisites from the
theory of monoids are given in \cref{section-prereqs}; in
\cref{section-enumeration} we define the infinite family of procedures for enumerating
a right congruence on a monoid; the validity of congruence
enumeration is shown in \cref{section-validity}; in
\cref{section-non-abstract} we discuss how congruence enumeration can be used to compute
congruences of a monoid that is not given \textit{a priori} by a finite
presentation; in \cref{section-hlt} and \cref{section-felsch} we
describe the two main strategies for congruence enumeration; in
\cref{section-implement} we discuss some issues related to the implementation
of congruence enumeration; and, finally, in \cref{section-variants} we present
some variants: for enumerating congruences of monoids with zero elements, Rees
congruences, and Stephen's procedure~\cite[Section 4.1]{Stephen1987aa}.

In \cref{appendix-extended-examples} we give a number of extended examples,
and in \cref{appendix-benchmarks} there is a comparison of the performance of
the implementations in \libsemigroups of various strategies when applied
to a number of group, monoid, or semigroup presentations.
The authors have also implemented basic versions of the algorithms described in
this paper in \textsf{python3} as a concrete alternative to the description
given here; see~\cite{Mitchell2021aa}.


\section{Prerequisites}\label{section-prereqs}

In this section, we provide some required mathematical prerequisites.
The contents of this section are well-known in the
theory of semigroups; for further background information we refer the reader
to~\cite{Howie1995aa}.

If $\mu$ is an equivalence relation on a set $X$, then we denote by $x / \mu$
the equivalence class of $x$ in $\mu$ and by $X / \mu$ the set $\set{ x /
\mu}{x\in X}$ of equivalence classes. The least equivalence relation on $X$
with respect to containment is $\Delta_X = \set{(x, x)\in X \times X}{x\in X}$;
we refer to this as the \defn{trivial equivalence} on $X$.

If $S$ is a semigroup and $R\subseteq S\times S$, then there is an
\defn{(two-sided) elementary sequence with respect to $R$} between $x, y\in S$
if  $x = y$ or $x=s_1, s_2, \ldots, s_m=y$, for some $m \geq 2$, where
$s_i=p_iu_iq_i$, $s_{i+1}=p_iv_iq_i$, with $p_i, q_i\in S ^ 1$ and $(u_i,
  v_i) \in R$ for all $i$. If $p_i = \varepsilon$ for all $i$ in an
elementary sequence, then we refer to the sequence as a \defn{right
  elementary sequence}; \defn{left elementary sequences} are defined
analogously.
We define $R ^ {\#} \subseteq S \times S$ so that $(x,y)\in R^{\#}$ if and only
if there is an elementary sequence between $x$ and $y$ with respect to $R$.
Note that $R^{\#}$ is the least congruence on the semigroup $S$ containing $R$;
see~\cite[Section 1.5]{Howie1995aa}. The least right congruence containing a
subset of $S \times S$ is defined analogously, using right elementary sequences
rather than two-sided elementary sequences; we will not reserve any special
notation for such least right congruences. Least left congruences are defined
dually.

Throughout this article we write functions to the right of their arguments, and
compose from left to right. Let $S$ be a semigroup and let $X$ be a set. A
function $\Psi:X\times S\to X$ is a \defn{right action} of $S$ on $X$ if  $((x,
s)\Psi, t)\Psi=(x, st)\Psi$ for all $x\in X$ and for all $s, t \in S$. If in
addition $S$ has an identity element $e$, we require $(x, e)\Psi = x$ for all
$x\in X$ also.

If $S$ is a semigroup, then we may adjoin an identity to $S$ (if necessary) so
that $S$ is a monoid. We denote $S$ with an adjoined identity $1_S$ by $S ^ 1$.
If $\Psi:X\times S\to X$ is a right action of a semigroup $S$ on a set $X$, then
$\Psi ^ 1 :X\times S ^ 1\to X$ defined by $ (x, s)\Psi ^ 1 = (x, s)\Psi$ for
all $x\in X$ and $s\in S$, and $(x, 1_S)\Psi ^ 1 = x$ for all $x\in X$, is a
right action of $S ^ 1$ on $X$ also.

For the sake of brevity, we will write $x\cdot s$ instead of $(x, s)\Psi$, and
we will say that $S$ acts on $X$ on the right.  The \defn{kernel} of a
function $f: X\to Y$, where $X$ and $Y$ are any sets, is the equivalence
relation
\[
  \ker(f) = \set{(x,y)\in X\times X}{(x)f = (y)f}.
\]
If $X$ is any set,
then $X ^ X$ denotes the set of all functions from $X$ to $X$. Endowed with the
operation of function composition, $X ^ X$ is a monoid, called the \defn{full
  transformation monoid} on $X$.  A right action $\Psi$ of a semigroup $S$ on a
set $X$ induces a homomorphism $\Psi': S \to X ^ X$ defined by $(s)\phi:
  x\mapsto x \cdot s$ for all $s\in S$ and all $x\in X$.

The \defn{kernel} of a right action $\Psi$ of a semigroup $S$ on a set $X$ is
the kernel of the function $\Psi': S \to X ^ X$,
$$\ker(\Psi)
  = \set{(s, t)}{(s)\Psi' = (t)\Psi'}
  = \set{(s, t)\in S\times S}{(x, s)\Psi = (x, t)\Psi,\ \text{for
      all }x\in X}.\footnote{This is an abuse of notation since $\ker(\Psi)$
      might mean its kernel as an action or as a function. However, the reader
    may be reassured to know that if $\Psi$ is an action, we will never use
  $\ker(\Psi)$ to mean its kernel as a function.}$$
It is
straightforward to verify that the kernel of a homomorphism, and the kernel of a
right action of a semigroup $S$, is a congruence on $S$.

If $S$ acts on the sets $X$ and $Y$ on the right, then we say that $\lambda:
  X\to Y$ is a \defn{homomorphism of right actions} if $(x \cdot
  s) \lambda=(x)\lambda\cdot s$ for all $x \in X$ and $s\in S^1$ (see
\cref{fig-homo-actions} for a diagram).  An \defn{isomorphism} of right
actions is a bijective homomorphism of right actions.

\begin{figure}
  \centering
  \begin{tikzcd}
    X\times S  \arrow[r] \arrow[d]
    & Y \times S \arrow[d] \\
    X \arrow[r, "\lambda"] & Y
  \end{tikzcd}
  \caption{A commutative diagram illustrating a homomorphism of actions where
    $X\times S \to Y\times S$ is the function defined by $(x, s) \mapsto
      ((x)\lambda, s)$.}
  \label{fig-homo-actions}
\end{figure}

If $A$ is a set, then a \defn{word} over $A$ is a finite sequence $(a_1,
\ldots, a_n)$ where $a_1, \ldots, a_n\in A$ and $n \geq 0$; say that the
\defn{length} of the word is $n$. We will write $a_1 \cdots a_n$ rather than
$(a_1, \ldots, a_n)$, and will refer to $A$ as an \defn{alphabet} and $a\in A$
as a \defn{letter}.  If $A$ is any set, then the \defn{free semigroup} on $A$
is the set $A ^ +$ consisting of all words over $A$ of length at least $1$ with
operation given by the juxtaposition of words. The \defn{free monoid} on $A$ is
the set $A ^ + \cup \{\varepsilon\}$ with the operation again the juxtaposition
of words and where $\varepsilon$ denotes the unique word of length $0$, the
so-called \defn{empty word}; the free monoid on $A$ is denoted $A ^ *$.

The next proposition might be viewed as an analogue of the Third Isomorphism
Theorem for semigroups (\cite[Theorem 1.5.4]{Howie1995aa}) in the context of
actions. More precisely, \cref{prop-quotient-quotient} will allow us to replace
the action of a free monoid on a quotient of a quotient by an action on a
single quotient. We require the following definition. Let $S$ be a semigroup,
let $\rho$ be a congruence of $S$, and let $\sigma$  be a right congruence of
$S$ such that $\rho \subseteq \sigma$. The set $\sigma / \rho$ defined as
\[
    \sigma / \rho = \set{(x/\rho, y/\rho)\in S/\rho\times S/\rho}{(x,y) \in
      \sigma}
  \]
 is a right congruence on $S/\rho$.

\begin{prop}\label{prop-quotient-quotient}
  Let $S$ be a semigroup, let $\rho$ be a congruence of $S$, and let $\sigma$
  be a right congruence of $S$ such that $\rho \subseteq \sigma$.
  The right actions of $S$ on $(S / \rho) / (\sigma / \rho)$ and $S / \sigma$ defined by
  \[
    (u/ \rho)/ (\sigma / \rho) \cdot v = (uv/ \rho)/ (\sigma / \rho)
    \quad
    \text{and}
    \quad
    u/\sigma\cdot v = uv/\sigma\quad \text{for all}\quad u, v\in S
  \]
  are isomorphic.
\end{prop}

The proof of \cref{prop-quotient-quotient} is similar to the proof
of~\cite[Theorem 1.5.4]{Howie1995aa}. An analogue of
\cref{prop-quotient-quotient} can be formulated for left congruences, but we do
not require this explicitly; see \cref{prop-dagger}, and the surrounding text,
for further details.


We also require the following result, the proof of which is routine, and hence omitted.

\begin{prop}\label{prop-quotient-action}
  Let $S$ be a semigroup, let $X$ be a set, and let $\Psi:X \times S\to X$ be a right action.
  If $\mu\subseteq \ker(\Psi)$ is a congruence on $S$, then
  $\overline{\Psi}: X \times S/\mu \to X$  defined by
  \[
    (x, s/\mu) \overline{\Psi} = (x, s)\overline{\Psi}
  \]
  for all $x \in X$ and $s\in S$, is a right action that is isomorphic to $\Psi$.
\end{prop}
%
%

\section{Congruence enumeration}\label{section-enumeration}

In this section we define what we mean by a congruence enumeration
for a semigroup or monoid, and establish some further notational
conventions.  Congruence enumeration, as described in this section,
will provide the general context for the more explicit algorithms discussed in
\cref{section-hlt} and \cref{section-felsch}; this section is based, at least
in spirit, on~\cite[Section 5.1]{Sims1994aa} and is also influenced
by~\cite{Stephen1987aa, Stephen1990aa}.

The purpose of a congruence enumeration is to determine, in some sense, the
structure of a finitely presented monoid; or more generally, a congruence on
such a monoid.  For the sake of simplicity, we will assume throughout that $\P
= \langle A | R \, \rangle$ is a monoid presentation defining a monoid $M$,
where $A$ is some totally ordered finite alphabet, and $R$ is a (possibly
empty) finite subset of  $A ^ * \times A ^ *$.   Additionally, we suppose that
$S$ is a finite subset of $A ^ * \times A ^ *$. We write $R ^ {\#}$ to be the
least two-sided congruence on $A ^ *$ containing $R$, and denote by $\rho$ the
least right congruence on $A ^ *$ containing $R ^ {\#}$ and $S$.  In this
notation, the monoid $M$ defined by the presentation $\P$ is $A ^ * / R ^
{\#}$. If the number of congruence classes of $\rho$ is finite, then the output
of a congruence enumeration is a description of the natural right action of $A
^ *$ on the congruence classes of $\rho$. In particular, the output yields the
number of such classes, can be used to determine whether or not two words in $A
^ *$ belong to the same class, and provides a homomorphism from $A ^ *$ to the
full transformation monoid ${(A ^ * / \rho)} ^ {(A ^ * / \rho)}$ on the set of
congruence classes of $\rho$. Since it is undecidable whether a finitely
presented monoid is finite or not~\cite{Markov} (see also~\cite[Remark
4]{Cain2009aa}), an upper bound for the number of steps required for a
  congruence enumeration for $\langle A | R \, \rangle$ to terminate cannot be
  computed as a function of $|A|$ and $|R|$. If the number of congruence
  classes of $\rho$ is infinite, then the enumeration will not terminate.  If
  the set $S$ is empty, then a congruence enumeration, if it terminates, gives
  us a description of the monoid $A ^ * /R ^ {\#}$ together with the natural
  right action of $A ^ *$ on $A ^ * / R ^ {\#}$.

Perhaps it is more natural to want to determine a right congruence of the
monoid $M$ defined by the presentation $\P$ rather than on the free monoid $A ^
  *$. However, by \cref{prop-quotient-quotient}, to determine the
right congruence on $M = A ^ * / R ^ {\#}$ generated by
\[
  \set{(s/R ^ {\#}, t/R ^ {\#})}{(s, t)\in S}
\]
it suffices to determine the right congruence $\rho$ on $A ^  *$ containing
$R ^ {\#}$ and the set $S$.

It is only for the sake of convenience that we have chosen to consider finitely
presented monoids, rather than finitely presented semigroups.  To apply the
algorithms described in this article to a semigroup that is not a monoid,
simply adjoin an identity, perform the algorithm, and disregard the adjoined
identity in the output.

The choice of ``right'' rather than ``left'' in the previous paragraphs was
arbitrary. If $w = a_1a_2\cdots a_n\in A ^ *$, then we denote the reversed word
$a_n\cdots a_2a_1$ by $w ^ {\dagger}$. If $W \subseteq A ^ * \times A ^ *$,
then we denote by $W ^ \dagger$ the set $\set{(u^ {\dagger}, v ^ \dagger)}{(u,
v) \in W}$.
\begin{prop}\label{prop-dagger}
  Let $A$ be an alphabet,  let $R, S\subseteq A ^ * \times A ^ *$,
  let $R ^ {\#}$ be the least congruence on $A ^ *$ containing $R$, and let
  $\rho$ be the least left congruence on $A ^ *$ containing $R ^ {\#}$ and $S$.
  Then $\rho ^ {\dagger}$ is the least right congruence on $A ^ *$ containing
  ${R ^ {\#}} ^ {\dagger}$ and $S ^ {\dagger}$.
\end{prop}
It follows from \cref{prop-dagger} that any algorithm for computing right
congruences can be used for left congruences, by simply reversing every word.


The inputs of a congruence enumeration are the finite alphabet $A$, the finite
set of relations $R \subseteq A ^ * \times A ^ *$, the finite set $ S \subseteq
A ^ * \times A ^ *$, and a certain type of digraph that is defined in the next
section.

\subsection{Word graphs}

One of the central components of the proofs presented in
\cref{section-validity} is that of a word graph, which is used in a natural
representation of equivalence relations on a free monoid. Let $A$ be any
alphabet and let $\Gamma = (N, E)$ be a digraph with non-empty finite set of
nodes $N\subseteq \N$ with $0\in N$ and edges $E\subseteq N \times A \times N$.
Then, following \cite{Stephen1987aa, Stephen1990aa}, we refer to $\Gamma$ as a
\defn{word graph}. The word graph $\Gamma = (\{0\}, \varnothing)$ is the
\defn{trivial word graph}.

If $(\alpha, a, \beta)\in E$ is an edge in a word graph $\Gamma$, then $\alpha$
is the \defn{source}, $a$ is the \defn{label}, and $\beta$ is the \defn{target}
of $(\alpha, a, \beta)$. An edge $(\alpha, a, \beta)$ is said to be
\defn{incident} to its source $\alpha$ and target $\beta$.

A word graph $\Gamma$ is \defn{deterministic} if for every node $\alpha$ and
every letter $a\in A$ there is at most one edge of the form $(\alpha, a,
\beta)$ in $\Gamma$. A word graph $\Gamma$ is \defn{complete} if for every node
$\alpha$ and every letter $a\in A$ there is at least one edge incident to
$\alpha$ labelled by $a$ in $\Gamma$.

If $\alpha, \beta \in N$, then an \defn{$(\alpha, \beta)$-path} is a sequence
of edges $(\alpha_1, a_1, \alpha_{2}), \ldots, (\alpha_{n}, a_{n}, \alpha_{n +
1})\in E$ where $\alpha_1 = \alpha$ and $\alpha_{n + 1} = \beta$ and $a_1,
\ldots, a_{n}\in A$ ; $\alpha$ is the \defn{source} of the path; the word
$a_1\cdots a_{n}\in A ^ *$ \defn{labels} the path; $\beta$ is the \defn{target}
of the path; and the \defn{length} of the path is $n$.  If $\alpha, \beta \in
N$ and there is an $(\alpha, \beta)$-path in $\Gamma$, then we say that $\beta$
is \defn{reachable} from $\alpha$. If $\Gamma = (N, E)$ is a word graph and
$\mathfrak{P}(A ^ * \times A ^ *)$ denotes the power set of $A ^ * \times A ^
*$, then the \defn{path relation of $\Gamma$} is the function $\pi_{\Gamma}: N
\to \mathfrak{P}(A ^ * \times A ^ *)$ defined by $(u, v)\in
(\alpha)\pi_{\Gamma}$ if there exists a node $\beta$ such that $u$ and $v$ both
label  $(\alpha, \beta)$-paths in $\Gamma$. If $\Gamma$ is a word graph and
$\alpha$ is a node in $\Gamma$, then $(\alpha)\pi_{\Gamma}$ is reflexive and
symmetric, and $(\alpha)\pi_{\Gamma}$ is transitive for all $\alpha$ if and
only if $\Gamma$ is deterministic. In particular, $(\alpha)\pi_{\Gamma}$ is an
equivalence relation for all $\alpha$ if and only if $\Gamma$ is deterministic.
If $R\subseteq A ^ *\times A ^ *$, $\Gamma$ is a word graph, and $\pi_{\Gamma}$
is the path relation of $\Gamma$, then we say that $\Gamma$ is \defn{compatible
with $R$} if $R\subseteq (\alpha)\pi_{\Gamma}$ for every node $\alpha$ in
$\Gamma$.

Suppose that $\Gamma = (N, E)$ is a word graph with nodes $N$ and edges
$E\subseteq N\times A \times N$, and $\kappa\subseteq N \times N$ is an
equivalence relation. The \defn{quotient $\Gamma/\kappa$ of $\Gamma$ by
$\kappa$} is the word graph with nodes $\set{\min \alpha/\kappa}{\alpha \in
N}$, and edges $(\min \alpha /\kappa, a, \min \beta/\kappa)$ whenever $(\alpha,
a, \beta) \in E$. It is routine to verify that if $\Gamma= (N, E)$ is a word
graph, $\kappa$ is any equivalence relation on $N$, and $\kappa'$ is any
equivalence relation on the set of nodes of $\Gamma/ \kappa$, then
$(\Gamma/\kappa) / \kappa' = \Gamma / \kappa''$ where $\kappa''$ is the least
equivalence on $N$ containing $\kappa$ and every $(\alpha, \beta)\in N \times
N$ where $(\min \alpha /\kappa, \min \beta/\kappa) \in \kappa'$. If $\Gamma_1=
(N_1, E_1)$ and $\Gamma_2= (N_2, E_2)$ are word graphs over the same alphabet
$A$, then $\phi:N_1\to N_2$ is a \defn{homomorphism} if $(\alpha, a, \beta)\in
E_1$ implies $((\alpha)\phi, a, (\beta)\phi)\in E_2$; and we write $\phi:
\Gamma_1\to \Gamma_2$. An \defn{isomorphism} of word graphs $\Gamma_1$ and
$\Gamma_2$ is a bijection $\phi: \Gamma_1\to \Gamma_2$ such that both $\phi$
and $\phi ^ {-1}$ are homomorphisms. If $\kappa$ is any equivalence relation on
a word graph $\Gamma = (N, E)$, then the function $\phi: \Gamma \to \Gamma /
\kappa$ defined by $(\alpha)\phi= \alpha/ \kappa$ is a homomorphism.
Conversely, if $\phi: \Gamma_1 \to \Gamma_2$  is a homomorphism, then
$(\Gamma_1)\phi := \set{(\alpha)\phi\in N_2}{\alpha\in N_1}$ is isomorphic to
the quotient $\Gamma_1 / \ker(\phi)$. If $(\alpha_1, a_1, \alpha_{2}), \ldots,
(\alpha_{n}, a_{n}, \alpha_{n + 1})\in E_1$ is a path in $\Gamma_1$ and $\phi:
\Gamma_1\to \Gamma_2$ is a homomorphism, then, by definition, $((\alpha_1)\phi,
a_1, (\alpha_{2})\phi), \ldots, ((\alpha_{n})\phi, a_{n}, (\alpha_{n +
1})\phi)\in E_2$ is a path in $\Gamma_2$ with the same label $a_1 \cdots a_n
\in A ^ *$. In this way, we say that homomorphisms of word graphs preserve
paths and labels of paths. This leads to the following straightforward lemma.

\begin{lemma}\label{lemma-path-equiv}
  Let $\Gamma_1= (N_1, E_1)$ and $\Gamma_2= (N_2, E_2)$ be word graphs, and let
  $\pi_{\Gamma_1}: N_1 \to \mathfrak{P}(A ^ * \times A ^ *)$ and
  $\pi_{\Gamma_2}:N_2\to \mathfrak{P}(A ^ * \times A ^ *)$ be the path
  relations of $\Gamma_1$ and $\Gamma_2$, respectively. If $\phi: \Gamma_1\to
  \Gamma_2$ is a homomorphism and $\alpha\in N_1$ is arbitrary, then
  $(\alpha)\pi_{\Gamma_1}\subseteq((\alpha)\phi) \pi_{\Gamma_2}$.
\end{lemma}

Before giving the definition of a congruence enumeration, we highlight that
many accounts of the Todd-Coxeter algorithm (see, for
instance,~\cite{B.-Eick2004aa, Neubuser1982aa, Sims1994aa}) are not formulated
in terms of digraphs, but rather as a \textit{table} whose rows are labelled by
a set $C$ of non-negative integers, and columns are labelled by the generating
set $A$. If $\Gamma = (N, E)$ is a deterministic word graph and $f: C \to N$ is
a bijection such that $(0)f = 0$, then the value in the row labelled $c$ and
column labelled $a$ is $f ^ {-1}$ of the target of the unique edge in $\Gamma$
with source $(c)f$ and label $a$. According to
Neub{\"u}ser~\cite{Neubuser1982aa}, until the 1950s, congruence enumeration was
often performed by hand, and, in this context, using tables is more
straightforward than using graphs. On the other hand, the language of word
graphs provide a more accessible means of discussing congruence enumeration in
theory.


\subsection{The definition}

Recall that we suppose throughout that $\langle A| R\rangle$ for some
$R\subseteq A ^ * \times A ^ *$ is a finite monoid presentation. Additionally,
we will suppose throughout that $A$ is a totally ordered alphabet. A congruence
enumeration is a sequence of the following steps \textbf{TC1}, \textbf{TC2},
and \textbf{TC3} where the input to the $i$-th step (where $i\in \N$) is
$(\Gamma_{i}, \kappa_{i})$ for some word graph $\Gamma_{i}$ with a totally
ordered set of nodes $N_i$ and set of edges $E_i$, and some equivalence
relation  $\kappa_{i}\subseteq N_i \times N_i$.

\begin{description}
  \item[TC1 (define a new node).]\label{description-tc-1}
    If $\alpha$ is a node in $\Gamma_{i}$ and there is no edge in $\Gamma_{i}$
    labelled by $a\in A$ with source $\alpha$, then we define $\Gamma_{i + 1}$
    to be the word graph obtained from $\Gamma_{i}$ by adding the new node
    $\beta := 1 + \max \bigcup_{j \leq i} N_j$ and the edge $(\alpha, a,
    \beta)$. We define $\kappa_{i + 1} := \kappa_{i} \cup \{(\beta, \beta)\}$.
  \item[TC2 (follow the paths defined by a relation).]\label{description-tc-2}
    Suppose that $\alpha\in N_i$ and $(u,v)\in R$ where $u = u_1a$ and
    $v= v_1b$ for some $u_1, v_1\in A ^ {*}$ and $a, b \in A$.
    \begin{enumerate}
      \item
            If $u$ and $v_1$ label paths from $\alpha$ to some nodes $\beta,
            \gamma \in N_i$ in $\Gamma_{i}$, respectively, but $\gamma$ is not
            the source of any edge labelled by $b$, then we set $\Gamma_{i +
            1}$ to be the word graph obtained from $\Gamma_{i}$ by adding the
            edge $(\gamma, b, \beta)$ and we define $\kappa_{i + 1} :=
            \kappa_{i}$.

      \item
            The dual of (a) where there are paths labelled by $u_1$ and $v$ to
            nodes $\beta$ and $\gamma$, respectively, but $\beta$ is not the
            source of any edge labelled by $a$.

      \item
            If $u$ and $v$ label paths from $w$ to some nodes $\beta$ and
            $\gamma$, respectively, and $\beta \neq \gamma$, then we define
            $\Gamma_{i + 1} := \Gamma_{i}$ and $\kappa_{i + 1}$ to be the least
            equivalence containing $\kappa_{i}$ and $(\beta , \gamma)$.
    \end{enumerate}
    Note that conditions (a), (b), and (c) are mutually exclusive, and it
    may be the case that none of them hold.

    \begin{figure}
      \centering
      \includegraphics{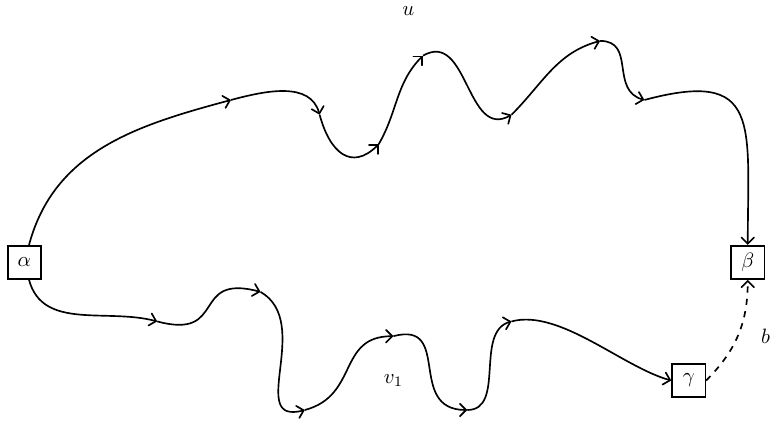}
      \caption{A diagram representing \textbf{TC2}(a), solid lines correspond
      to paths in $\Gamma_{i}$ and the dashed edge is the one defined in
      \textbf{TC2}(a).}\label{fig-TC2}
    \end{figure}

  \item[TC3 (process coincidences or a determination).]\label{description-tc-3}
    We define $\Gamma_{i + 1}$ to be the quotient of $\Gamma_{i}$ by
    $\kappa_{i}$ and define $\kappa_{i + 1}$ to be the least equivalence on
    $N_{i + 1}$ containing every $(\beta, \gamma)$, $\beta \neq \gamma$, for
    which there exist $\alpha \in N_{i + 1}$ and $a\in A$ such that $(\alpha,
    a, \beta), (\alpha, a, \gamma)\in E_{i + 1}$. Recall that the quotient
    $\Gamma_i/\kappa_i$ of $\Gamma_i$ by $\kappa_i$ is the word graph with
    nodes $\set{\min \alpha/\kappa_i}{\alpha \in N_i}$ and hence after an
    application of \textbf{TC3} each node in $\Gamma_{i + 1}$ is set to be
    equal to the minimum of the set of nodes in its equivalence class in
    $\kappa_i$.
\end{description}

There are only finitely many possible quotients of any word graph. Hence if
\textbf{TC3} is applied repeatedly, then after finitely many iterations the
output $\kappa_{i + 1}$ will equal $\Delta_{N_i}$, and $\Gamma_{i + 1}$ and
$\Gamma_{i}$ will be equal.

If $w$ labels an $(\alpha, \beta)$-path $P$ in some $\Gamma_{i}$, then neither
\textbf{TC1} nor \textbf{TC2} changes any of the edges belonging to $P$. Hence
if $\Gamma_{i + 1}$ is obtained from $\Gamma_{i}$ by applying \textbf{TC1} or
\textbf{TC2}, then $w$ labels an $(\alpha,\beta)$-path in $\Gamma_{i + 1}$
also. If $\Gamma_{i + 1}$ is obtained from $\Gamma_{i}$ by applying
\textbf{TC3}, then $\Gamma_{i + 1} $ is a homomorphic image of $\Gamma_{i}$. As
already noted, homomorphisms preserve paths, and so $w$ labels a path in
$\Gamma_{i + 1}$ also.

We can now formally define a congruence enumeration.


\begin{de}[\textbf{\textit{Congruence enumeration.}}]\label{de-2-sided-enumeration-process}
  Suppose that $A$ is a finite alphabet, that $R, S \subseteq A ^ * \times A ^
  *$ are finite, that $\rho$ is the least right congruence on $A ^ *$
  containing both $R ^ {\#}$ and $S$, and that $\Gamma_{1} = (N_1, E_1)$ is a
  word graph with path relation $\pi_{\Gamma_1}: N_1 \to \mathfrak{P}(A ^ *
  \times A ^ *)$ such that $(0)\pi_{\Gamma_1}\subseteq \rho$ and $\kappa_{1} =
  \Delta_{N_1}$. Then  a \defn{congruence enumeration for $\rho$} with input
  $(\Gamma_{1}, \kappa_{1})$ consists of:
  \begin{enumerate}
    \item
          For every $(u, v)\in S$, by repeatedly applying \textbf{TC1} (if
          necessary), add edges to $\Gamma_{1}$ so that it contains paths
          labelled by $u$ and $v$ both with source $0$.
    \item
          Apply \textbf{TC2} to $0$ and every $(u, v) \in S$.
  \end{enumerate}
  If $(\Gamma_{m}, \kappa_{m})$ is the output of steps (a) and (b), then the
  enumeration is concluded by performing any sequence of applications of
  \textbf{TC1}, \textbf{TC2}, and \textbf{TC3} such that the following
  conditions hold for $\Gamma_{i} = (N_i, E_i)$ for every $i\in \N$, $i\geq m$:
  \begin{enumerate}
    \addtocounter{enumi}{2}
    \item
          If $\alpha \in N_i$ and there is no edge incident to $\alpha$ with
          label $a\in A$, then there exists $j\geq i$ such that either:
          $\alpha$ is no longer a node in $\Gamma_j$; or there is an edge
          incident to $\alpha$ with label $a$ in $\Gamma_j$.

    \item
          If  $\alpha\in N_i$ and $(u, v)\in R$, then there exists $j\geq i$ such that either:
          $\alpha$ is no longer a node in $\Gamma_j$ or $(u, v)\in (\alpha)\pi_{\Gamma_j}$.

    \item
          If $\kappa_{i}\neq \Delta_{N_i}$ for some $i$, then there exists
          $j\geq i$ such that $\kappa_j = \Delta_{N_j}$.
  \end{enumerate}
\end{de}


The initial value of the word graph $\Gamma_{1}$ that forms the input to a
congruence enumeration is usually either the trivial word graph or, if $M$ is
finite, the right Cayley graph of the monoid $M$ defined by the presentation
$\langle A |R \rangle$; see~\cref{section-non-abstract} for more details.

A congruence enumeration \textbf{\textit{terminates}} if the output
$(\Gamma_{i}, \kappa_{i})$ has the property that applying any of \textbf{TC1},
\textbf{TC2}, or \textbf{TC3} to $(\Gamma_{i}, \kappa_{i})$ results in no
changes to the output, i.e.\ $(\Gamma_{i + 1}, \kappa_{i + 1}) = (\Gamma_{i},
\kappa_{i})$. It is straightforward to verify that a congruence enumeration
terminates at step $i$ if and only if $\Gamma_{i}$ is complete, compatible with
$R \cup S$, and deterministic.


For any given finite monoid presentation, there is a wide range of choices
for the order in which steps \textbf{TC1}, \textbf{TC2}, and \textbf{TC3}
are performed, and to which nodes, generators, and relations they are applied.
We will examine two specific strategies for enumerating congruences for an
arbitrary finite monoid presentation in more detail in \cref{section-hlt} and
\cref{section-felsch}.


\section{Validity}\label{section-validity}

In this section we address the validity of congruence enumeration as
defined in \cref{section-enumeration}; we will continue to use the
notation established therein.

The main results in this section are the following.

\begin{restatable}{thm}{validity}\label{validity}
  Let $A$ be a finite alphabet, let $R\subseteq A ^ * \times A ^ *$ be a finite set,
  and let $R ^ {\#}$ be the least two-sided congruence on $A ^ *$ containing
  $R$. If $S\subseteq A ^ * \times A ^ *$ is any finite set, and $\rho$ is
  the least right congruence on $A ^ *$ containing $R ^ {\#}$ and $S$, then the
  following hold:
  \begin{enumerate}
    \item[\rm (a)]
      If a congruence enumeration for $\rho$ terminates with output word graph
      $\Gamma=(N, E)$, then $A ^ * / \rho$ is finite and the function $\phi: N
      \times A ^ * \to N$, defined by $(\alpha, w)\phi = \beta$ whenever $w$
      labels an $(\alpha, \beta)$-path in $\Gamma$, is a right action that is
      isomorphic to the natural action of $A ^ *$ on $A ^ * / \rho$ by right
      multiplication.

    \item[\rm (b)]
      If $A ^ * / \rho$ is finite, then any congruence enumeration for $\rho$
      terminates.
  \end{enumerate}
\end{restatable}

\begin{restatable}{cor}{corollaryvalidity}
  \label{cor-tc-terminates}
  Let $A$ be a finite alphabet, let $R\subseteq A ^ * \times A ^ *$ be a finite
  set, and let $R ^ {\#}$ be the least two-sided congruence on $A ^ *$
  containing $R$. Then the following hold:
  \begin{enumerate}
    \item[\rm (a)]
      If a congruence enumeration for $R ^ {\#}$ terminates with output word
      graph $\Gamma=(N, E)$, then $A ^ * / R ^ {\#}$ is finite and the function
      $\phi : N \times A ^ * \to N$ defined by $(\alpha, w)\phi = \beta$
      whenever $w$ labels an $(\alpha, \beta)$-path in $\Gamma$ is a right
      action that is isomorphic to the (faithful) natural action of $A ^ * / R
      ^ {\#}$ on itself by right multiplication.

    \item[\rm (b)]
      If $A ^ * / R ^ {\#}$ is finite, then any congruence enumeration for $R ^
      {\#}$ terminates.
  \end{enumerate}
\end{restatable}

We will prove \cref{validity} and \cref{cor-tc-terminates} in
\cref{subsection-validity-proof}. We start by showing that the steps
\textbf{TC1}, \textbf{TC2}, and \textbf{TC3}  preserve certain properties of
word graphs in  \cref{prop-tc-invariant}. In
\cref{subsection-validity-termination}, we show that every congruence
enumeration eventually stabilises and is eventually compatible with $R$; and in
\cref{subsection-validity-proof} we give the proofs of \cref{validity} and
\cref{cor-tc-terminates}.


We will make repeated use, without reference, to the following straightforward lemma.

\begin{lemma}\label{lemma-connected}
  If any of \textbf{TC1}, \textbf{TC2}, or \textbf{TC3} is applied to
  $(\Gamma_{i}, \kappa_{i})$ where  every node in $\Gamma_i$ is reachable from
  $0$, then every node in the output $\Gamma_{i + 1}$ satisfies is reachable
  from $0$ also.
\end{lemma}

The next proposition also plays a crucial role in the proof of \cref{validity}
and \cref{cor-tc-terminates}.

\begin{prop}\label{prop-tc-invariant}
  If \textbf{TC1}, \textbf{TC2}, or \textbf{TC3} is applied to $(\Gamma_{i},
  \kappa_{i})$ where  $\Gamma_{i} = (N_i, E_i)$ is a word graph such that
  $(0)\pi_{\Gamma_i}\subseteq \rho$, and $(0)\pi_{\Gamma_i / \kappa_i}\subseteq
  \rho$, then the output $(\Gamma_{i + 1}, \kappa_{i + 1})$ satisfies
  $(0)\pi_{\Gamma_i} \subseteq (0)\pi_{\Gamma_{i + 1}} \subseteq \rho$, and
  $(0)\pi_{\Gamma_{i}/ \kappa_{i}} \subseteq (0)\pi_{\Gamma_{i + 1}/ \kappa_{i
  + 1}}\subseteq \rho$.
\end{prop}

The proof of \cref{prop-tc-invariant} is split into two parts due to
commonalities in the proofs rather than the statements. The two cases are:
\textbf{TC1}, \textbf{TC2}(a), \textbf{TC2}(b); and \textbf{TC2}(c),
\textbf{TC3}.

\begin{lemma}
  \cref{prop-tc-invariant} holds when \textbf{TC1}, \textbf{TC2}(a), or
  \textbf{TC2}(b) is applied.
\end{lemma}
\begin{proof}
  Since (a) and (b) of \textbf{TC2} are symmetric, it suffices to prove the
  proposition when \textbf{TC1} or \textbf{TC2}(a) is applied. In \textbf{TC1},
  $\alpha$ is a node in $\Gamma_{i}$ and there is no edge incident to $\alpha$
  labelled by $a$. In this case, $\Gamma_{i + 1}$ is obtained from $\Gamma_{i}$
  by adding the single node $\beta := 1 + \min \bigcup_{j \leq i} N_j$ and the
  single edge $(\alpha, a, \beta)$, and $\kappa_{i + 1} := \kappa_i \cup
  \{(\beta, \beta)\}$. In \textbf{TC2}(a), there exists $(u, v)\in R$, $b\in
  A$, and $v_1\in A ^ *$ such that $v = v_1b$, and there exist nodes $\alpha$,
  $\beta$, and $\gamma$  in $\Gamma_i$ such that $u$ and $v_1$ label $(\alpha,
  \beta)$- and $(\alpha, \gamma)$-paths in $\Gamma_i$, respectively. In this
  case, $\Gamma_{i + 1}$ is obtained from $\Gamma_{i}$ by adding the single
  edge $(\gamma, b, \beta)$ and $\kappa_{i + 1} := \kappa_i$.

  Hence if either \textbf{TC1} or \textbf{TC2}(a) is applied, then the identity
  map $N_i \to N_{i + 1}$ is a homomorphism from $\Gamma_i$ to $\Gamma_{i + 1}$
  and so the identity map $N_i/\kappa_i\to N_{i + 1} / \kappa_i$ is also a
  homomorphism from $\Gamma_i/\kappa_i$ to $\Gamma_{i+ 1}/\kappa_{i + 1}$.
  Hence, by \cref{lemma-path-equiv}, $(0)\pi_{\Gamma_i}\subseteq
  (0)\pi_{\Gamma_{i + 1}}$ and $(0)\pi_{\Gamma_i/\kappa_i}\subseteq
  (0)\pi_{\Gamma_{i + 1}/\kappa_{i + 1}}$.

  To show that $(0)\phi_{\Gamma_{i + 1}}\subseteq \rho$ we consider the cases
  when \textbf{TC1} and \textbf{TC2(a)} are applied separately.
  \medskip

  \noindent\textbf{TC1:}
  Suppose that $(u, v) \in (0)\pi_{\Gamma_{i + 1}}\setminus (0)\pi_{\Gamma_i}$
  is arbitrary. Then there are paths in $\Gamma_{i + 1}$ from $0$ to some node
  $\gamma$ labelled by $u$ and $v$ but there are no such paths in $\Gamma_i$.
  Since $\beta$ is not a node in $\Gamma_i$, $\beta$ is not the target of any
  path in $\Gamma_i$. It follows that the edge $(\alpha, a, \beta)$ must occur
  at least once in both paths. But $\beta$ is the source of no edges in
  $\Gamma_{i + 1}$, and so $(\alpha, a, \beta)$ occurs once, and it must be the
  last edge, in both paths. Hence $u = u_1a$ and $v= v_1a$ where $u_1, v_1\in A
  ^ *$ label $(0, \alpha)$-paths in $\Gamma_{i}$. Therefore $(u_1, v_1) \in
  \rho$ by assumption and since $\rho$ is a right congruence, $(u, v) = (u_1a,
  v_1a) \in \rho$ as well. Hence $(0)\pi_{\Gamma_{i + 1}} \subseteq \rho$.
  \medskip

  \noindent\textbf{TC2(a):}
  We proceed by induction on the total number $k$ of occurrences of the edge
  $(\gamma, b, \beta)$ (defined at the start of the proof) in any pair of paths
  in $\Gamma_{i + 1}$ with source $0$ and the same target node. The inductive
  hypothesis is: if $(x, y)\in (0)\pi_{\Gamma_{i + 1}}$, $X$ and $Y$ are paths
  with source $0$ labelled by $x$ and $y$, respectively, and the total number
  of occurrences of the edge $(\gamma, b, \beta)$ in $X$ and $Y$ is strictly
  less than $k$, then $(x, y) \in \rho$. The base case is when $k = 0$ and, in
  this case,  $(x, y) \in (0)\pi_{\Gamma_{i}}\subseteq \rho$, as required.

  Suppose that $k \geq 1$, that $x, y\in A ^ *$ are any words labelling $(0,
  \delta)$-paths $X$ and $Y$, respectively, for some node $\delta$, and that
  $(\gamma, b, \beta)$ occurs $k$ times in $X$ and $Y$. Without loss of
  generality we can assume that $(\gamma, b, \beta)$ occurs at least once in
  $X$. If $X$ consists of the edges $e_1, \ldots, e_r\in E_{i + 1}$, and $j\in
  \{1, \ldots, r\}$ is the largest value such that $e_j = (\gamma, b, \beta)$,
  then we may write $x = x_1bx_2$ where $x_1$ labels the $(0,\gamma)$-path
  $(e_1, \ldots, e_{j - 1})$ and $x_2$ labels the $(\beta, \delta)$-path $(e_{j
  + 1}, \ldots, e_r)$. If $w'$ labels any path from $0$ to $\alpha$ in
  $\Gamma_i$, then $w'v_1$  labels an $(0,\gamma)$-path in $\Gamma_{i}$. It
  follows that $(x_1, w'v_1) \in (0)\pi_{\Gamma_{i + 1}}$. The number of
  occurrences of $(\gamma, b, \beta)$ in $(e_1, \ldots, e_{j -1})$ is strictly
  less than $k$, and the number in the path (in $\Gamma_{i}$) labelled by
  $w'v_1$ is $0$. Hence, by the inductive hypothesis, $(x_1, w'v_1) \in \rho$
  and so $(x, w'vx_2) = (x_1bx_2, wv_1bx_2) \in \rho$.

  The word $w'$ labels a $(0, \alpha)$-path, $u$ labels a $(\alpha,
  \beta)$-path, and $x_2$ labels a $(\beta, \delta)$-path. The first two of
  these paths belong to $\Gamma_i$ by assumption, and the third does also, by
  the maximality of $j$. Hence $w'ux_2$ labels a $(0, \delta)$-path in
  $\Gamma_i$, which therefore contains no occurrences of the edge $(\gamma, b,
  \beta)$. The word $y$ also labels the  $(0, \delta)$-path $Y$ in $\Gamma_{i +
  1}$ and this path contains at most $k - 1$ occurrences of $(\gamma, b,
  \beta)$, since $X$ and $Y$ together contained $k$ occurrences, and $X$
  contained at least 1 occurrence of $(\gamma, b, \beta)$. Therefore, by
  induction, $(w'ux_2, y) \in \rho$. Finally, since $(u, v)\in R\subseteq
  \rho$, $(w'ux_2, w'vx_2)\in R ^ {\#} \subseteq \rho$ also. Hence $(x,
  w'vx_2), (w'vx_2, w'ux_2), (w'ux_2, y) \in \rho$ and so by transitivity, $(x,
  y) \in \rho$. \medskip

  We conclude the proof by showing that $(0)\pi_{\Gamma_{i + 1}/ \kappa_{i +
  1}} \subseteq \rho$ when either \textbf{TC1} or \textbf{TC2}(a) is applied.
  Suppose that $(u, v) \in (0)\pi_{\Gamma_{i + 1}/ \kappa_{i + 1}}$. If $(u, v)
  \in (0)\pi_{\Gamma_{i + 1}}$, then $(u, v)\in \rho$, as required. If $(u,
  v)\not\in (0)\pi_{\Gamma_{i + 1}}$, then $u$ and $v$ label $(0, \zeta)$- and
  $(0, \eta)$-paths in $\Gamma_{i + 1}$, respectively, for some nodes
  $\zeta\neq \eta$.  Recall that $ \kappa_{i + 1} = \kappa_{i}\cup \{(\beta,
  \beta)\}$ if \textbf{TC1} is applied, and $\kappa_{i + 1} = \kappa_i$ if
  \textbf{TC2}(a) is applied. In either case, $(\zeta, \eta)\in\kappa_{i + 1}$
  and $\zeta\neq \eta$ implies that $(\zeta, \eta) \in \kappa_i$. But $\zeta$
  and $\eta$ are nodes in $\Gamma_i$ and hence they are reachable from $0$ in
  $\Gamma_i$ by \cref{lemma-connected}. In other words, there exist $u', v'\in
  A ^ * $ labelling $(0, \zeta)$- and $(0, \eta)$-paths in $\Gamma_i$,
  respectively. Hence $(u, u'), (v', v)\in (0)\pi_{\Gamma_{i + 1}} \subseteq
  \rho$ and, by construction, $(u', v')\in (0)\pi_{\Gamma_i /
  \kappa_i}\subseteq \rho$. Thus, by transitivity, $(u, v)\in \rho$, as
  required.
\end{proof}

\begin{lemma}
  \cref{prop-tc-invariant} holds when \textbf{TC2}(c) or \textbf{TC3} is applied.
\end{lemma}
\begin{proof}
  If \textbf{TC2}(c) is applied, then there exist a node $\alpha$ in $\Gamma_i$
  and $(u, v) \in R$ such that $u$ and $v$ label $(\alpha, \beta)$- and
  $(\alpha, \gamma)$-paths in $\Gamma_i$ for some distinct nodes $\beta$ and
  $\gamma$ in $\Gamma_i$. In this case, $\Gamma_{i + 1} = \Gamma_i$ and
  $\kappa_{i + 1}$ is the least equivalence relation containing $\kappa_i$ and
  $(\beta, \gamma)$. If \textbf{TC3} is applied, then $\Gamma_{i + 1} =
  \Gamma_i / \kappa_i$  and $\kappa_{i + 1}$ is the least equivalence on $N_{i
  + 1}$ containing every $(\delta, \zeta)$, $\delta \neq \zeta$, for which
  there exist $\eta \in N_{i + 1}$ and $a\in A$ such that $(\eta, a, \delta),
  (\eta, a, \zeta)\in E_{i + 1}$.

  In either case, there is a homomorphism from $\Gamma_{i + 1}$ to $\Gamma_{i +
  1} / \kappa_{i + 1}$ and so \cref{lemma-path-equiv} implies that
  $(0)\pi_{\Gamma_{i + 1}}\subseteq (0)\pi_{\Gamma_{i + 1}/ \kappa_{i + 1}}$.
  If \textbf{TC2}(c) is applied, then the identity map is a homomorphism from
  $\Gamma_i$ to $\Gamma_{i + 1}$. If \textbf{TC3} is applied, then $\Gamma_{i +
  1} = \Gamma_{i}/ \kappa_i$ is a homomorphic image of $\Gamma_i$ also. Hence,
  in either case, $(0)\pi_{\Gamma_i}\subseteq (0)\pi_{\Gamma_{i + 1}}$ by
  \cref{lemma-path-equiv}.

  Hence, in both cases, it suffices to show that $(0)\pi_{\Gamma_{i}/
  \kappa_{i}} \subseteq (0)\pi_{\Gamma_{i + 1}/ \kappa_{i + 1}}\subseteq \rho$.
  We consider the cases when \textbf{TC2}(c) and \textbf{TC3} are applied
  separately.
  \medskip

  \noindent\textbf{TC2(c):}
  If $\beta$ and $\gamma$ are the nodes defined above, and $\sigma$ is the
  least equivalence relation on the set of nodes in $\Gamma_i/\kappa_i$
  containing $\beta/\kappa_i$ and $\gamma/ \kappa_i$, then
  $(\Gamma_i/\kappa_i)/\sigma = \Gamma_{i} / \kappa_{i + 1} = \Gamma_{i + 1}/
  \kappa_{i + 1}$. Hence $\Gamma_{i + 1}/ \kappa_{i + 1}$ is a homomorphic
  image of $\Gamma_i/\kappa_i$ and so  $(0)\pi_{\Gamma_i/ \kappa_i} \subseteq
  (0)\pi_{\Gamma_{i + 1}/ \kappa_{i + 1}}$.

  Suppose that $(x, y) \in (0)\pi_{\Gamma_{i + 1}/ \kappa_{i + 1}}$. Then $x$
  and $y$ both label $(0, \delta)$-paths in $\Gamma_{i + 1}/ \kappa_{i + 1}$
  for some node $\delta$ in $\Gamma_{i + 1}/ \kappa_{i + 1}$.  It follows that
  there are nodes $\zeta$ and $\eta$ in $\Gamma_{i + 1}$ such that $(\delta,
  \zeta), (\delta, \eta)\in \kappa_{i + 1}$ and $x$ and $y$ label $(0, \zeta)$-
  and $(0, \eta)$-paths in $\Gamma_{i + 1} = \Gamma_i$, respectively.

  If $(\zeta, \eta) \in \kappa_i$, then $(u, v)\in (0)\pi_{\Gamma_i /
  \kappa_i}$ and so $(u, v) \in \rho$ by assumption. If $(\zeta, \eta)\not\in
  \kappa_i$, then without loss of generality we may assume that $(\zeta, \eta)
  = (\beta, \gamma)$. In this case, if $w\in A ^ *$ labels a $(0, \alpha)$ path
  in $\Gamma_i$, then $wu$ and $x$ both label $(0, \beta) = (0, \zeta)$-paths
  in $\Gamma_i$, and so $(wu, x)\in (0)\pi_{\Gamma_i}\subseteq \rho$.
  Similarly, $wv$ and $y$ both label $(0, \gamma)$-paths in $\Gamma_i$, and so
  $(wv, y) \in \rho$, also. Finally, $(u, v) \in R$ and so $(wu, wv)\in R ^
  {\#} \subseteq \rho$, and so $(x, y) \in \rho$ also.
  \medskip

  \noindent\textbf{TC3:}
  Clearly, $\Gamma_{i + 1}/ \kappa_{i + 1}$ is a homomorphic image of
  $\Gamma_{i + 1}$ and $\Gamma_{i + 1}$ is defined to be $\Gamma_i/ \kappa_i$
  in \textbf{TC3}. Hence $(0)\pi_{\Gamma_i/ \kappa_i} \subseteq
  (0)\pi_{\Gamma_{i + 1}/ \kappa_{i + 1}}$ by \cref{lemma-path-equiv}. It
  remains to show that $(0)\pi_{\Gamma_{i + 1}/ \kappa_{i + 1}}\subseteq \rho$.

  Suppose that $(x, y) \in (0)\pi_{\Gamma_{i + 1}/ \kappa_{i + 1}}$. Then there
  exist nodes $\zeta$ and $\eta$ in $\Gamma_{i + 1}$ such that $(\zeta,
  \eta)\in \kappa_{i + 1}$ where $x$ labels a $(0, \zeta)$-path and $y$ labels
  a $(0, \eta)$-paths in $\Gamma_{i + 1}$. In \textbf{TC3}, $\kappa_{i + 1}$ is
  defined to be the least equivalence containing $(\beta, \gamma) \in N_{i + 1}
  \times N_{i + 1}$ such that there are edges $(\alpha, a, \beta)$ and
  $(\alpha, a, \gamma)$ in $\Gamma_{i + 1}$. Hence $(\eta, \zeta) \in \kappa_{j
  + 1}$ implies that there exist $\alpha_1, \ldots, \alpha_{n - 1}, \beta_1 :=
  \zeta, \beta_1, \ldots, \beta_n := \eta \in N_{j + 1}$ and $a_1, \ldots, a_{n
  - 1}\in A$ where $(\alpha_j, a_j, \beta_j), (\alpha_j, a_j, \beta_{j + 1})
  \in E_{j + 1}$ and $\beta_j \neq \beta_{j + 1}$ for every $j$. Suppose that
  $w_j\in A ^ *$ is any word labelling a $(0, \alpha_j)$-path for $j = 1,
  \ldots, n - 1$, that $x_1 = x$, that $x_j$ labels a $(0, \beta_j)$-path for
  $j = 2, \ldots, n - 1$, and that $x_n = y$.
  Then $w_ja_j$ and $x_j$ both label $(0, \beta_j)$-paths in $\Gamma_{i + 1}$
  for every $j$. Similarly, $w_ja_j$ and $x_{j + 1}$ both label $(0, \beta_{j +
  1})$-paths in $\Gamma_{i + 1}$ for every $j$. Hence  $(w_ja_j, x_j), (w_ia_j,
  x_{j + 1}) \in (0)\pi_{\Gamma_{i + 1}} \subseteq \rho$ for every $j$.
  Therefore $(x_j, x_{j + 1}) \in \rho$ for every $j$, and so by transitivity
  $(x_1, x_n) = (x, y)\in \rho$, as required.
\end{proof}


\subsection{Completeness, determinism, and compatibility}\label{subsection-validity-termination}

In this section, we will prove that if at some step $i$ in a congruence
enumeration the word graph $\Gamma_{i}$ is complete, then that enumeration
terminates, and that every congruence enumeration is eventually compatible with
$\rho$.


We say that a congruence enumeration \defn{stabilises} if for every $i\in \N$
and every node $\alpha$ of $\Gamma_{i}$ there exists $K\in \N$ such that for
all $j\geq K$ either $\alpha$ is not a node in $\Gamma_{j}$ or if $(\alpha, a,
\beta) \in E_j$, then $(\alpha, a, \beta) \in E_{j + 1}$ for any $a\in A$.


\begin{lemma}\label{lemma-stabilise}
  Every congruence enumeration stabilises.
\end{lemma}
\begin{proof}
  We will prove a slightly stronger statement, that any sequence of
  \textbf{TC1}, \textbf{TC2}, and \textbf{TC3} stabilises.

  Suppose that $i\in \N$, $\alpha$ is an arbitrary node in $\Gamma_{i}$, and
  $a\in A$ is arbitrary.  Either there exists $K > i$ such that $\alpha$ is no
  longer a node in $\Gamma_{K}$, or $\alpha$ is a node in $\Gamma_{j}$ for all
  $j \geq i$. If $\alpha$ is no longer a node in $\Gamma_{K}$, then, from the
  definitions of \textbf{TC1}, \textbf{TC2}, and \textbf{TC3}, $\alpha$ is not
  a node in $\Gamma_{k}$ for any $k\geq K$.

  Suppose that $\alpha$ is a node in $\Gamma_{j}$ for all $j \geq i$ and that
  $(\alpha, a, \beta_1) \in E_j$ for some $\beta_1\in A ^ *$. If $\beta_1\in
  N_j$ is replaced by $\beta_2\in N_k$ at some step $k\geq j$, then step $k$ is
  an application of \textbf{TC3}. In particular, $0 \leq \beta_2 \leq \beta_1$.
  This process can be repeated only finitely many times because there are only
  finitely many natural numbers less than $\beta_1$.
\end{proof}


If $\Gamma_{i}$ is complete for some $i$, then \textbf{TC1} cannot be applied
again at any step after $i$. In this case, \textbf{TC2} and \textbf{TC3} make
$\kappa$ coarser and hence reduce the number of nodes in $\Gamma_{i}$. Since
the number of nodes in $\Gamma_i$ is finite, it follows that the enumeration
must terminate at some point. We record this in the following corollary.

\begin{cor}\label{cor-terminates-if-complete}
  If $\Gamma_{i}$ is complete at some step $i$ of a congruence
  enumeration, then the enumeration terminates at some step $j\geq i$.
\end{cor}

The next lemma shows that, roughly speaking, if the word graph at some step of
a congruence enumeration is non-deterministic, then at some later step it is
deterministic.

\begin{lemma}\label{lemma-determinism}
  If $w\in A ^ *$ labels a path in $\Gamma_i$ starting at $0$ for some $i\in
  \N$, then there exists $j\geq i$ such that $w$ labels a unique path in
  $\Gamma_j$ starting at $0$.
\end{lemma}
\begin{proof}
  We proceed by induction on the length of the word $w$. If $w = \varepsilon$,
  then $w$ labels a unique path in every $\Gamma_i$ starting (and ending) at
  $0$.

  Suppose that every word of length at most $n - 1$ labelling a path starting
  at $0$ in $\Gamma_i$ labels a unique path, and let $w\in A ^ *$ be any word
  of length $n\geq 1$, labelling a path in $\Gamma_i$ starting at $0$. If $w =
  w_1a$ for some $w_1\in A ^ *$ and $a \in A$, then $w_1$ labels a unique path
  in $\Gamma_i$ starting at $0$ by induction. Suppose that $\alpha$ is the
  target of the unique path labelled $w_1$. By \cref{lemma-stabilise}, we may
  assume without loss of generality that $\alpha$ is a node in $\Gamma_j$ for
  every $j\geq i$. If there exist edges $(\alpha, a, \beta_1), \ldots, (\alpha,
  a, \beta_r)$ in $\Gamma_i$, then at most one of these edges was created by an
  application of \textbf{TC1} or \textbf{TC2}. So, if $r > 1$, then $(\beta_1,
  \beta_2), \ldots, (\beta_{r - 1}, \beta_r) \in \kappa_i$.

  By part (e) of the definition of a congruence enumeration
  (\cref{de-2-sided-enumeration-process}), there exists $j > i$ such that
  $\kappa_j = \Delta_{N_j}$ and so at step $j$, there is only one edge incident
  to $\alpha$ labelled $a$.
\end{proof}

Next we show that every congruence enumeration is eventually compatible with
$\rho$.

\begin{lemma}\label{lemma-injective}
  If $(u, v) \in R ^ {\#}$, then there exists $K\in \N$ such that  $(u, v) \in
  (0)\pi_{\Gamma_i}$ for all $i \geq K$.
\end{lemma}
\begin{proof}
  If $(u, v)\in R ^ {\#}$, then there exists an elementary sequence $u = w_1,
  w_2, \ldots, w_k = v$ in $A ^ *$ with respect to $R$. If $(w_j, w_{j + 1})
  \in (0)\pi_{\Gamma_i}$  for every $j \in \{1, \ldots, k - 1\}$, by
  transitivity $(w_1, w_k) = (u, v) \in (0)\pi_{\Gamma_i}$ for a large enough
  $i$. Hence it suffices to show that $(w_j, w_{j + 1}) \in (0)\pi_{\Gamma_i}$
  for all $j$ and for sufficiently large $i$.

  Suppose that $j\in \{1, \ldots, k - 1\}$ is arbitrary, and that $w_j = pxq$
  and $w_{j + 1} = pyq$ where $p, q\in A ^ *$ and $(x, y) \in R$. From part (c)
  of the definition of a congruence enumeration
  (\cref{de-2-sided-enumeration-process}), $p$ labels a path with source $0$ in
  $\Gamma_{j}$ for all $j \geq i$ for some sufficiently large $i$. We may
  choose $i$ large enough so that the target node $\alpha$ of this path is a
  node in every $\Gamma_j$ for $j\geq i$. By
  \cref{de-2-sided-enumeration-process}(d), for sufficiently large $i$,  $(x,
  y) \in (\alpha)\pi_{\Gamma_i}$ and so $(px, py) \in (0)\pi_{\Gamma_i}$. Again
  if $i$ is large enough, there is a path labelled $q$ from the target of the
  path from $0$ labelled by $px$ (or $py$) and so $(w_j, w_{j + 1}) = (pxq,
  pyq)\in (0)\pi_{\Gamma_i}$.
\end{proof}

\begin{cor}\label{cor-injective}
  If $(u, v) \in\rho$, then there exists $K\in \N$ such that  $(u, v) \in
  (0)\pi_{\Gamma_i}$ for all $i \geq K$.
\end{cor}
\begin{proof}
  Recall that $\rho$ is the least right congruence containing $R ^ {\#}$ and
  $S$ (as defined in \cref{validity}). If $(u, v)\in \rho$, then there exists a
  right elementary sequence $u = w_1, w_2, \ldots, w_k =v$ such that $w_j = xq$
  and $w_{j + 1} = yq$ for some $(x,y) \in R ^ {\#} \cup S$ and some $q\in A ^
  *$. Hence it suffices to show that if $(x, y) \in R ^ {\#} \cup S$ and $q\in
  A ^ *$, then, for sufficiently large $i$,  $(xq, yq)\in (0)\pi_{\Gamma_i}$.

  If $(\Gamma_{m}, \kappa_{m})$ is the output of applying steps (a) and (b) of
  the congruence enumeration (\cref{de-2-sided-enumeration-process}), then
  $S\subseteq \pi_{\Gamma_m / \kappa_m}(0)$. \cref{prop-tc-invariant} implies
  that $\pi_{\Gamma_m/ \kappa_m}(0) \subseteq \pi_{\Gamma_i/ \kappa_i}(0)$ for
  all $i\geq m$. By \cref{de-2-sided-enumeration-process}(e), \textbf{TC3} is
  eventually applied to $(\Gamma_{i - 1}, \kappa_{i - 1})$ for some $i > m$.
  Hence $S \subseteq (0)\pi_{\Gamma_{i - 1}/ \kappa_{i - 1}} =
  (0)\pi_{\Gamma_{i}}$ since $\Gamma_{i} = \Gamma_{i - 1}/ \kappa_{i - 1}$, for
  large enough $i$.

  Suppose that $q\in A ^ *$ and $(x, y) \in R ^ {\#} \cup S$ are arbitrary. If
  $(x, y) \in R ^ {\#}$ then $(x, y) \in (0)\pi_{\Gamma_i}$ for large enough
  $i$ by \cref{lemma-injective}. If $(x, y) \in S$, then it follows from the
  previous paragraph that $(x, y) \in (0)\pi_{\Gamma_i}$ for large enough $i$.
  As in the proof of \cref{lemma-injective}, for sufficiently large $i$, $xq$
  and $yq$ label paths in $\Gamma_{i}$, and so $(xq, yq) \in
  (0)\pi_{\Gamma_i}$.
\end{proof}


\begin{cor}\label{cor-rho-tau-again}
  Suppose that $u, v \in A ^ *$ are arbitrary. Then $(u, v)\in \rho$ if and
  only if there exists $K\in \N$ such that $(u,v) \in (0)\pi_{\Gamma_i}$ for
  all $i\geq K$.
\end{cor}
\begin{proof}
  ($\Rightarrow$) Assume that $(u, v)\in \rho$. By \cref{lemma-injective}, $(u,
  v) \in (0)\pi_{\Gamma_K}$ for large enough $K$. But $(0)\pi_{\Gamma_i}
  \subseteq (0)\pi_{\Gamma_{i + 1}}$ for all $i$, and so $(u, v) \in
  (0)\pi_{\Gamma_i}$ for all $i\geq K$.
  \medskip

  ($\Leftarrow$) By assumption, $\Gamma_{1}$ is a word graph with
  $(0)\pi_{\Gamma_1}\subseteq \rho$. By \cref{prop-tc-invariant},
  $(0)\pi_{\Gamma_i} \subseteq \rho$ for all $i\in \N$. Hence if $(u, v)\in
  (0)\pi_{\Gamma_i}$ for some $i$, then $(u, v)\in \rho$.
\end{proof}

We also obtain the following stronger result for an enumeration of a two-sided
congruence.

\begin{cor}\label{cor-rho-tau}
  Suppose that $u, v\in A ^ *$ are arbitrary. Then  the following are equivalent:
  \begin{enumerate}
    \item [\rm (i)]
          $(u, v)\in R ^ {\#}$;
    \item [\rm (ii)]
          there exists $K\in \N$ such that $(u, v) \in (\alpha)\pi_{\Gamma_i}$
          for all nodes $\alpha$ in $\Gamma_{i}$ and for all $i\geq K$;
    \item [\rm (iii)]
          there exists $K\in \N$ such that $(u, v) \in (0)\pi_{\Gamma_i}$ for
          all $i\geq K$.
  \end{enumerate}
\end{cor}
\begin{proof}
  \textbf{(i) $\Rightarrow$ (ii).}
  If $(u, v)\in R ^ {\#}$ and $w\in A ^ *$ is arbitrary, then $(wu, wv) \in R ^
  {\#}$ and so, by \cref{lemma-injective}, $(wu, wv)\in (0)\pi_{\Gamma_i}$ for
  sufficiently large $i$. In particular, for large enough $i$, $wu$ and $wv$
  both label $(0, \beta)$-paths for some node $\beta$. If $\alpha$ is the
  target of the path starting at $0$ labelled by $w$, then $u$ and $v$ both
  label a $(\alpha, \beta)$-paths in $\Gamma_{i}$ and so $(u, v) \in
  (\alpha)\pi_{\Gamma_{i}}$ for $i$ large enough.
  \medskip

  \noindent \textbf{(ii) $\Rightarrow$ (iii).}
  This follows immediately since $0$ is a node in every word graph.
  \medskip

  \noindent \textbf{(iii) $\Rightarrow$ (i).}
  Follows by applying \cref{cor-rho-tau-again} to $\rho = R ^ {\#}$.
\end{proof}


\subsection{The proofs of \cref{validity} and \cref{cor-tc-terminates}}\label{subsection-validity-proof}

In this section, we prove \cref{validity} and \cref{cor-tc-terminates}.

\validity*
\begin{proof}
  We start the proof by showing that there is always a bijection between a
  certain set of nodes in the word graphs $\Gamma_i$ (defined in a moment) and
  the classes of $\rho$. Let $M$ be the subset of $\bigcup_{i \in \N} N_i$ so
  that $\alpha\in M$ if $\alpha\in N_i$ for all large enough $i$. By
  \cref{lemma-stabilise}, for every $\alpha \in M$ there exists $w_\alpha\in A
  ^ *$ such that $w_\alpha$ labels a path from $0$ to $\alpha$ in every
  $\Gamma_i$ for sufficiently large $i$. We define $f: M \to A ^ * / \rho $ by
  \begin{equation}\label{eq-action}
    (\alpha)f = w_\alpha/\rho.
  \end{equation}
  We will show that $f$ is a bijection.

  If $\alpha, \beta\in M$ are such that $(\alpha)f = (\beta)f$, then
  $(w_\alpha, w_\beta) \in \rho$ and so $(w_\alpha, w_\beta) \in
  (0)\pi_{\Gamma_i}$ for large enough $i$ by \cref{cor-rho-tau-again}. In
  particular, $w_\alpha$ and $w_\beta$ both label paths from $0$ to the same
  node in every $\Gamma_i$ for large enough $i$. By \cref{lemma-stabilise}, we
  can also choose $i$ large enough so that the target $\gamma$ of these paths
  is a node in $\Gamma_j$ for all $j\geq i$. But $w_\alpha$ and $w_\beta$ also
  label $(0, \alpha)$- and $(0, \beta)$-paths in every $\Gamma_i$ for
  sufficiently large $i$. By \cref{lemma-determinism}, there exists $j \geq i$
  such that $w_{\alpha}$ labels a unique path starting at $0$, and so $\alpha =
  \gamma$. Similarly, there exists $k\geq j$ such that $w_{\beta}$ labels a
  unique path from $0$, and so $\beta = \gamma$ also. In particular, $\alpha =
  \beta$ and so  $f$ is injective.

  For surjectivity, let $u\in A ^ *$ be arbitrary. Then $u$ labels a path in
  $\Gamma_{i}$ for large enough $i$. Since every congruence enumeration
  stabilises, we can choose $i$ large enough so that every edge in the path
  labelled by $u$ belongs to every $\Gamma_{j}$ for $j\geq i$. In particular,
  $u$ labels a $(0, \alpha)$-path in $\Gamma_{i}$ for some $\alpha\in M$. But
  $w_{\alpha}$  also labels a $(0, \alpha)$-path in $\Gamma_i$, and so $(u,
  w_{\alpha}) \in (0)\pi_{\Gamma_{i}}$. Hence, again by
  \cref{cor-rho-tau-again},  for large enough $i$, $(u, w_{\alpha}) \in \rho$
  and so $(\alpha)f = w_{\alpha} /\rho = u /\rho$, and $f$ is surjective.
  \medskip

  \textbf{(a).}
  Suppose that $(\Gamma, \kappa)$ is the output of the enumeration where
  $\Gamma=(N, E)$ and that $\pi:N \to \mathfrak{P}(A ^ * \times A ^ *)$ is the
  path relation of $\Gamma$. Since the congruence enumeration process has
  terminated,  $\Gamma$ is finite and its set of nodes $N$ coincides with the
  set $M$ defined at the start of the proof. Hence $A ^ * / \rho$ is finite,
  since $f:N \to A ^ * / \rho$ defined in \eqref{eq-action} is a bijection.

  Since any application of \textbf{TC3} results in no changes to $\Gamma$, it
  follows that $\kappa = \Delta_{N}$. Hence $\Gamma$ is deterministic, and so
  every $w\in A ^ *$ labels exactly one path starting at every $u \in N$. In
  particular, $\phi$ (as defined in the statement) is a well-defined function.
  If $w_1, w_2\in A ^ *$ label  $(\alpha_1,\alpha_2)$- and $(\alpha_2,
  \alpha_3)$-paths in $\Gamma$, respectively, then certainly $w_1w_2$ labels a
  $(\alpha_1, \alpha_3)$-path, and so $\phi$ is an action.

  We will show that $f: N \to A ^ * / \rho$ defined in \eqref{eq-action} is an
  isomorphism of the actions $\phi$ and the natural action $\Phi: A ^ * / \rho
  \times A ^ * \to  A ^ * / \rho$ defined by $(u / \rho, v)\Phi = uv /\rho$. We
  define $g: N \times A ^ * \to (A ^ * / \rho) \times A ^ *$ by $(\alpha, v)g =
  ((\alpha)f, v) = (w_{\alpha}/\rho, v)$ for all $(\alpha, v)\in N \times A ^
  *$. To show that $f$ is a homomorphism of actions, we must prove that the
  diagram in \cref{fig-validity-a} commutes.
  \begin{figure}
    \centering
    \begin{tikzcd}
      N\times A ^ *
      \arrow[d, "\phi"]
      \arrow[r, "g"]
      & A ^ * / \rho \times A ^ *
      \arrow[d, "\Phi"] \\
      N \arrow[r, "f"]
      & A ^ * / \rho
    \end{tikzcd}
    \caption{The commutative diagram in the proof of
      \cref{validity}(a).}
    \label{fig-validity-a}
  \end{figure}
  Suppose that $(\alpha, v) \in N\times A ^ *$. Then
  \[
    ((\alpha, v)g)\Phi = (w_{\alpha}/\rho, v)\Phi = w_{\alpha}v / \rho
  \]
  and if $v$ labels a $(\alpha, \beta)$-path in $\Gamma$,  then
  \[
    ((\alpha, v)\phi)f = (\beta)f = w_\beta / \rho.
  \]
  But $w_{\beta}$ and $w_{\alpha}v$ both label $(0, \beta)$-paths in $\Gamma$
  and so $(w_{\alpha}v, w_{\beta}) \in (0)\pi$. By \cref{cor-rho-tau-again},
  $(0)\pi \subseteq \rho$ and so $(w_{\alpha}v, w_{\beta})\in \rho$. Hence
  $((\alpha, v)g)\Phi = ((\alpha, v)\phi)f$ and so $f$ is homomorphism of the
  actions $\phi$ and $\Phi$.
  \bigskip

  \textbf{(b).}
  If $A ^ * / \rho$ is finite, then the set $M$ of nodes defined at the start
  of the proof is also finite.  Since every congruence enumeration stabilises,
  it follows that for every $\alpha\in M$, there exists $K\in \N$ such that
  $(\alpha, a, \beta)\in E_i$ for all $i\geq K$. It follows that $\beta \in M$.
  Since $A$ and $M$ are finite, there exists $K\in \N$  such that $(\alpha, a,
  \beta) \in E_i$ for all $\alpha\in M$, all $a\in A$, and all $i\geq K$. It
  follows that none of \textbf{TC1}, \textbf{TC2}, nor \textbf{TC3} results in
  any changes to $\Gamma_i$ or $\kappa_i$, $i\geq K$, and so the enumeration
  terminates.
\end{proof}

We conclude this section by proving \cref{cor-tc-terminates}.

\corollaryvalidity*
\begin{proof}
  If $S = \varnothing$ in \cref{validity}, then the least right congruence
  $\rho$ containing $S$ and $R ^ {\#}$ is just $R ^ {\#}$. By \cref{validity},
  the action $\phi: N\times A ^ *\to N$ is isomorphic to the natural action of
  $A ^ *$ on $A ^ * / R ^ {\#}$ by right multiplication. By \cref{cor-rho-tau},
  the kernel of $\phi$ is $R ^ {\#}$, and it is routine to verify that the
  kernel of the action of $A ^ *$ on $A ^ * / R ^ {\#}$ by right multiplication
  is also $R ^ {\#}$. Hence, by \cref{prop-quotient-action}, the action $\phi$
  of $A ^ * $ on $N$ and the induced action of $A ^ * / R ^ {\#}$ on $N$ are
  isomorphic. Similarly, the action of $A ^ *$ on $A ^ * / R ^ {\#}$ is
  isomorphic to the action of $A ^ * /R ^ {\#}$ on $A ^ * / R ^ {\#}$ by right
  multiplication. Thus the actions of  $A ^ * / R ^ {\#}$ on $N$ and $A ^ * /R
  ^ {\#}$ on $A ^ * / R ^ {\#}$ are isomorphic also. The latter action is
  faithful since $A ^ * / R ^ {\#}$ is a monoid. Hence the action of $A ^ * / R
  ^ {\#}$ on $N$ is faithful also, and there is nothing further to prove.
\end{proof}

\section{Monoids not defined by a presentation}\label{section-non-abstract}

Suppose that the monoid $M$ defined by the presentation $\langle A | R\rangle$
is finite and that $\phi: A ^* \to M$ is the unique surjective homomorphism
extending the inclusion of $A$ in $M$. Then $\ker(\phi) = R ^ {\#}$ and
$(a)\phi = a$ for every $a\in A$. If $M = \{m_1, m_2, \ldots, m_{|M|}\}$, then
the \textbf{\textit{right Cayley graph of $M$ with respect to $A$}} is the word
graph $\Gamma$ with nodes $N = \{0, \ldots, |M|-1\}$ and edges $(\alpha, a,
\beta)$ for all $\alpha, \beta \in N$ and every $a\in A$ such that $m_{\alpha}a
= m_{\beta}$. The right Cayley graph $\Gamma$ is complete and deterministic,
and so if $\pi : N \to \mathfrak{P}(A ^ * \times A ^ *)$ is the path relation
on $\Gamma$, then $(\alpha)\pi = R ^ {\#}$ for all $\alpha\in N$.


\begin{ex}\label{ex-cayley-digraph}
  Suppose that $M$ is the monoid generated by the matrices
  \[
    a =
    \begin{pmatrix}
      1 & 1 & 0 \\
      0 & 1 & 1 \\
      1 & 0 & 1
    \end{pmatrix},
    \quad
    b =
    \begin{pmatrix}
      1 & 1 & 0 \\
      0 & 1 & 0 \\
      0 & 0 & 1
    \end{pmatrix},
    \quad
    c =
    \begin{pmatrix}
      1 & 0 & 1 \\
      1 & 1 & 0 \\
      0 & 1 & 1
    \end{pmatrix}
  \]
  over the boolean semiring $\mathbb{B} = \{0, 1\}$ with addition defined by
  \begin{align*}
    0 + 1 = 1 + 0 = 1 + 1 = 1, \qquad
    0 + 0 = 0
  \end{align*}
  and multiplication defined as usual for the real numbers $0$ and $1$.  The
  right Cayley graph of $M$ with respect to its generating set $\{a, b, c\}$
  is shown in \cref{fig-ex-cayley-digraph}; see also \cref{table-ex-cayley-digraph}.

  \begin{figure}
    \centering
      \includegraphics{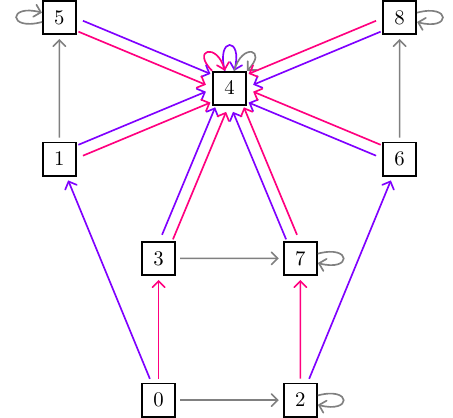}
    \caption{The right Cayley graph of the monoid $M$ from
      \cref{ex-cayley-digraph} with respect to the generating set $\{a, b,
      c\}$; see \cref{table-ex-cayley-digraph} for a representative word for
      each node. Purple arrows correspond to edges labelled $a$, gray labelled
      $b$, and pink labelled $c$.}
    \label{fig-ex-cayley-digraph}
  \end{figure}
  \begin{table}\centering
    \begin{tabular}{l|l|l|l|l|l|l|l|l}
      \renewcommand{\arraystretch}{1.2}
      0             & 1   & 2   & 3   & 4       & 5    & 6    & 7    & 8     \\ \hline
      $\varepsilon$ & $a$ & $b$ & $c$ & $a ^ 2$ & $ab$ & $ba$ & $bc$ & $bab$
    \end{tabular}
    \caption{A word labelling a path from $0$ to each node in the right Cayley
    graph on the monoid $M$ from \cref{ex-cayley-digraph}; see
  \cref{fig-ex-cayley-digraph}.}\label{table-ex-cayley-digraph}
  \end{table}
\end{ex}


Suppose that we want to determine a right congruence on a finite monoid $M$
generated by a non-abstract set of generators $A$ --- such as transformations,
partial permutations, or matrices over a semiring. Of course, every finite
monoid $M$ is finitely presented, and so congruence enumeration can be applied
to any presentation for $M$. It is possible that a presentation for $M$ is
already known, or we can compute a presentation for $M$; for example, by using
the Froidure-Pin Algorithm (see~\cite{Froidure1997aa}).

It is also possible, instead of starting with the trivial word graph, to start
a congruence enumeration with  the right Cayley graph $\Gamma_{1} = (N_1, E_1)$
of $M$ with respect to $A$ as the input to the process. As mentioned above,
$(0)\pi_{\Gamma_1} = R ^ {\#}$ and so $(\Gamma_1, \Delta_{N_1})$ is a valid
input to a congruence enumeration. The right Cayley graph $\Gamma_{1}$ is
already complete, deterministic, and compatible with $R ^ {\#}$ by definition.
It follows that \textbf{TC2} and \textbf{TC3} are the only steps actually
applied in a congruence enumeration with input $\Gamma_{1}$. More precisely,
\textbf{TC2}(c) is applied to $0$ and every $(u, v) \in S$, and then
\textbf{TC3} is repeatedly applied until $\kappa_{i} = \Delta_{N_i}$. The right
Cayley graph of a monoid $M = \genset{A}$ of one of the types mentioned above,
can also be computed using the Froidure-Pin Algorithm, which has complexity
$O(|A||M|)$. Given that the right Cayley graph has $|M|$ nodes and $|A||M|$
edges, it is unlikely that, in general, there is a better method for finding
the Cayley graph of a monoid from a generating set.

Experiments indicate that if the congruence being enumerated has a relatively
large number of classes in comparison to $|M|$, then this second approach is
faster than starting from the trivial word graph. On the other hand, if the
number of congruence classes is relatively small compared with $|M|$, then
starting from the trivial word graph is often faster. Since the number of
congruence classes is usually not known in advance, the implementation in
\libsemigroups runs both these approaches in parallel, accepting whichever
provides an answer first.

We end this section with an example of performing a congruence enumeration with
a right Cayley graph as an input.


\begin{ex}
  \label{ex-cayley-digraph-cong}
  Suppose that $M$ is the monoid from \cref{ex-cayley-digraph}. To compute the
  least right congruence $\rho$ on $M$ containing $S= \{(a, b)\}$ we perform
  the following steps.

  Set $\Gamma_{1}$ to be the right Cayley graph of $M$ with respect to the
  generating set $\{a, b, c\}$; see \cref{fig-ex-cayley-digraph} and
  \cref{table-ex-cayley-digraph}. Suppose that $a < b < c$.

  \begin{description}
    \item[Step 1:]
      Apply \textbf{TC2} to the only pair $(a, b)$ in $S$, and the output of
      this is $(\Gamma_{2}, \kappa_{2})$ where $\kappa_{2}$ is the least
      equivalence on the nodes of $\Gamma_2$ containing  $(1, 2)$; see
      \cref{fig-delta-1}.

    \item [Step 2:]
          Apply \textbf{TC3} to obtain the quotient of $\Gamma_{2}$ by
          $\kappa_{2}$, the output is $(\Gamma_{3}, \kappa_{3})$ where
          $\Gamma_{3} = \Gamma_{2} / \kappa_{2}$ and $\kappa_{3}$ is the least
          equivalence containing $(1, 5)$, $(4, 6)$, and $(4, 7)$; see
          \cref{fig-delta-2}.

    \item [Step 3:]
          Apply \textbf{TC3} with input $(\Gamma_{3}, \kappa_{3})$ and output
          $(\Gamma_{4}, \kappa_{4})$ where $\Gamma_{4} = \Gamma_{3} /
          \kappa_{3}$ and $\kappa_{4}$ is the least equivalence containing $(4,
          8)$; see \cref{fig-delta-3}.

    \item [Step 4:]
          Apply \textbf{TC3} with input $(\Gamma_{4}, \kappa_{4})$ and output
          $(\Gamma_{5}, \kappa_{5})$ where $\Gamma_{5} = \Gamma_{4} /
          \kappa_{4}$ and $\kappa_{5} = \Delta_{N_5}$. Since $\Gamma_{5}$ is
          complete and compatible with $\rho$ and $\kappa_{5}$ is trivial, the
          enumeration terminates with $\Gamma_{5}$; see \cref{fig-delta-4}.
  \end{description}

  We conclude that $\rho$ has $4$ equivalence classes whose representatives
  $\varepsilon$, $a$, $c$, and $a ^ 2$ correspond to   the nodes $\{0, 1, 3,
  4\}$ in the graph $\Gamma_{5}$.  By following the paths in $\Gamma_{5}$
  starting at $0$ labelled by each of the words in
  \cref{table-ex-cayley-digraph} we determine that the congruence classes of
  $\rho$ are: $\{\varepsilon\}$, $\{a, b, ab\}$, $\{c\}$, $\{a ^ 2, ba, bc,
  bab\}$.


  \begin{figure}
    \begin{subfigure}{0.49\textwidth}
      \centering
      \includegraphics{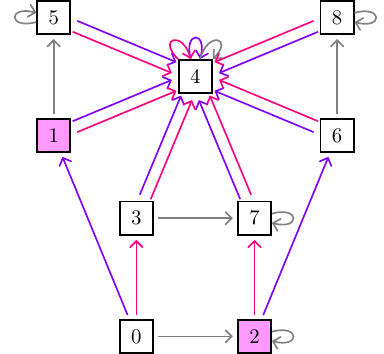}
      \caption{Step 1.}\label{fig-delta-1}
    \end{subfigure}
    \begin{subfigure}{0.49\textwidth}
      \centering
      \includegraphics{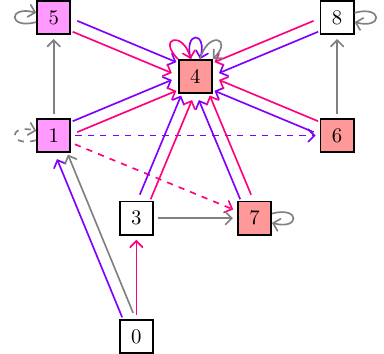}
      \caption{Step 2.}\label{fig-delta-2}
    \end{subfigure}
    \vspace{\baselineskip}

    \begin{subfigure}{0.49\textwidth}
      \centering
      \includegraphics{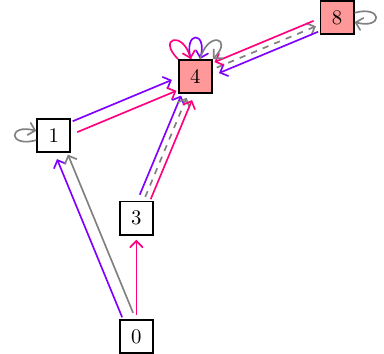}
      \caption{Step 3.}\label{fig-delta-3}
    \end{subfigure}
    \begin{subfigure}{0.49\textwidth}
      \centering
      \includegraphics{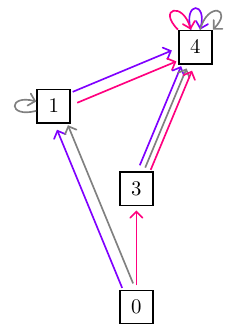}
      \caption{Step 4.}\label{fig-delta-4}
    \end{subfigure}
    \caption{The output $(\Gamma_{i}, \kappa_{i})$ of each step in
      \cref{ex-cayley-digraph-cong}.  Purple arrows correspond to $a$, gray to
      $b$, pink to $c$, shaded nodes of the same colour belong to $\kappa_{i}$,
      and unshaded nodes belong to singleton classes; see
      \cref{table-ex-cayley-digraph} for a representative word for each node. A
      dashed edge with a single arrowhead denotes an edge that is obtained from
      \textbf{TC2} or \textbf{TC3}, solid edges correspond to edges that
      existed at the previous step.}\label{fig-ex-cayley-digraph-cong}
  \end{figure}
\end{ex}

\section{The HLT strategy}
\label{section-hlt}

In this section, we describe the so-called \defn{HLT strategy} for congruence
enumeration. The acronym HLT stands for Haselgrove, Leech, and Trotter; for
further details see~\cite{Cannon1973aa} or~\cite[Section 5.2]{Sims1994aa}.  For
the record, this is Walker's Strategy 1 in~\cite{Walker1992aa}, and is referred
to as R-style in ACE~\cite{Havas1999aa} (``R'' for ``relators'').

The HLT strategy, like every congruence enumeration, starts by applying steps (a) and (b) in the definition (\cref{de-2-sided-enumeration-process}). We then repeatedly apply \textbf{TC3} until the resulting $\kappa$ is trivial, after which we repeatedly perform the following steps:
\begin{enumerate}
  \item[\textbf{HLT1.}]
    If $\alpha$ is the minimum node in $\Gamma_{i}$ such that there exists $(u, v)\in R$
    where either: $u$ or $v$ does not label a path starting at $\alpha$, or $u$ and $v$ label paths ending in different nodes. In the former case, we apply \textbf{TC1} repeatedly until $u$ and $v_1$ label paths starting at $\alpha$ where $v= v_1b$ for some $b\in A$, and then apply \textbf{TC2} where part (a) applies. In the latter case, we simply apply \textbf{TC2} to $\alpha$ and $(u, v)\in R$ where part (c) applies.

  \item[\textbf{HLT2.}]
    Repeatedly apply \textbf{TC3} until the resulting $\kappa_{i + 1}$ is trivial.

  \item[\textbf{HLT3.}]
    If $\alpha$ is the node from \textbf{HLT1}, then we apply
    \textbf{TC1} to $\alpha$ and every $a\in A$, if any, such that there is no edge incident to $\alpha$ labelled by $a$.
\end{enumerate}

We note that if for every $a\in A$ there exists $(u, v)\in R$ such that either $u$ or $v$ starts with $a$, then \textbf{HLT3} does not have to be performed.

Next we will show that the HLT strategy fulfils the definition of a congruence enumeration.

\begin{prop}\label{prop-hlt-is-valid}
  If $\langle A| R\rangle$ is a finite monoid presentation, $S$ is a finite
  subset of $A ^ * \times A ^ *$, $R ^ {\#}$ is the least two-sided congruence on
  $A ^ *$ containing $R$, and $\rho$ is the least right congruence on $A ^ *$
  containing $S$ and $R ^ {\#}$, then the HLT strategy applied to $(\Gamma_{1}, \kappa_{1})$ where
  $\Gamma_{1} = (N_1, E_1)$ is a word graph  such that $(0)\pi_{\Gamma_1}\subseteq \rho$ and $\kappa_{1} = \Delta_{N_1}$ is a congruence enumeration.
\end{prop}
\begin{proof}
  It suffices to check that the conditions in \cref{de-2-sided-enumeration-process}
  hold. Clearly parts (a) and (b) are included in the HLT strategy and so there is nothing to check for these.
  \medskip

  \noindent \textbf{(c).}
  Suppose that $\alpha$ is a node in $\Gamma_{i}= (N_i, E_i)$ such that there is no edge labelled by some $a\in A$ incident to $\alpha$.
  If $\alpha\in N_j$ for all $j \geq i$, then there
  exists $k\in \N$ such that $\alpha$ is the minimum node to which
  \textbf{HLT1} is applied.  If there exists $(u, v)\in R$ such that the
  first letter of $u$ is $a$, then an edge labelled by $a$ incident to $\alpha$ is defined in
  \textbf{HLT1}.  Otherwise, such an edge is defined in \textbf{HLT3}.
  \medskip

  \noindent \textbf{(d).}
  Suppose that $\alpha$ is a node in $\Gamma_{i}$ and $(u, v)\in R$ are such that
  $u$ and $v$ both label paths starting at $\alpha$.  If $\alpha\in
    N_j$ for all $j \geq i$, then there exists $k\in \N$ such that $\alpha$
  is the minimum node to which \textbf{HLT1} is applied, and so
  \textbf{TC2} is applied to $\alpha$ and $(u, v)$ at some later step.
  Hence either the paths labelled by $u$ and $v$ starting at $\alpha$ end at the same node, or they end at distinct nodes $\beta$ and $\gamma$ such that $(\beta, \gamma)\in \kappa_{i}$. But \textbf{TC3} is applied at some later step, and so $\beta$ and $\gamma$ are identified in the corresponding quotient. Eventually the paths labelled by $u$ and $v$ starting at $\alpha$ end at the same node, i.e.\ $(u, v) \in (\alpha)\pi_{\Gamma_{i}}$ for sufficiently large $i$.
  \medskip

  \noindent \textbf{(e).}
  If $(\beta, \gamma)\in \kappa_{i}$ for some $i$, then there exists
  $j\geq i$ such that \textbf{TC3} is applied by the definition of \textbf{HLT2}. If we choose the smallest such $j$, then $\kappa_{i}\subseteq \kappa_{j}$ and so $(\beta, \gamma)\not\in \kappa_{j + 1}$.
\end{proof}

\begin{ex}
  \label{ex-tc-1}
  Let $M$ be the monoid defined in \cref{ex-cayley-digraph}.
  Then $M$ is isomorphic to the monoid defined by the presentation
  \[\langle a,b,c \; | \;
    ac = a^2,\  b^2= b,\  ca= a^2,\  cb =bc,\
    c^2=  a^2,\  a^3= a^2,\  aba= a^2
    \rangle.
  \]

  We will apply the HLT strategy to the presentation and the set $S = \{(a, b)\}$.
  In other words, we enumerate the least right congruence on the monoid $M$ containing $(a, b)$. This is the same right congruence that we enumerated in \cref{ex-cayley-digraph}.
  The input word graph $\Gamma_{1}$ is the trivial graph and the input $\kappa_{1}$ is $\Delta_{N_1}$. We suppose that $a<b<c$.
  \begin{description}
    \item[Step 1:] We apply \cref{de-2-sided-enumeration-process}(a). We add the nodes $1, 2 \in N_1$ and define the edges $(0, a, 1)$ and $(0, b, 2)$. Since $S = \{(a,b)\}$ this concludes \cref{de-2-sided-enumeration-process}(a). We continue by applying \textbf{TC2} to $0$ and $(a, b)$. Since $a$ and $b$ label paths from $0$ to the nodes $1$ and $2$, respectively, we define $\kappa_{2}$ to be the least equivalence containing $\kappa_{1} = \Delta_{N_1}$ and $(1, 2)$. The output is graph $\Gamma_{2}$ in \cref{fig-hlt-1}.

    \item[Step 2:] We continue with \cref{de-2-sided-enumeration-process}(b). We apply \textbf{TC3} and we get the graph $\Gamma_{3}$ in \cref{fig-hlt-2} which is the quotient of $\Gamma_{2}$ by $\kappa_{2}$. The set $\kappa_{3}$ is defined to be $\Delta_{N_3}$.

    \item[Step 3:] Now we are ready to apply \textbf{HLT1}. Since $0$ is the minimum node in $\Gamma_{3}$ such that $(ac, a ^2) \in R$ and $ac$ and $a^2$ do not label paths in $\Gamma_{3}$, we apply \textbf{TC1} until node $3$ and a path labelled by $a^2$ are defined and we continue by applying \textbf{TC2}(b) for $u = a^2, v = ac $ and $v_1 = a$. The output of this step is the graph $\Gamma_{4}$ in \cref{fig-hlt-3}.

    \item[Step 4-8:] We apply \textbf{HLT1} to the node $0$ and each of the relations $(b^2, b), (ca, a^2), (cb, bc), (c^2, a^2), (a^3, a^2)\in R$ (in this order) and the output $(\Gamma_{i}, \kappa_{i})$, $i = 5, \ldots, 9$ are shown in \cref{fig-hlt-4}, \cref{fig-hlt-5},
      \cref{fig-hlt-6}, \cref{fig-hlt-7}, and \cref{fig-hlt-8}.

    \item[Step 9:] In this step, we do not apply \textbf{HLT1} to $0$ and $(aba, a^2)$ since $aba$ and $a^2$ both label paths that reach the same node in $\Gamma_{9}$. We apply \textbf{HLT1} to the node $1$ and the relation $(ac, a^2) \in R$ and the output is graph $\Gamma_{10}$ in \cref{fig-hlt-9}.

    \item[Step 10:] We apply \textbf{HLT1} to the node $1$ and the relation $(cb, bc) \in R$ and the output is graph $\Gamma_{11}$ in \cref{fig-hlt-10}.

      After Step 10, $\Gamma_{11}$ is complete, deterministic, and compatible with $R$. Hence the enumeration terminates; see \cref{ex-cayley-digraph} for details of the enumerated congruence.
  \end{description}

  \begin{figure}
    \begin{subfigure}{0.33\textwidth}
      \centering
      
    \includegraphics{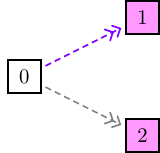}

      \caption{Step 1.}
      \label{fig-hlt-1}
    \end{subfigure}
    \begin{subfigure}{0.33\textwidth}
      \centering
      
    \includegraphics{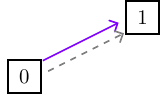}

      \caption{Step 2.}
      \label{fig-hlt-2}
    \end{subfigure}
    \begin{subfigure}{0.33\textwidth}
      \centering
      
    \includegraphics{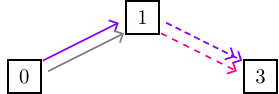}

      \caption{Step 3.}
      \label{fig-hlt-3}
    \end{subfigure}
    \vspace{\baselineskip}

    \begin{subfigure}{0.33\textwidth}
      \centering
      
    \includegraphics{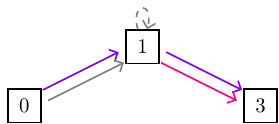}

      \caption{Step 4.}
      \label{fig-hlt-4}
    \end{subfigure}
    \begin{subfigure}{0.33\textwidth}
      \centering
      
    \includegraphics{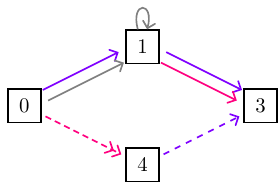}

      \caption{Step 5.}
      \label{fig-hlt-5}
    \end{subfigure}
    \begin{subfigure}{0.33\textwidth}
      \centering
      
    \includegraphics{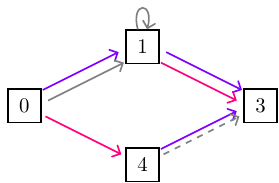}

      \caption{Step 6.}
      \label{fig-hlt-6}
    \end{subfigure}
    \vspace{\baselineskip}

    \begin{subfigure}{0.33\textwidth}
      \centering
      
    \includegraphics{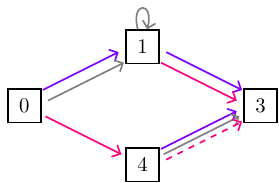}

      \caption{Step 7.}
      \label{fig-hlt-7}
    \end{subfigure}
    \begin{subfigure}{0.33\textwidth}
      \centering
      
    \includegraphics{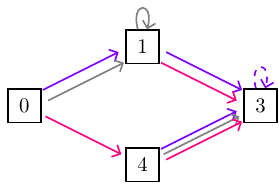}

      \caption{Step 8.}
      \label{fig-hlt-8}
    \end{subfigure}
    \begin{subfigure}{0.33\textwidth}
      \centering
      
    \includegraphics{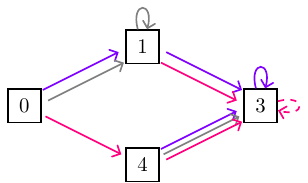}

      \caption{Step 9.}
      \label{fig-hlt-9}
    \end{subfigure}
    \vspace{\baselineskip}

    \centering
    \begin{subfigure}{0.33\textwidth}
      \centering
      
    \includegraphics{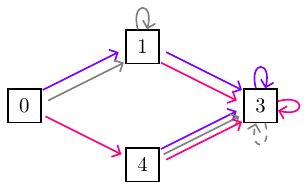}

      \caption{Step 10.}
      \label{fig-hlt-10}
    \end{subfigure}
    \caption{The output $(\Gamma_{i}, \kappa_{i})$ of each step in the HLT
      strategy in \cref{ex-tc-1}. Purple arrows correspond to $a$, gray to $b$, pink to $c$, shaded nodes of the same colour belong to $\kappa_{i}$, and unshaded nodes belong to singleton classes; see \cref{table-ex-tc-1b} for a representative word for each node.
      A dashed edge with a double arrowhead indicates the edge being defined in \textbf{TC1}, a dashed edge with a single arrowhead denotes an edge that is obtained from \textbf{TC2} or \textbf{TC3}, solid edges correspond to edges that existed at the previous step.}
    \label{fig-ex-tc-1b}
  \end{figure}
  \begin{table}\centering
    \begin{tabular}{l|l|l|l|l|l|l|l|l|l|l}
      0             & 1   & 2   & 3     & 4   \\ \hline
      $\varepsilon$ & $a$ & $b$ & $a^2$ & $c$
    \end{tabular}
    \caption{A word labelling a path from $0$ to each node in the right Cayley graph on the monoid $M$ from \cref{ex-tc-1}; see \cref{fig-ex-tc-1b}}
    \label{table-ex-tc-1b}
  \end{table}
\end{ex}

\section{The Felsch strategy}
\label{section-felsch}

In this section, we describe two versions of the so-called \textbf{\textit{Felsch strategy}} for
congruence
enumeration. For the record, this is Walker's Strategy 2
in~\cite{Walker1992aa}, and is referred to as C-style in
ACE~\cite{Havas1999aa} (``C'' for ``cosets'').

\subsection{First version}

The Felsch strategy starts with the trivial word graph $\Gamma_{1} = (N_{1}, E_{1})$ and with the trivial equivalence $\kappa_1 =\Delta_{N_1}$ on the nodes $N_1$ of $\Gamma_{1}$.
Just like the HLT strategy, the Felsch strategy starts by applying steps (a) and (b) from the definition of a congruence enumeration (\cref{de-2-sided-enumeration-process}).
We then repeatedly apply \textbf{TC3} until the resulting $\kappa$ is trivial.
The following steps are then repeatedly applied:
\begin{enumerate}
  \item[\textbf{F1.}]
    If $\alpha\in N_i$ is the minimum node in $\Gamma_{i}$ such that there exists $a\in A$ and there is no edge with source $\alpha$ labelled by $a$, then apply \textbf{TC1} to $\alpha$ and $a$.
    Apply \textbf{TC2} to every node in $\Gamma_{i}$ and every relation in $R$.
  \item[\textbf{F2.}]
    Apply \textbf{TC3} repeatedly until $\kappa_{i + 1}$ is trivial.
\end{enumerate}

It follows immediately from the definition of the Felsch strategy that it satisfies the definition of a congruence enumeration. To illustrate the Felsch strategy, we repeat the calculation from \cref{ex-tc-1}.

\begin{ex}
  \label{ex-felsch-1}
  Let $M$ be the monoid defined by defined in \cref{ex-cayley-digraph}.
  Then $M$ is isomorphic to the monoid defined by the presentation
  \[\langle a,b,c \; | \;
    ac = a^2,\  b^2= b,\  ca= a^2,\  cb =bc,\
    c^2=  a^2,\  a^3= a^2,\  aba= a^2
    \rangle.
  \]

  We will apply the Felsch strategy to the presentation (with $a<b<c$) and the set $S = \{(a, b)\}$.
  In other words, we enumerate the least right congruence on the monoid $M$ containing $(a, b)$. This is the same right congruence that we enumerated in \cref{ex-cayley-digraph}.
  The input word graph $\Gamma_{1}$ is the trivial graph and the input $\kappa_{1}$ is $\Delta_{N_1}$.

  \begin{description}
    \item[Step 1:] We apply \cref{de-2-sided-enumeration-process}(a). We add the nodes $1, 2 \in N_1$ and define the edges $(0, a, 1)$ and $(0, b, 2)$. Since $S = \{(a,b)\}$ this concludes \cref{de-2-sided-enumeration-process}(a). We continue by applying \textbf{TC2} to $0$ and $(a, b)$. Since $a$ and $b$ label paths from $0$ to the nodes $1$ and $2$, respectively, we define $\kappa_{2}$ to be the least equivalence containing $\kappa_{1} = \Delta_{N_1}$ and $(1, 2)$. The output is graph $\Gamma_{2}$ in \cref{fig-felsch-1}.

    \item[Step 2:] We continue with \cref{de-2-sided-enumeration-process}(b). We apply \textbf{TC3} and we get graph $\Gamma_{3}$ in \cref{fig-felsch-2} which is the quotient of $\Gamma_{2}$ by $\kappa_{2}$. The set $\kappa_{3}$ is defined to be $\Delta_{N_3}$

    \item[Step 3:] We apply \textbf{F1} and hence we apply \textbf{TC1} to $0$ and $c$. We add the node $3 \in N_4$ and define the edge $(0, c, 3)$. The output is graph $\Gamma_{4}$ in \cref{fig-felsch-3}.

    \item[Step 4:] We continue with the application of the second part of \textbf{F1} and hence we apply \textbf{TC2} to $1$ and $(b^2, b)$. The output is graph $\Gamma_{5}$ in \cref{fig-felsch-4}.

    \item[Step 5:] We apply \textbf{F1} and hence we apply \textbf{TC1} to $1$ and $a$. We add the node $4 \in N_6$ and define the edge $(1, a, 4)$. The output is graph $\Gamma_{6}$ in \cref{fig-felsch-5}.

    \item[Steps 6-12:] We continue with the application of the second part of \textbf{F1} and hence we apply \textbf{TC2} to $0$ and $(a^3, a^2)$, to $0$ and $(ac, a^2)$, to $1$ and $(ac, a^2)$, to $0$ and $(ca, a^2)$, to $0$ and $(c^2, a^2)$ and to $3$ and $(b^2, b)$.

     After Step 12, $\Gamma_{13}$ in \cref{fig-felsch-6} is complete, deterministic, and compatible with $R$. Hence the enumeration terminates; see \cref{ex-cayley-digraph} for details of the enumerated congruence.
  \end{description}

  \begin{figure}
    \begin{subfigure}{0.33\textwidth}
      \centering
      
    \includegraphics{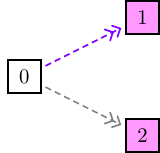}

      \caption{Step 1.}
      \label{fig-felsch-1}
    \end{subfigure}
    \begin{subfigure}{0.33\textwidth}
      \centering
      
    \includegraphics{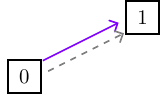}

      \caption{Step 2.}
      \label{fig-felsch-2}
    \end{subfigure}
    \begin{subfigure}{0.33\textwidth}
      \centering
      
    \includegraphics{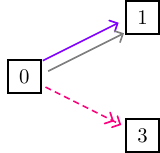}

      \caption{Step 3.}
      \label{fig-felsch-3}
    \end{subfigure}
    \begin{subfigure}{0.33\textwidth}
      \centering
      
    \includegraphics{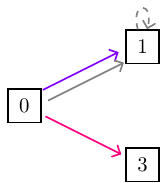}

      \caption{Step 4.}
      \label{fig-felsch-4}
    \end{subfigure}
    \begin{subfigure}{0.33\textwidth}
      \centering
      
    \includegraphics{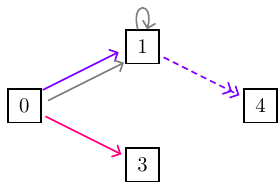}

      \caption{Step 5.}
      \label{fig-felsch-5}
    \end{subfigure}
    \begin{subfigure}{0.33\textwidth}
      \centering
      
    \includegraphics{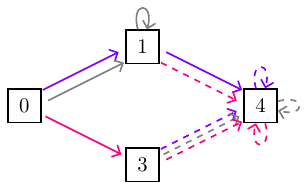}

      \caption{Steps 6-12.}
      \label{fig-felsch-6}
    \end{subfigure}

    \caption{The output $(\Gamma_{i}, \kappa_{i})$ of each step in the Felsch
      strategy in \cref{ex-felsch-1}. Purple arrows correspond to $a$, gray to $b$, pink to $c$, shaded nodes of the same colour belong to $\kappa_{i}$, and unshaded nodes belong to singleton classes; see \cref{table-ex-felsch-1} for a representative word for each node.
      A dashed edge with a double arrowhead indicates the edge being defined in \textbf{TC1}, a dashed edge with a single arrowhead denotes an edge that is obtained from \textbf{TC2} or \textbf{TC3}, solid edges correspond to edges that existed at the previous step.}
    \label{fig-ex-felsch-1b}
  \end{figure}
  \begin{table}\centering
    \begin{tabular}{l|l|l|l|l|l|l|l|l|l|l}
      0             & 1   & 2   & 3   & 4     \\ \hline
      $\varepsilon$ & $a$ & $b$ & $c$ & $a^2$
    \end{tabular}
    \caption{A word labelling a path from $0$ to each node in the right Cayley graph on the monoid $M$ from \cref{ex-felsch-1}; see \cref{fig-ex-felsch-1b}}
    \label{table-ex-felsch-1}
  \end{table}
\end{ex}
\subsection{Second version}

The purpose of \textbf{F2} in the Felsch stategy is to squeeze as much
information as possible out of every definition of an edge $(\alpha, a, \beta)$ made in
\textbf{F1}. The implementation of the Felsch strategy in
\cite{Mitchell2021aa} spends most of its time performing \textbf{TC2}.
In this section, we propose a means of reducing the number of times
\textbf{TC2} is applied in \textbf{F2} of the Felsch strategy. Roughly speaking,
we do this by only applying \textbf{TC2} to $(u, v) \in R$ and a node $\alpha$ in $\Gamma_{i}$
if the path in $\Gamma_{i}$ starting at $\alpha$ and labelled by $u$, or $v$, goes
through a part of $\Gamma_{i}$ that has recently changed.

To enable us to do this, we require a new set $D_i \subseteq N_i \times A$, in addition to
the word graph $\Gamma_{i}$ and equivalence relation $\kappa_{i}$, at every step of a congruence enumeration.  An element $(\alpha, a)$ of $D_i$ corresponds to a recently defined edge incident to $\alpha$ and labelled by $a$.  More precisely, we define
$D_1 = \varnothing$ and for $i\geq 1$ we define
\[
  D_{i+1} := \set{(\alpha, a)\in N_i\times A}{(\alpha, a, \beta)\in E_{i + 1}\setminus E_i}.
\]
In order to efficiently use the information recorded in $D_i$, we require the following definition.

\begin{de}
  If $\langle A | R\rangle$ is a finite monoid presentation, then we define the
  \defn{Felsch tree} $F(A, R)$ of this presentation to be a pair $(\Theta, \iota)$ where:
  \begin{enumerate}
    \item $\Theta= (N, E)$
          is a graph with
          nodes $N$ consisting of every (contiguous) subword $w\in A^ *$ of a word in a
          relation in $R$ and edges $(u, a, au)\in E$
          whenever $u, au \in N$ and $a\in A$; and
    \item a function $\iota$ from  $N$ to the power set
          $\mathfrak{P}(R)$ of $R$ such that $(u, v)\in (w)\iota$ whenever $w\in N$ is a prefix of
          $u$ or $v$.
  \end{enumerate}
\end{de}


\begin{ex}
  \label{ex-felsch-tree}
  If $A = \{a, b\}$ and $R = \{(a ^ 4, a), (b ^ 3, b), ((ab) ^ 2, a ^
    2)\}$, then the nodes in the word graph $\Theta$ in $F(A, R)$ are
  \[N = \{\varepsilon, a, b, a ^ 2, ab, ba, b ^ 2,
    a ^ 3, aba,bab, b ^ 3, a ^4, (ab) ^ 2 \}.\]
  A diagram of $\Theta$ can found in \cref{fig-felsch-tree}.
  \begin{figure}
    \centering
    
    \includegraphics{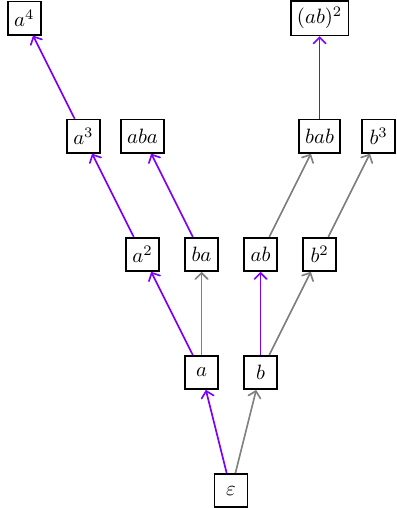}

    \caption{The Felsch tree $F(A, R)$ from \cref{ex-felsch-tree}.}
    \label{fig-felsch-tree}
  \end{figure}

  The function $\iota: N \to \mathfrak{P}(R)$ is:
  \begin{eqnarray*}
    (\varepsilon)\iota & = & R, \\
    (a)\iota  =  (a ^ 2)\iota & = & \{(a ^ 4, a), ((ab) ^ 2, a ^ 2)\}\\
    (a ^ 3)\iota  =  (a ^ 4)\iota & = & \{(a ^ 4, a)\}\\
    (aba)\iota & = & \{((ab) ^ 2, a ^ 2)\}\\
    (b)\iota = (b ^ 2)\iota = (b ^ 3)\iota & = &  \{ (b ^ 3, b)\}\\
    (ab)\iota = ((ab) ^ 2)\iota & = & \{((ab) ^ 2, a ^ 2)\} \\
    (ba)\iota = (bab)\iota & = & \varnothing.
  \end{eqnarray*}
\end{ex}


We modify
\textbf{TC1}, \textbf{TC2}, and \textbf{TC3} so that the set $D_{i}$
is defined appropriately at every step, and then, roughly speaking,
we replace \textbf{F2} in the Felsch strategy by a backtrack
search through $\Theta$ in $F(A, R)$ for every pair in $D_{i}$.  We will refer to this as the
\defn{modified Felsch strategy}.


If $\alpha$ is a node in $\Gamma_{i}$ and $v\in A ^ *$ is a node in $\Theta$, then we perform a backtrack search consisting of the following steps:
\begin{description}
  \item[PD1.]
    Apply \textbf{TC2} to $\alpha$ and every relation in $(v)\iota$.
  \item[PD2.]
    For every edge $(v, b, bv)$ in $\Theta$ that has not been traversed:
    for every node $\beta$ in $\Gamma_{i}$ such that $(\beta, b, \alpha) \in E_i$,
    apply \textbf{PD1} and \textbf{PD2} to $\beta$ and $bv$.
  \item[PD3.]
    Repeatedly apply \textbf{TC3} until $\kappa$ is trivial.
\end{description}

We will refer
 to either removing a pair from $D_i$ or performing the
 backtrack search as \textit{processing a deduction}.

The backtrack search is initiated for every $(\alpha, a) \in D_i$.
If $a$ is not a node in $\Theta$, then $a$ does not occur in any relation in $R$ and so following the path in $\Gamma_{i}$ labelled by any $u$ such that $(u, v) \in R$ from any node in $w$ in $\Gamma_{i}$ cannot contain any edge labelled by $a$. In particular,
no such path contains the edge $(\alpha, a, \beta)\in \Gamma_{i}$; the definition of which caused $(\alpha, a)$ to belong to $D_{i}$ in the first place. Hence if $(\alpha, a)\in D_i$, then there is a node $a$ in $\Theta$.

We repeatedly apply \textbf{PD1}, \textbf{PD2}, and
\textbf{PD3} starting from every pair in $D_{i}$.
This must terminate eventually because $\Theta$ is finite and we only apply \textbf{TC2} and \textbf{TC3} in \textbf{PD1}, \textbf{PD2}, and \textbf{PD3}.

To show that the modified Felsch strategy is a congruence enumeration, we require the following lemma.

\begin{lemma}\label{lemma-backtrack}
  Suppose that $\alpha$ is a node in $\Gamma_{j}$ for all $j\geq i$, and that $(u, v)\in R$ is such that $u$ and $v$ label paths $P_u$ and $P_v$ in $\Gamma_{i}$ starting at $\alpha$.
  If $(\beta, a) \in D_{i}$ for some $i$, and there is an edge $(\beta, a, \gamma)\in \Gamma_i$ in either $P_u$ or $P_v$, then $(u, v) \in (\alpha)\pi_{\Gamma_k}$ for some $k\geq i$.
\end{lemma}
\begin{proof}
  Suppose that $P_u$ consists of the edges $(\mu_1, b_1, \mu_2), \ldots, (\mu_{r}, b_r, \mu_{r + 1})$ where
  $\mu_1 = \alpha$ and $u = b_1\cdots b_r$. Similarly, suppose that $P_v$ consists of the edges $(\nu_1, c_1, \nu_2), \ldots, (\nu_s, c_s, \nu_{s + 1})$ where $\nu_1 = \alpha$  and $v = c_1\cdots c_s$. We may assume that $\mu_{r + 1} \neq \nu_{s + 1}$ and that there exists $t\in \{1, \ldots, r\}$ such that $(\mu_t, b_t, \mu_{t + 1}) = (\beta, a, \gamma)$.

  Since  $(\beta, a)\in D_{i}$,  the backtrack search in \textbf{PD1} and \textbf{PD2} will be applied to $(\beta, a)$ at some step $j\geq i$ of the enumeration. This search begins with \textbf{PD1}, where \textbf{TC2} is applied to $\beta$ and every relation in $(a)\iota$. At step \textbf{PD2} we apply \textbf{PD1} and \textbf{PD2}
  to every $\zeta\in N_i$ and $w$ such that there exists an edge of the form $(a, b, w)$ in $\Theta$ and an edge $(\zeta, b, \beta) \in E_i$.

  By definition, $u = b_1\cdots b_{t - 1}a b_{t + 1}\cdots b_r$.
  If $m\in \{1, \ldots, t\}$, then,
  since $(u, v) \in R$, $b_m\cdots b_{t - 1}a$ and $b_{m - 1}\cdots b_{t - 1}a$ are nodes of $\Theta$ and there exists an edge $(b_m\cdots b_{t - 1}a, b_{m - 1}, b_{m - 1}\cdots b_{t - 1}a)$ in $\Theta$. Also by assumption $(\mu_{m - 1}, b_{m - 1}, \mu_m) \in E_i$ and so
  $(\mu_{m - 1}, b_{m - 1}\cdots b_{t - 1}a)$ is one of the pairs in \textbf{PD2} if it is applied to
  $(\mu_m, b_{m}\cdots b_{t - 1}a)$. If $m = t$, then $(\mu_m, b_{m}\cdots b_{t - 1}a) = (\mu_t, a) = (\beta, a)$
  and hence \textbf{PD2} is applied to
  $(\mu_m, b_{m}\cdots b_{t - 1}a)$ for every $m\in \{1, \ldots, t\}$.
  In particular, \textbf{PD1} is applied to $(\mu_1, b_1\cdots b_{t - 1}a) = (\alpha, b_1\cdots b_{t - 1}a)$ at some point. This means that \textbf{TC2} is applied to $\alpha$ and every relation in $(b_1\cdots b_{t - 1}a)\iota$. Since $b_1\cdots b_{t - 1}a = b_1\cdots b_{t}$ is a prefix of $u$, it follows that $(u, v)\in (b_1\cdots b_{t - 1}a)\iota$ and so \textbf{TC2} is applied to $\alpha$ and $(u, v)$ at some step $j \geq i$. Because $u$ and $v$ were assumed to label paths originating at $\alpha$, neither \textbf{TC2}(a) nor (b) applies. Additionally, we assumed that $\mu_{r + 1}\neq \nu_{s + 1}$ and so
  \textbf{TC2}(c) is applied and  $(\mu_{r + 1}, \nu_{s + 1})\in \kappa_{i + 1}$.

  As noted above, $\Theta$ is finite, and \textbf{PD1} and \textbf{PD2} are only applied finitely many times before there is an application of \textbf{PD3}. If the first application of \textbf{TC3} (in \textbf{PD3}) after step $j$ occurs at step $k$, then $\Gamma_{k + 1}$ is the quotient of $\Gamma_{k}$ by $\kappa_{k}$ and $\kappa_{i}\subseteq \kappa_{k}$. Therefore $(u, v)\in (\alpha)\pi_{\Gamma_{k + 1}}$ as required.
\end{proof}

\begin{prop}\label{prop-modified-felsch-is-valid}
  If $\langle A| R\rangle$ is a finite monoid presentation, $S$ is a finite
  subset of $A ^ * \times A ^ *$, $R ^ {\#}$ is the least two-sided congruence on
  $A ^ *$ containing $R$, and $\rho$ is the least right congruence on $A ^ *$
  containing $S$ and $R ^ {\#}$, then the modified Felsch strategy applied to $(\Gamma_{1}, \kappa_{1})$, where
  $\Gamma_{1} = (N_1, E_1)$ is the trivial word graph and $\kappa_{1} = \Delta_{N_1}$, is a congruence enumeration.
\end{prop}
\begin{proof}
  The only difference between the modified Felsch strategy and the original Felsch strategy is that after \textbf{F1}, \textbf{TC2} is only applied to particular nodes and relations. Hence, it suffices to show that \cref{de-2-sided-enumeration-process}(d) holds. Assume $\alpha \in N_i$ at some step $i$ in the congruence enumeration and let $(u, v) \in R$. In order for modified Felsch to be a congruence enumeration we need to show that there exists $j \geq i$ such that either $\alpha \notin N_j$ or $(u, v) \in (\alpha)\pi_{\Gamma_j}$.
  If $\alpha\notin N_j$ for some $j > i$, then $\alpha\notin N_k$ for all $k \geq j$
  (since new nodes introduced in \textbf{TC1} are  larger than all previous nodes).
  Hence it suffices to prove that if $\alpha$ is a node for all $j \geq i$, there exists some $k$ such that $(u,v) \in (\alpha)\pi_{\Gamma_k}$.
  Since \textbf{F1} is repeatedly applied there exists a step $k + 1$ when
  $\alpha$ is a node in $\Gamma_{k + 1}$ and paths $P_u$ and $P_v$ leaving $\alpha$ labelled by $u$ and $v$, respectively, exist in $\Gamma_{k + 1}$.
  Suppose that $k\in \N$ is the least value such that this holds. Then, by \cref{lemma-backtrack}, it suffices to show that there exists an edge $(\beta, a, \gamma)$ in either $P_u$ or $P_v$ that belongs to $E_{k + 1}\setminus E_k$  so that $(\beta, a)\in D_{k + 1}$. If every edge in $P_u$ and $P_v$ belongs to $E_k$, then $u$ and $v$ both label paths in $\Gamma_{k}$ starting at $\alpha$, which contradicts the assumed minimality of $k$.
\end{proof}

\section{Implementation issues}\label{section-implement}

In this section, we briefly address some issues relating to any implementation of the Todd-Coxeter algorithm for semigroups and monoids as described herein.

In some examples, it can be observed that the HLT
strategy defines many more nodes than the Felsch strategy.
One possible antidote to
this is to, roughly speaking, perform
\begin{quote}
  \textit{``periods of definition \textit{\`a la HLT} [that] alternate with
    periods of intensive scan \textit{\`a la Felsch}''}~\cite[p. 14]{Neubuser1982aa}.
\end{quote}
In both ACE~\cite{Havas1999aa} and \libsemigroups it is possible to specify the precise lengths
of the periods of applications of HLT and Felsch. As might be expected, some choices work better than others in particular examples, and this is difficult (or impossible) to predict in advance. It is routine to show that alternating between HLT and Felsch in this way still meets the definition of a congruence enumeration given in \cref{de-2-sided-enumeration-process}. Although we presented the HLT and Felsch strategies separately, it seems that some combination of the two sometimes offers better performance.

The next issue is: how to represent the equivalence relations $\kappa_i$?
A method suggested in~\cite[Section 4.6]{Sims1994aa}, which is now a standard approach for representing
equivalence relations, is to use the disjoint-sets data structure to represent the
least equivalence relation containing the pairs $(u, v)$ added to $\kappa_i$ in \textbf{TC2}(c)
or \textbf{TC3}.  The theoretical time complexity of updating the data structure to merge two classes, or to find a canonical representative of a class given another representative,  is $O(\alpha(m))$  time (in both the worst and the average case) and requires $O(m)$ space where $m$ is the number of elements in the underlying set and $\alpha$ is the inverse Ackermann function; see~\cite{Tarjan1975aa}.

Another issue that arises in the implementation is how to represent the word graphs $\Gamma_i$. In order to efficiently obtain the word graph $\Gamma_{i + 1}$ from $\Gamma_i$ in \textbf{TC3} it is necessary to keep track of both the edges with given source node, and with given target. This is more complex for semigroups and monoids than for groups, because a word graph $\Gamma=(N, E)$ ouput by a coset enumeration for a group has the property that for every $\beta\in N$ and every $a \in A$ there is exactly one $\alpha\in N$ such that $(\alpha, a, \beta)$ is an edge. As such if it ever arises that there are edges $(\alpha_1, a, \beta)$ and $(\alpha_2, a, \beta)$ in a word graph during a coset enumeration, the pair $(\alpha_1, \alpha_2)$ can immediately be added to $\kappa_i$. It is therefore possible to represent every word graph arising during a coset enumeration to have the property that $|\set{\alpha\in N}{(\alpha, a, \beta)\in E}| = 1$, which simplifies the data structure required to represent such a graph.

In contrast, if $\Gamma = (N, E)$ is a word graph arising in a congruence enumeration for a monoid, then
$|\set{\alpha\in N}{(\alpha, a, \beta)\in E}|$ can be as large as $|N|$. In practice, in \textbf{TC3} pairs of nodes belonging to $\kappa_i$ are merged successively.
A balance must be struck between repeatedly updating the data structure for the edges with given target in $\Gamma_{i + 1}$  or only retaining the edges with given source and rebuilding the data structure for the target edges later in the process. The former works better if the number of nodes in  $\Gamma_i$ is comparable to the number in $\Gamma_{i + 1}$, i.e.\ only a relatively small number of nodes are merged. On the other hand, if $\Gamma_{i + 1}$ is considerably smaller than $\Gamma_i$, then it can be significantly faster to do the latter.

Depending on the sequence of applications of \textbf{TC1}, \textbf{TC2}, and \textbf{TC3} in two successful congruence enumerations with the same input, the output word graphs may not be equal. However, the output word graphs are
unique up to isomorphism. \textit{Standardization} is a process for transforming a word graph into a standard form.
To discuss this we require the following notion.

If $A$ is any alphabet and $\preceq$ is a total order on $A ^ *$, then we say
that $\preceq$ is a \textbf{\textit{reduction ordering}} if $\preceq$ has no infinite
descending chains and if $u\preceq v$ for some $u, v\in A ^ *$, then $puq \preceq
  pvq$ for all $p, q\in A ^ *$. It follows from this definition that $\varepsilon$ is the $\preceq$-minimum word in $A ^ *$ for every reduction ordering on $A ^ *$.

If the set $A$ is totally ordered by $\leq$, then
we may extend $\leq$ to $\leq_{\text{lex}}$ over $A^*$, by defining $\varepsilon\leq_{\text{lex}} w$ for all $w\in A ^*$ and
$u\leq_{\text{lex}} v$ whenever $u = au_1$ and $v = bv_1$ for some $a, b\in A$ with $a<b$, or $a=b$ and $u_1 \leq_{\text{lex}} v_1$. This order is usually referred to as the \defn{lexicographic order} on $A ^ *$. Note that lexicographic order is not a reduction ordering. The \defn{short-lex order} $\leq_{\text{slex}}$ on $A ^ *$ is defined as follows: if $u, v \in A ^*$, then $u \leq_{\text{slex}} v$ if $|u| < |v|$ or $|u| = |v|$ and $u \leq_{\text{lex}} v$. It is straightforward to verify that the short-lex order on $A ^ *$ is a reduction ordering.
Further examples of reduction orderings on $A ^ *$ include recursive path
descent, as well as the wreath product of any finite collection of reduction orderings; see~\cite[Section 2.1]{Sims1994aa} for further details.

Suppose that $\preceq$ is a reduction ordering on $A ^ *$.
We will say that the word graph $\Gamma = (N, E)$ is \defn{standardized
  with respect to $\preceq$} if
$\alpha < \beta$ if and only if $w_{\alpha} \prec w_{\beta}$
for any $\alpha, \beta \in N$ where $w_{\alpha}, w_{\beta}\in A ^ *$  are
the $\preceq$-minimum words labelling $(0, \alpha)$- and $(0, \beta)$-paths, respectively.
Any process that transforms a word graph $\Gamma$ into a standardized word graph, is referred to as \defn{standardization}; for example, see \textsf{STANDARDIZE} in~\cite[p195]{Sims1994aa}.

A word graph $\Gamma_i$ can be replaced by any standardized word graph at any step of a congruence enumeration, provided that the values in $\kappa_i$ and $D_i$ are also updated accordingly.
In particular, it can be applied repeatedly during a congruence enumeration, or only at the end.
Standardization during a congruence enumeration can be very costly in the context of semigroups and monoids. However, it can also be somewhat beneficial in some examples.

The order of the definitions of new nodes in $\Gamma_{i + 1}$ in both \textbf{HLT1} and \textbf{F1}
depends on the numerical values of their source nodes. Hence replacing $\Gamma_i$ by a standardized word graph can change the order of these definitions, which in turn can influence the number of steps in the enumeration.  In an actual implementation, standardising a word graph is a rather complicated process that involves applying a permutation to the data structure representing the graph.

\section{Further variants}\label{section-variants}

In this section, we present some variants of the Todd-Coxeter algorithm that appear in the literature, which are used to compute special types of congruences (namely Rees congruences) and for computing congruences on finitely presented inverse monoids.

\subsection{Monoids with zero}

The first such variant is for finitely presented monoids with a zero element $0$.
For example, if $M$ is the monoid defined by the presentation
\[
  \langle a, b, 0 | ab = 0, a ^ 4 = a, b ^ 3 = b, (ab) ^ 2 = 0, a0=0a=0=b0=0b=0 ^2\rangle,
\]
then the relations $a0=0a=0=b0=0b=0 ^2$ indicate that $0$ is a zero element of $M$. Of course, the algorithms described above can be applied to this finite monoid presentation, as well as every other. On the other hand, the inclusion of the relations $a0=0a=0=b0=0b=0 ^2$ is rather cumbersome, and so we might rather write:
\[
  \langle a, b | ab = 0, a ^ 4 = a, b ^ 3 = b, (ab) ^ 2 = 0\rangle
\]
where the relations $a0=0a=0=b0=0b=0 ^2$ are implicit.
This is directly analogous to the implicit relations for the identity in a monoid presentation, and in a group presentation for inverses. We refer to such a presentation as a \defn{finite monoid-with-$0$ presentation}.
Both the HLT and Felsch strategies can be adapted  for monoid-with-$0$ presentations without much difficulty as follows.

Suppose that $\langle A | R\rangle$ is a finite monoid-with-$0$ presentation.
We refer to a word graph $\Gamma_i = (N_i, E_i)$ over $A \cup \{0\}$ with a distinguished node $\omega\in N_i$ such that the only edges with source $\omega$ are loops of the form $(\omega, a, \omega) \in E_i$ for all $a\in A$ and edges $(\alpha, 0, \omega)$ for all $\alpha \in N_i$ as a \defn{word graph-with-$0$}.
We augment \textbf{TC1} with the following step:
\begin{enumerate}
  \item[\textbf{Z:}]
    If $\beta$ is the new node introduced in \textbf{TC1}, then we  define the edge $(\beta, 0, \omega)$.
\end{enumerate}
Given a monoid-with-$0$ presentation,  it is routine to verify that if we perform any congruence enumeration (where \textbf{TC1} is augmented with \textbf{Z})  with input $(\Gamma_1, \kappa_1)$ where $\Gamma_1 = (N_1, E_1)$ is a word graph-with-$0$ and $\kappa_1 = \Delta_{N_1}$, then the conclusions in \cref{cor-tc-terminates} still hold.

\subsection{Rees congruences}
Following~\cite[Chapter 12]{Ruskuc1995aa}, we may extend the discussion of the previous section, to obtain a procedure for enumerating a left, right, or two-sided Rees congruence on a finitely presented monoid (with or without zero element).
If $I$ is a left, right, or two-sided ideal of the monoid $S$, then the \defn{Rees congruence} associated with $I$ is the congruence
$\Delta_{S} \cup (I\times I)$.
Such a procedure applies to a finite monoid presentation
$\langle A | R \rangle$ and set finite $S \subseteq A ^ * \times \{0\}$ (rather than $S \subseteq A ^ * \times A ^ *$ as in \cref{de-2-sided-enumeration-process}). The input to such an enumeration is a word graph-with-zero $\Gamma_1 = (N_1, E_1)$  and $\kappa_1 = \Delta_{N_1}$. The first steps are identical to those given in \cref{de-2-sided-enumeration-process}(a) and (b) except that \textbf{TC1} and \textbf{Z} are applied in part (a). The subsequent steps are just any sequence of applications of \textbf{TC1}+\textbf{Z}, \textbf{TC2}, and \textbf{TC3} satisfying \cref{de-2-sided-enumeration-process}(c), (d) and (e).
It follows immediately from \cref{validity}, and the validity of the congruence enumeration for monoid-with-$0$ presentations, that this process is valid.

\subsection{Stephen's procedure}

Another variant of the Todd-Coxeter algorithm is that of Stephen~\cite[Chapter 4]{Stephen1987aa}, mentioned in the introduction of the current article. Note that a similar method for constructing the Cayley graph of groups was described by Dehn~\cite{Dehn1911}.
Suppose that $M$ is the monoid defined by a finite monoid presentation $\langle A | R\rangle$ and that
$\Gamma=(N, E)$ is the right Cayley graph of $M$ with respect to $A$. If $w\in A ^ *$ is arbitrary and $w$ labels a $(0, \alpha)$-path in $\Gamma$ for some $\alpha \in N$, then
the aim of this variant is to output the subgraph $\Lambda$ of  $\Gamma$ induced by the set $X$ of nodes
in $N$ from which $\alpha$ is reachable. Note that the set of these nodes corresponds to the set of elements in $M$ which are $\geq_{\mathcal{R}} w/{R^{\#}}$ (recall that two monoid elements $s, t \in M$ satisfy $s \geq_{\mathcal{R}} t$ if $tM \subseteq sM$).
If $\mathcal{A}$ is the automata with alphabet $A$, state set $X$, initial state $0$, accept state $\alpha$, and edges consisting of those in $\Lambda$,
then the language $L(\mathcal{A})$ accepted by $\mathcal{A}$ is the set of words in $v\in A ^ *$ that represent the same element of $M$ as $w$ (i.e.\ $v/R ^ {\#} = w / R ^ {\#}$). As such if the (as yet to be described) procedure terminates, the output allows us to decide the word problem for $w$ in $M$.

Suppose that $w = a_1 \cdots a_n\in A ^ *$ for some $a_1, \ldots, a_n \in A$, that $\Gamma_1 = (N_1, E_1)$ is the trivial word graph, and $\kappa_1=\Delta_{N_1}$.
A special case of Stephen's procedure consists of the following steps described in terms of \textbf{TC1}, \textbf{TC2}, and \textbf{TC3}. The first step is always:
\begin{enumerate}
  \item[\textbf{S1:}] \textbf{TC1} is applied to $0$ and $a_1$, then to $i$ and $a_{i + 1}$ for every $i$ such that $1\leq i \leq n - 1$. The resulting $\Gamma_n$ consists of the single path from $0$ to the node $n$. (The graph $\Gamma_n$ is referred to as the \defn{linear graph of $w$} in~\cite{Stephen1987aa}.)
\end{enumerate}
\textbf{S1} is then followed by any sequence of the following steps:
\begin{enumerate}
  \item[\textbf{S2:}]
    At step $i$, suppose that the word graph $\Gamma_i = (N_i, E_i)$ contains a path with source $\alpha\in N_i$ labelled $u$ for some $(u, v) \in R$. If $v = v_1b$ where $v_1\in A ^ *$ and $b\in A$, then \textbf{TC1} is applied until there is a path with source $\alpha$ labelled by $v_1$
    and then \textbf{TC2} is applied to $\alpha$ and $(u, v)\in R$. (This is referred to as an \defn{elementary expansion} in~\cite{Stephen1987aa}.)
  \item[\textbf{S3:}]
    Apply \textbf{TC3}.
    (Quotienting $\Gamma_i$ by the least equivalence containing a single $(\alpha, \beta) \in N_i \times N_i$ is referred to as a \defn{determination} in~\cite{Stephen1987aa}. In \textbf{TC3} we quotient $\Gamma_i$ by the entire equivalence $\kappa_i$; this is the only point where the procedure described here differs from the description in~\cite{Stephen1987aa}.)
\end{enumerate}

It is shown in~\cite{Stephen1987aa} that if any sequence of \textbf{S2} and \textbf{S3} has the property that after finitely many steps any subsequent applications of \textbf{S2} and \textbf{S3} result in no changes to the output (i.e.\ $(\Gamma_{i + 1}, \kappa_{i + 1}) = (\Gamma_{i}, \kappa_{i})$), then $\Gamma_i$ is isomorphic to the induced subgraph $\Lambda$ defined at the start of this section. Note that \textbf{S2} and \textbf{S3} are similar to \textbf{HLT1}, \textbf{HLT2}, and \textbf{HLT3} described in \cref{section-hlt}. While it is not possible to use congruence enumeration, at least as described in this paper, to solve the word problem when the monoid $M$ defined by a presentation $\langle A | R \rangle$ is infinite, it is possible to decide whether or not two words $u, v\in A ^ *$ represent the same element of $M$ using the procedure defined in this section, whenever the induced subgraph $\Lambda$ is finite. Since the set of nodes in $\Lambda$ corresponds to the set of elements in $M$ which are $\geq_{\mathcal{R}} u/R^{\#}$ (assuming that $u$ is the input word for Stephen's procedure) the word graph $\Lambda$ is finite when there are only finitely many elements of $M$ that are $\geq _{\mathcal{R}} u/R^{\#}$.

%


\section*{Acknowledgements}
The authors would like to thank the referee for their careful reading of the paper, and for their helpful suggestions.
The third author was supported by the EPSRC doctoral training partnership number EP/N509759/1 when working on this project. 
The fourth author would like to thank the School of Mathematics and Statistics of the University of St Andrews and the Cyprus State Scholarship Foundation for their financial support.


\printbibliography


\appendix
\newpage
\section{Extended examples}
\label{appendix-extended-examples}

This appendix contains a number of extended examples of congruence enumerations.

\begin{ex}
	\label{ex-felsch-A1}
	We will apply the Felsch strategy to the presentation
	\[
		\P = \langle a, b\ |\ a ^ 3 = a,\ b ^ 3 = b,\ (ab) ^ 2 = a ^ 2\rangle
	\]
with $a<b$. 
	The input word graph $\Gamma_{1}$ is the trivial graph and the input $\kappa_{1} = \Delta_{N_1}$. Since $S = \varnothing$, we do not apply steps \cref{de-2-sided-enumeration-process}(a) and (b).

	For the sake of simplicity, the steps in this example correspond to either a single application of \textbf{F1} (a single application of \textbf{TC1} and multiple applications of \textbf{TC2}) or a single application of \textbf{TC3}. If a step produces no change to $\Gamma_i$ or $\kappa_i$, this step is skipped and does not have a number; see \cref{table-ex-felsch-A1}, \cref{fig-ex-felsch-A1-a}, \cref{fig-ex-felsch-A1-b}, \cref{fig-ex-felsch-A1-c}, \cref{fig-ex-felsch-A1-d}, and \cref{fig-ex-felsch-A1-e}.

	\begin{description}
		\item[Step 1-2:] The only node in $\Gamma_{1}$ is $0$. Since there is no edge incident to $0$ labelled by $a$, in Step 1 we apply \textbf{F1} and add the node $1$ and the edge $(0, a, 1)$. Similarly, in Step 2 we add the node $2$ and the edge $(0, b, 2)$. The output is the word graph $\Gamma_{3}$ in \cref{subfig-felsch-2}.

		\item[Step 3:]  An application of \textbf{TC1} which leads to the definition of the node $3$ and the edge $(1, a, 3)$. At this point \textbf{F1} leads to an application of \textbf{TC2} to the node $3$ and the relation $(a^3, a)$ and hence we define the edge $(3, a, 1)$. The output of step 3 is $\Gamma_{4}$; see \cref{subfig-felsch-3}.

		\item[Step 4-5:] We apply \textbf{TC1} twice (since there are no applications of \textbf{TC2} that yield new information). In step 4 we define the node $4$ and the edge $(1, b, 4)$ and in step 5 the node $5$ and the edge $(2, a, 5)$; see \cref{subfig-felsch-5}.

		\item[Step 6:] The node $6$ and the edge $(2, b, 6)$ are defined. Applying \textbf{TC2} to the node $6$ and the relation $(b^3, b) \in R$ leads to the definition of the edge $(6, b, 2)$. The output of step 6 is $\Gamma_{7}$; see \cref{subfig-felsch-6}.

		\item[Step 7:] We apply \textbf{TC1} and define the node $7$ and the edges $(3, b,  7)$; see \cref{subfig-felsch-7}.

		\item[Step 8:]
			The node $8$ and edge $(4, a, 8)$ are defined in \textbf{TC1}. Applying \textbf{TC2} leads to the definition of edges $(7, b, 3)$, $(8, b, 3)$, and $(8, a, 4)$; see \cref{subfig-felsch-8}.

		\item[Step 9:] We apply \textbf{TC1} and define the node $9$ and the edge $(4, b, 9)$. We apply \textbf{TC2} to $4$ and the relation $(b^3, b)$ and define the edge $(9, b, 4)$; see \cref{subfig-felsch-9}.

		\item[Step 10:] We apply \textbf{TC1} and define the node $10$ and the edge $(5, a, 10)$. We apply \textbf{TC2} to the node $5$ and the relation $(a^3, a)$ and define the edge $(10, a, 5)$; see \cref{subfig-felsch-10}.

		\item[Step 11-12:]  We apply \textbf{TC1} twice, since there are no applications of \textbf{TC2} yielding any new information. The nodes $11$ and $12$ and the edges $(5, b, 11)$ and $(6, a, 12)$ are defined; see \cref{subfig-felsch-12}.

		\item[Step 13:] We apply \textbf{TC1} and define the node $13$ and the edge $(7, a, 13)$. We apply \textbf{TC2} to the node $1$ and the relation $((ab) ^ 2, a^2)$ and define the edge $(13, b, 1)$. We also apply \textbf{TC2} to node $7$ and the relation $((ab) ^ 2, a^2)$ and define the edge $(13, a, 7)$. Next, we apply \textbf{TC2} to the node $1$ and the relation $(b^3, b)$ resulting in the pair $(1, 9)$ begin added to $\kappa_{14}$. Finally, we apply \textbf{TC2} to the node $13$ and the relation $((ab) ^ 2, a^2)$ which yields $(4, 13)$ being added to $\kappa_{14}$; see \cref{subfig-felsch-13}.

		\item[Step 14:] In this step $\kappa_{14}$ is the least equivalence relation containing $\{(1, 9), (4, 13)\}$. Taking the quotient gives graph $\Gamma_{15}$. In $\Gamma_{15}$ there exist two edges $(4, a, 7)$ and $(4, a, 8)$ with the same source and label, and so the pair $(7, 8)$ is added to $\kappa_{15}$; see \cref{subfig-felsch-14}.

		\item[Step 15:]  Taking the quotient of $\Gamma_{15}$ in \textbf{F2} by $\kappa_{15}$ yields the graph in \cref{subfig-felsch-15}.

		\item[Step 16:] We apply \textbf{TC1} and define the node $14$ and edge $(10, b, 14)$; see \cref{subfig-felsch-16}.

		\item[Step 17:] We apply \textbf{TC1} and define the node $15$ and the edge $(11, a, 15)$. Applying \textbf{TC2} to $2$ and the relation $((ab) ^ 2, a^2)$ leads to the definition of the edge $(15, b, 10)$ and applying \textbf{TC2} to $11$ and the relation $((ab) ^ 2, a^2)$ leads to the definition of edge $(15, a, 11)$. Finally, an application of \textbf{TC2} to $15$ and the relation $(b^3, b)$ leads to the definition of the node $(14, b, 10)$; see \cref{subfig-felsch-17}.

		\item[Step 18:] The node $16$ and the edge $(11, b, 16)$ are defined. After an application of \textbf{TC2} to $5$ and $(b^3, b)$ leads to the definition of $(16, b, 11)$; see \cref{subfig-felsch-18}.

		\item[Step 19:] The node $17$ and the edge $(12, a, 17)$ are defined. An application of \textbf{TC2} to $6$ and $(b^3, b)$ leads to the definition of $(17, a, 12)$; see \cref{subfig-felsch-19}.

		\item[Step 20:] We apply \textbf{TC1} and the node $18$ and the edge $(12, b, 18)$ are defined; see \cref{subfig-felsch-20}.

		\item[Step 21:] We apply \textbf{TC1} and the node $19$ and the edge $(14, a, 19)$ are defined. Applying \textbf{TC2} to $5$ and the relation $((ab) ^ 2, a^2)$ leads to the definition of the edge $(19, b, 5)$ and applying  \textbf{TC2} to $14$ and the relation $((ab) ^ 2, a^2)$ leads to the definition of edge $(19, a, 14)$. We apply \textbf{TC2} to $19$ and $(b^3, b)$ and to node $19$ and $((ab) ^ 2, a^2)$ yielding the pairs $(5, 16)$ and $(11, 19)$ are added to $\kappa_{22}$; see \cref{subfig-felsch-21}.

		\item[Step 22:] Taking the quotient of $\Gamma_{22}$ by $\kappa_{22}$ gives the graph $\Gamma_{23}$. In $\Gamma_{23}$ there are the edges $(11, a, 14)$ and $(11, a, 15)$ with source $11$ and label $a$ and hence the pair $(14, 15)$ is added to $\kappa_{23}$; see \cref{subfig-felsch-22}.

		\item[Step 23:] Taking the quotient gives graph $\Gamma_{24}$; see \cref{subfig-felsch-23}.

		\item[Step 24:] We apply \textbf{TC1} so that the node $20$ and the edge $(17, b, 20)$ are defined; see \cref{subfig-felsch-24}.

		\item[Step 25:] We apply \textbf{TC1} and the node $21$ and the edge $(18, a, 21)$ are defined. We apply \textbf{TC2} to $6$ and $((ab) ^ 2, a^2)$ and define the edge $(21, b, 17)$. We also apply \textbf{TC2} to $18$ and $((ab) ^ 2, a^2)$ and define $(21, a, 18)$. Finally, we apply \textbf{TC2} to $6$ and $(b^3, b)$ and we define $(20, b, 17)$; see \cref{subfig-felsch-25}.

		\item[Step 26:] We apply \textbf{TC1} and define $22$ and $(18, b, 22)$. Next, we apply \textbf{TC2} to $12$ and $(b^3, b)$. This leads to the definition of $(22, b, 18)$; see \cref{subfig-felsch-26}.

		\item[Step 27:]  We apply \textbf{TC1} and define $23$ and $(20, a, 23)$. We apply \textbf{TC2} to $12$ and $((ab) ^ 2, a^2)$ and define $(23, b, 12)$. We also apply \textbf{TC2} to $20$ and $((ab) ^ 2, a^2)$ and define $(23, a, 20)$. Finally, we apply \textbf{TC2} to $23$ and $(b^3, b)$ and to $23$ and $((ab) ^ 2, a^2)$ yielding the pairs $(12, 22)$ and $(18, 23)$ in $\kappa_{28}$; see \cref{subfig-felsch-27}.

		\item[Step 28:]
			Taking the quotient of $\Gamma_{28}$ by $\kappa_{28}$ gives $\Gamma_{29}$. In $\Gamma_{29}$ there exist two edges $(18, a, 20)$ and $(18, a, 21)$ labelled $a$ leaving $b^2ab$ and hence the pair $(20, 21)$ is added to $\kappa_{29}$; see \cref{subfig-felsch-28}.

		\item[Step 29:] Taking the quotient of $\Gamma_{29}$ by $\kappa_{29}$ gives the graph $\Gamma_{30}$; see \cref{subfig-felsch-29}.

			After Step 29, $\Gamma_{30}$ is complete, deterministic, and compatible with $R$. Hence the enumeration terminates.
	\end{description}

	\begin{table}\centering
		\renewcommand{\arraystretch}{1.4}
		\begin{tabular}{c|c|c|c|c|c|c|c|c|c|c|c}
			0             & 1             & 2              & 3          & 4               & 5           & 6  & 7 & 8 & 9 & 10
			              & 11                                                                                                \\ \hline
			$\varepsilon$ & $a$           & $b$            & $a^2$      &
			$ab$          & $ba$          & $b ^ 2$        & $a ^ 2b$   & $aba$           & $ab ^ 2$    &
			$b
			a ^ 2$        & $bab$                                                                                             \\\hline\hline
			12            & 13            & 14             & 15         & 16
			              & 17            & 18             & 19         & 20              & 21          & 22
			              & 23                                                                                                \\ \hline
			$b^2a$        & $a ^ 2ba$     & $b a ^ 2 b$    & $(ba) ^ 2$ & $bab ^
			2$            & $b ^ 2 a ^ 2$ & $b ^ 2 ab$     & $ba ^ 2ba$ & $b ^ 2 a ^ 2 b$ & $b ^ 2 aba$
			              & $b ^ 2ab ^ 2$ & $b ^ 2 a^ 2ba$
		\end{tabular}
		\caption{A word labelling a path from $0$ to each node in
			\cref{ex-felsch-A1}; see \cref{fig-ex-felsch-A1-a}, \cref{fig-ex-felsch-A1-b}, \cref{fig-ex-felsch-A1-c},
			\cref{fig-ex-felsch-A1-d}, and \cref{fig-ex-felsch-A1-e}.}
		\label{table-ex-felsch-A1}
	\end{table}
	\begin{figure}
		\begin{subfigure}{0.24\textwidth}
			\centering
			
    \includegraphics{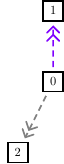}

			\caption{Steps 1 \& 2.}
			\label{subfig-felsch-2}
		\end{subfigure}
		\begin{subfigure}{0.24\textwidth}
			\centering
			
    \includegraphics{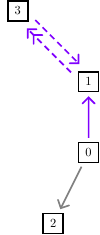}

			\caption{Step 3.}
			\label{subfig-felsch-3}
		\end{subfigure}
		\begin{subfigure}{0.24\textwidth}
			\centering
			
    \includegraphics{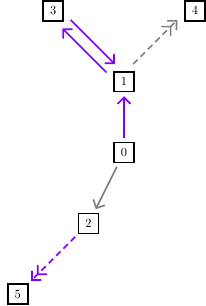}

			\caption{Step 4 \& 5.}
			\label{subfig-felsch-5}
		\end{subfigure}
		\begin{subfigure}{0.24\textwidth}
			\centering
			
    \includegraphics{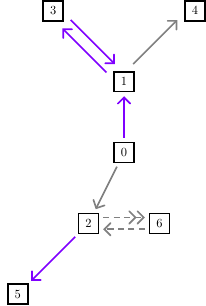}

			\caption{Step 6.}
			\label{subfig-felsch-6}
		\end{subfigure}
		\vspace{\baselineskip}

		\begin{subfigure}{0.33\textwidth}
			\centering
			
    \includegraphics{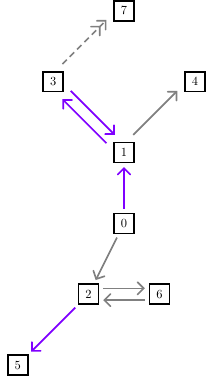}

			\caption{Step 7.}
			\label{subfig-felsch-7}
		\end{subfigure}
		\begin{subfigure}{0.33\textwidth}
			\centering
			
    \includegraphics{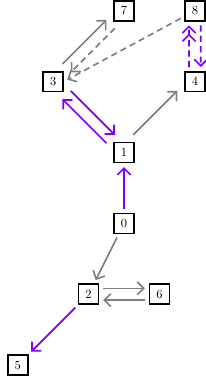}

			\caption{Step 8.}
			\label{subfig-felsch-8}
		\end{subfigure}
		\begin{subfigure}{0.33\textwidth}
			\centering
			
    \includegraphics{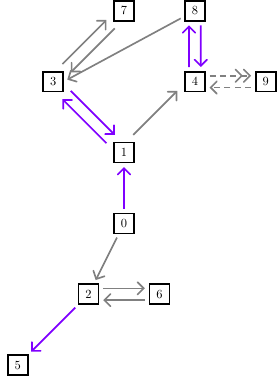}

			\caption{Step 9.}
			\label{subfig-felsch-9}
		\end{subfigure}
		\vspace{\baselineskip}

		\begin{subfigure}{0.33\textwidth}
			\centering
			
    \includegraphics{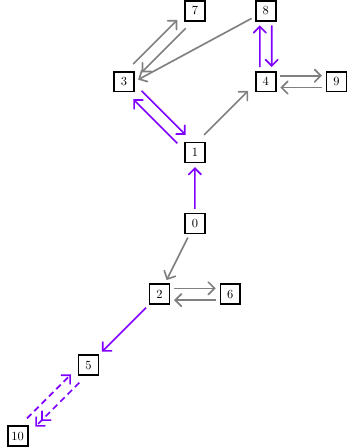}

			\caption{Step 10.}
			\label{subfig-felsch-10}
		\end{subfigure}
		\begin{subfigure}{0.33\textwidth}
			\centering
			
    \includegraphics{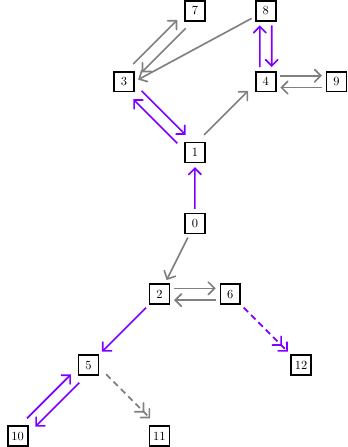}

			\caption{Steps 11 \& 12.}
			\label{subfig-felsch-12}
		\end{subfigure}
		\begin{subfigure}{0.33\textwidth}
			\centering
			
    \includegraphics{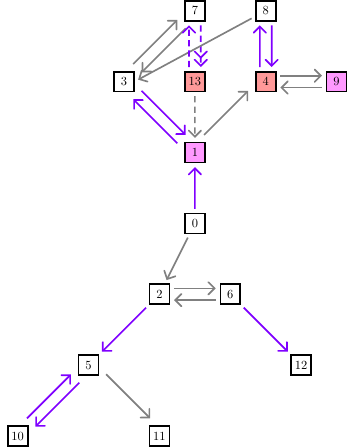}

			\caption{Step 13.}
			\label{subfig-felsch-13}
		\end{subfigure}
		\caption{The output $(\Gamma_{i}, \kappa_{i})$ for $i = \{3, 4, 6, 7, 8, 9, 10, 11, 13, 14\}$ in the Felsch
			Strategy in \cref{ex-felsch-1}. Purple arrows correspond to $a$, gray to $b$, shaded nodes of the same colour belong to $\kappa_{i}$, and unshaded nodes belong to singleton classes.
			A dashed edge with a double arrowhead indicates the edge being defined in \textbf{TC1}, a dashed edge with a single arrowhead denotes a new edge obtained from \textbf{TC2} or \textbf{TC3}, solid edges correspond to edges that existed at the previous step.
		}
		\label{fig-ex-felsch-A1-a}
	\end{figure}
	\begin{figure}
		\begin{subfigure}{0.33\textwidth}
			\centering
			
    \includegraphics{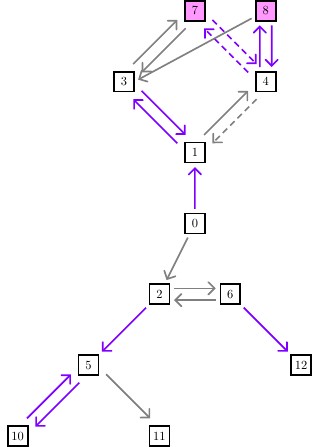}

			\caption{Step 14.}
			\label{subfig-felsch-14}
		\end{subfigure}
		\begin{subfigure}{0.33\textwidth}
			\centering
			
    \includegraphics{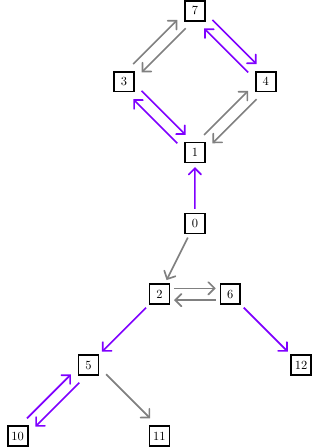}

			\caption{Step 15.}
			\label{subfig-felsch-15}
		\end{subfigure}
		\begin{subfigure}{0.33\textwidth}
			\centering
			
    \includegraphics{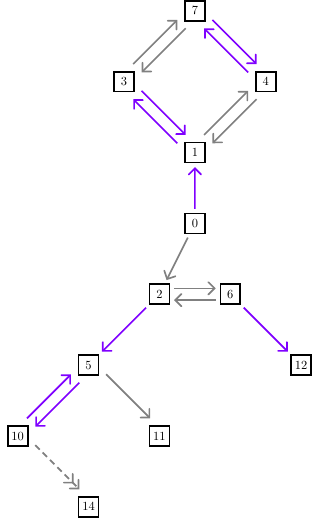}

			\caption{Step 16.}
			\label{subfig-felsch-16}
		\end{subfigure}
		\vspace{\baselineskip}

		\begin{subfigure}{0.33\textwidth}
			\centering
			
    \includegraphics{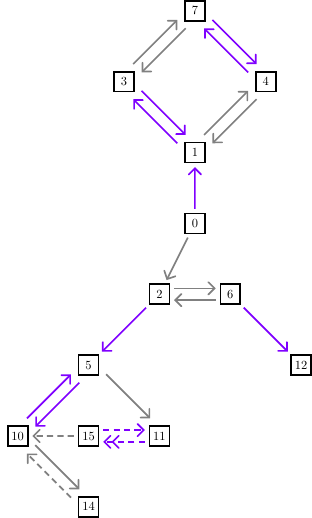}

			\caption{Step 17.}
			\label{subfig-felsch-17}
		\end{subfigure}
		\begin{subfigure}{0.33\textwidth}
			\centering
			
    \includegraphics{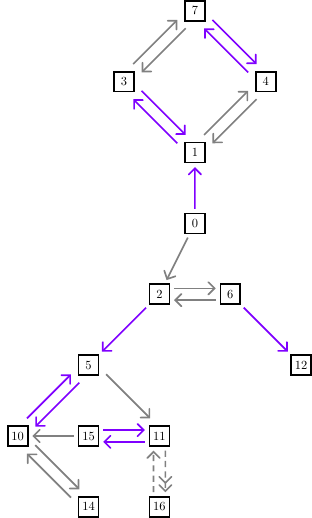}

			\caption{Step 18.}
			\label{subfig-felsch-18}
		\end{subfigure}
		\begin{subfigure}{0.33\textwidth}
			\centering
			
    \includegraphics{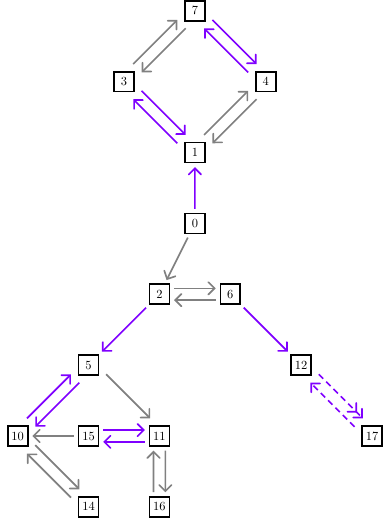}

			\caption{Step 19.}
			\label{subfig-felsch-19}
		\end{subfigure}
		\caption{The output $(\Gamma_{i}, \kappa_{i})$  for $i \in \{15, \ldots, 20\}$ in the Felsch
			strategy in \cref{ex-felsch-1}. Purple arrows correspond to $a$, gray to $b$, shaded nodes of the same colour belong to $\kappa_{i}$, and unshaded nodes belong to singleton classes.
			A dashed edge with a double arrowhead indicates the edge being defined in \textbf{F1}, a dashed edge with a single arrowhead denotes an edge that is obtained from \textbf{F2}, solid edges correspond to edges that existed at the previous step.
		}
		\label{fig-ex-felsch-A1-b}
	\end{figure}

	\begin{figure}
		\begin{subfigure}{0.49\textwidth}
			\centering
			
    \includegraphics{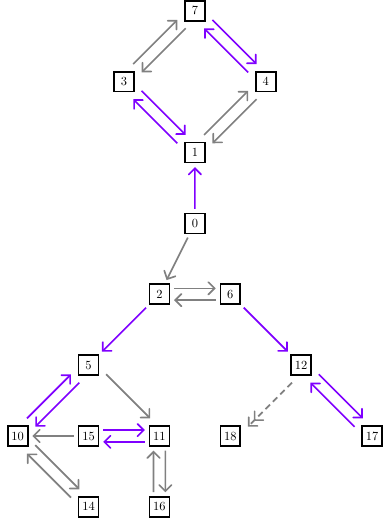}

			\caption{Step 20.}
			\label{subfig-felsch-20}
		\end{subfigure}
		\begin{subfigure}{0.49\textwidth}
			\centering
			
    \includegraphics{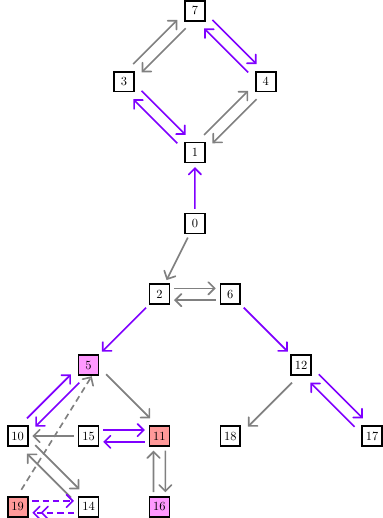}

			\caption{Step 21.}
			\label{subfig-felsch-21}
		\end{subfigure}
		\vspace{\baselineskip}

		\begin{subfigure}{0.49\textwidth}
			\centering
			
    \includegraphics{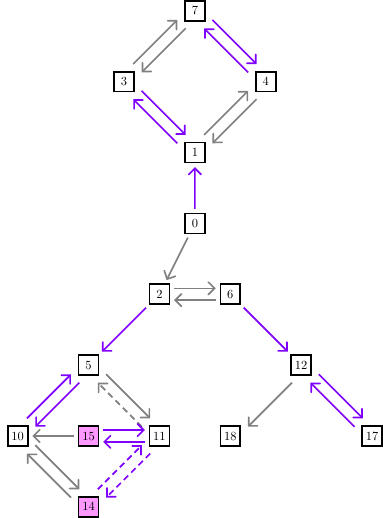}

			\caption{Step 22.}
			\label{subfig-felsch-22}
		\end{subfigure}
		\begin{subfigure}{0.49\textwidth}
			\centering
			
    \includegraphics{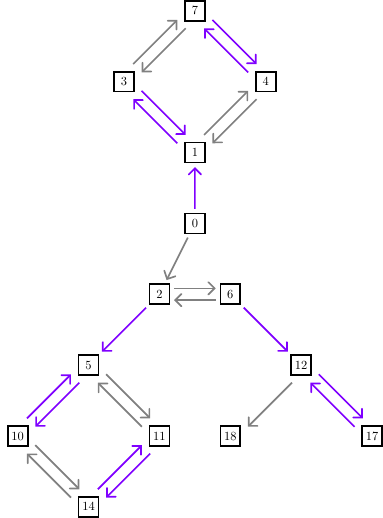}

			\caption{Step 23.}
			\label{subfig-felsch-23}
		\end{subfigure}
		\caption{The output $(\Gamma_{i}, \kappa_{i})$ for $i \in \{21, 22, 23, 24\}$ in the Felsch
			strategy in \cref{ex-felsch-1}. Purple arrows correspond to $a$, gray to $b$, shaded nodes of the same colour belong to $\kappa_{i}$, and unshaded nodes belong to singleton classes.
			A dashed edge with a double arrowhead indicates the edge being defined in \textbf{TC1}, a dashed edge with a single arrowhead denotes an edge that is obtained from \textbf{TC2} or \textbf{TC3}, solid edges correspond to edges that existed at the previous step.
		}
		\label{fig-ex-felsch-A1-c}
	\end{figure}

	\begin{figure}
		\begin{subfigure}{0.49\textwidth}
			\centering
			
    \includegraphics{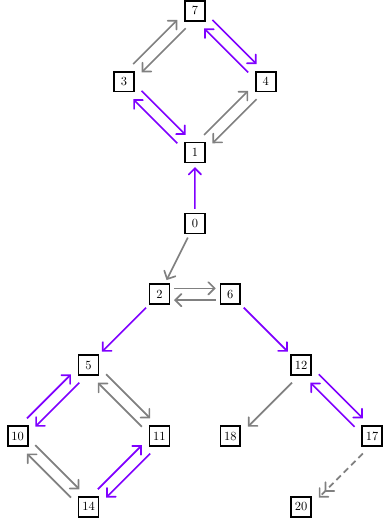}

			\caption{Step 24.}
			\label{subfig-felsch-24}
		\end{subfigure}
		\begin{subfigure}{0.49\textwidth}
			\centering
			
    \includegraphics{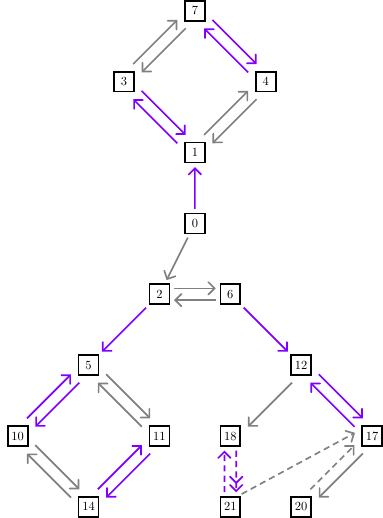}

			\caption{Step 25.}
			\label{subfig-felsch-25}
		\end{subfigure}
		\vspace{\baselineskip}

		\begin{subfigure}{0.49\textwidth}
			\centering
			
    \includegraphics{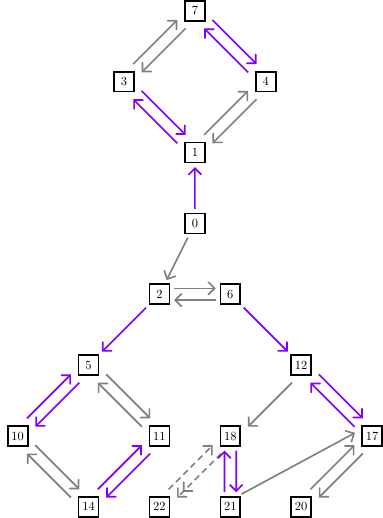}

			\caption{Step 26.}
			\label{subfig-felsch-26}
		\end{subfigure}
		\begin{subfigure}{0.49\textwidth}
			\centering
			
    \includegraphics{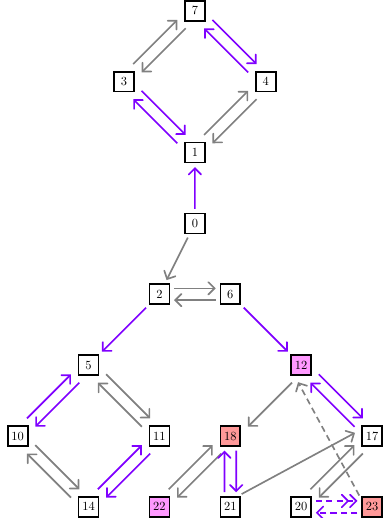}

			\caption{Step 27.}
			\label{subfig-felsch-27}
		\end{subfigure}
		\caption{The output $(\Gamma_{i}, \kappa_{i})$ for $i \in \{25, \ldots, 28\}$ in the Felsch
			strategy for \cref{ex-felsch-1}. Purple arrows correspond to $a$, gray to $b$, shaded nodes of the same colour belong to $\kappa_{i}$, and unshaded nodes belong to singleton classes.
			A dashed edge with a double arrowhead indicates the edge being defined in \textbf{TC1}, a dashed edge with a single arrowhead denotes an edge that is obtained from \textbf{TC2} or \textbf{TC3}, solid edges correspond to edges that existed at the previous step.
		}
		\label{fig-ex-felsch-A1-d}
	\end{figure}
	\begin{figure}
		\begin{subfigure}{0.49\textwidth}
			\centering
			
    \includegraphics{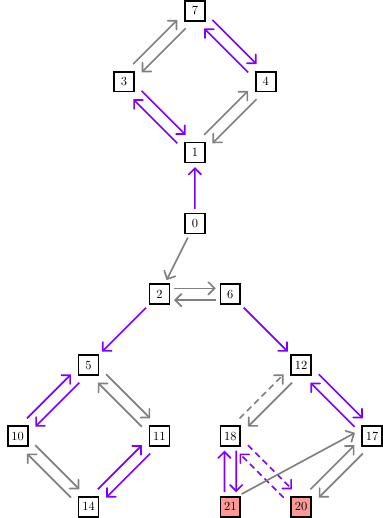}

			\caption{Step 28.}
			\label{subfig-felsch-28}
		\end{subfigure}
		\begin{subfigure}{0.49\textwidth}
			\centering
			
    \includegraphics{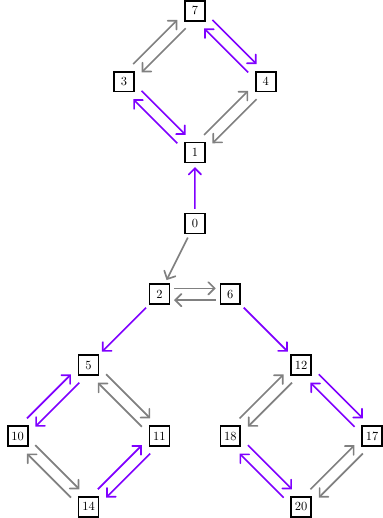}

			\caption{Step 29.}
			\label{subfig-felsch-29}
		\end{subfigure}
		\caption{The output $(\Gamma_{i}, \kappa_{i})$ for $i = 29$ and $30$ in the Felsch
			strategy of \cref{ex-felsch-1}. Purple arrows correspond to $a$, gray to $b$, shaded nodes of the same colour belong to $\kappa_{i}$, and unshaded nodes belong to singleton classes.
			A dashed edge with a double arrowhead indicates the edge being defined in \textbf{TC1}, a dashed edge with a single arrowhead denotes an edge that is obtained from \textbf{TC2} or \textbf{TC3}, solid edges correspond to edges that existed at the previous step.}
		\label{fig-ex-felsch-A1-e}
	\end{figure}

\end{ex}

\begin{ex}
	\label{ex-bool-mat-presentation-hlt}
	In this example, we perform the HLT strategy to find the right Cayley graph of the monoid defined by the presentation:
	\[
		\langle a,b,c \; | \; a^2 = ac,\  b^2= b,\  a^ 2 = ca,\  bc = cb,\ a^2=
		c^2,\  a^3= a^2,\  a^2 = aba \rangle
	\]
with $a<b<c$.
	Each step of this enumeration corresponds to either: at least one application
	of \textbf{TC1} (in \textbf{HLT1}) followed by multiple applications of
	\textbf{TC2} (in the same \textbf{HLT1} step as the application of
	\textbf{TC1} or subsequent applications of \textbf{HLT1} where \textbf{TC1}
	is not invoked); or a single application of \textbf{TC3} (in \textbf{HLT2});
	see \cref{fig-ex-bool-mat-presentation-hlt-1} and
	\cref{fig-ex-bool-mat-presentation-hlt-2} for the word graphs $\Gamma_i$
	and equivalence relation $\kappa_i$ after every step $i$; see
	also~\cref{table-ex-bool-mat-presentation-hlt}.

	\begin{description}
		\item [Step 1:] At this step we apply \textbf{HLT1} to the node
		      $0$ and relation $(a ^ 2, ac)$. \textbf{TC1} yields the new
		      nodes $1$ and $2$, and the edges $(0, a, 1)$ and $(1, a, 2)$;
		      \textbf{TC2}(b) then yields the edge $(1, c, 2)$; see
		      \cref{subfig-bool-mat-hlt-1}.

		\item [Step 2:] We apply \textbf{HLT1} to the node
		      $0$ and relation $(b ^ 2, b)$. \textbf{TC1} produces the new
		      node $3$ and edge $(0, b, 3)$; \textbf{TC2}(a) then yields the
		      edge $(3, b, 3)$; see \cref{subfig-bool-mat-hlt-2}.

		\item [Step 3:] \textbf{HLT1} is applied to $0$ and $(a ^ 2, ca)$;
		      \textbf{TC1} generates the node $4$ and edge $(0, c, 4)$; \textbf{TC2}(b)
		      gives us the edge $(4, a, 2)$; see \cref{subfig-bool-mat-hlt-3}.

		\item [Steps 4-6:] Step 4 is \textbf{HLT1} applied to $0$ and $(bc,
			      cb)$; \textbf{TC1} gives the new node $5$ and edge $(3, c, 5)$;
		      \textbf{TC2}(b) yields the edge $(4, b, 5)$. Step 5 is \textbf{HLT1}
		      applied to $0$ and $(a ^ 2, c ^ 2)$. There is already a path
		      from $0$ labelled by $a ^ 2$ and another by $c$, and so
		      \textbf{TC1} is not invoked, but \textbf{TC2}(b) yields the edge $(4, c,
			      2)$. Step 6: \textbf{HLT1} applied to $0$ and $(a ^ 3, a ^
			      2)$, \textbf{TC1} is not invoked again, and \textbf{TC2}(a) yields the
		      edge $(2, a, 2)$; see \cref{subfig-bool-mat-hlt-6}.

		\item [Steps 7-9:] Step 7: \textbf{HLT1} applied to $0$ and $(a ^ 2,
			      aba)$; \textbf{TC1} adds the node $6$ and edge $(1, b, 6)$;
		      \textbf{TC2}(b) yields $(6, a, 2)$. Step 8:  \textbf{HLT1} applied
		      to $1$ and $(a ^ 2, ac)$; \textbf{TC1} does not apply;
		      \textbf{TC2}(b) yields $(2, c, 2)$. Step 9: \textbf{HLT1} applied
		      to $1$ and $(b ^ 2, b)$;  \textbf{TC1} does not apply;
		      \textbf{TC2}(a) yields $(6, b, 6)$.
		      See \cref{subfig-bool-mat-hlt-9}.

		\item [Steps 10-13:] Step 10: \textbf{HLT1} applied to $1$ and $(bc, cb)$;
		      \textbf{TC1} adds node $7$ and edge $(6, c, 7)$. Step 11:
		      \textbf{HLT1} applied to $1$ and $(a ^ 2, aba)$; \textbf{TC1} does not
		      apply; \textbf{TC2}(b) yields edge $(7, a, 2)$.
		      Step 12: \textbf{HLT1} applied to $2$ and $(b ^ 2, b)$; \textbf{TC1}
		      does not apply; \textbf{TC2}(a) yields edge $(7, b, 7)$. Step 13:
		      \textbf{HLT1} applied to $2$ and $(bc, cb)$; \textbf{TC1} does not apply;
		      \textbf{TC2}(a) yields edge $(7, c, 7)$. See
		      \cref{subfig-bool-mat-hlt-13}.

		\item [Steps 14-18:] Step 14: \textbf{HLT1} applied to $3$ and $(a ^ 2,
			      ac)$; \textbf{TC1} yields nodes $8$ and $9$ and edges $(3, a, 8)$,
		      $(8, a, 9)$; \textbf{TC2}(b) yields edge $(8, c, 9)$.
		      Step 15: \textbf{HLT1} for $3$ and $(a ^ 2, ca)$; \textbf{TC1} not
		      applied; \textbf{TC2}(b) yields edge $(5, a, 9)$.
		      Step 16: \textbf{HLT1} for $b$ and $(bc, cb)$; \textbf{TC1} does
		      not apply; \textbf{TC2}(b) yields $(5, b, 5)$.
		      Step 17: \textbf{HLT1} for $b$ and $(a ^ 2, c ^ 2)$;
		      \textbf{TC1} does not apply; \textbf{TC2}(b) yields $(5, c, 9)$.
		      Step 18: \textbf{HLT1} for $3$ and $(a ^ 3, a ^ 2)$;
		      \textbf{TC1} does not apply; \textbf{TC2}(a) yields $(9, a, 9)$.
		      See \cref{subfig-bool-mat-hlt-18}.

		\item [Steps 19-20:]
		      Step 19: \textbf{HLT1} for $3$ and $(a ^ 2, aba)$: \textbf{TC1}
		      adds node $10$ and edge $(8, b, 10)$; \textbf{TC2}(b) yields
		      $(10, a, 9)$. Step 20: \textbf{HLT1} for $4$ and $(bc, cb)$:
		      \textbf{TC1} does not apply; \textbf{TC2}(c) indicates $\kappa_{20}$ is
		      the  least equivalence containing $(7, 9)$. See
		      \cref{subfig-bool-mat-hlt-20}.

		\item [Step 21:] This step is an application of \textbf{TC3} to
		      produce $\Gamma_{21} := \Gamma_{20}/ \kappa_{20}$ resulting in
		      the new edges $(7, a, 7)$, $(5, a, 7)$, $(5, c, 7)$, $(8, a,
			      7)$, $(8, c, 7)$, and $(10, b, 7)$. The edge
		      $(7, a, 2)$ is also an edge in $\Gamma_{21}$ meaning that
		      $\kappa_{21}$ is the least equivalence containing $(2, 7)$.
		      See \cref{subfig-bool-mat-hlt-21}.

		\item [Step 22:] This step is also an application of \textbf{TC3} to
		      produce $\Gamma_{22} := \Gamma_{21} / \kappa_{21}$, the new edges
		      created are $(2, b, 2)$, $(10, a, 2)$, $(8, a, 2)$,
		      $(8, c, 2)$, $(5, a, 2)$, and $(5, c, 2)$.
		      See \cref{subfig-bool-mat-hlt-22}.

		\item [Steps 23-24:]
		      Step 23: \textbf{HLT1} applied to $(b ^ 2, b)$ from $8$; \textbf{TC1}
		      does not apply; \textbf{TC2}(a) yields $(10, b, 10)$. Step 24:
		      \textbf{HLT1} applied to $8$ and $(bc, cb)$ yields $(10, c, 2)$.
		      See \cref{subfig-bool-mat-hlt-24}.
	\end{description}
	After Step 24 the graph $\Gamma_{25}$ is complete and compatible with the
	relations in the presentation, and $\kappa_{25}$ is trivial, the enumeration
	terminates, and we see that the semigroup defined by the presentation is
	isomorphic to that defined in \cref{ex-cayley-digraph}.

	\begin{table}\centering
		\begin{tabular}{l|l|l|l|l|l|l|l|l|l|l}
			0             & 1   & 2       & 3   & 4   & 5    & 6    & 7     & 8    & 9      & 10    \\ \hline
			$\varepsilon$ & $a$ & $a ^ 2$ & $b$ & $c$ & $bc$ & $ab$ & $abc$ & $ba$ & $ba^2$ & $bab$
		\end{tabular}
		\caption{A word labelling a path from $0$ to each node in the right Cayley
			graph on the monoid $M$ from \cref{ex-bool-mat-presentation-hlt}; see
			\cref{fig-ex-bool-mat-presentation-hlt-1} and
			\cref{fig-ex-bool-mat-presentation-hlt-2}.}
		\label{table-ex-bool-mat-presentation-hlt}
	\end{table}

	\begin{figure}
		\begin{subfigure}{0.33\textwidth}
			\centering
			
    \includegraphics{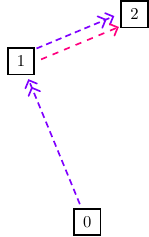}

			\caption{Step 1.}
			\label{subfig-bool-mat-hlt-1}
		\end{subfigure}
		\begin{subfigure}{0.33\textwidth}
			\centering
			
    \includegraphics{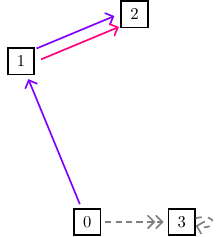}

			\caption{Step 2.}
			\label{subfig-bool-mat-hlt-2}
		\end{subfigure}
		\begin{subfigure}{0.33\textwidth}
			\centering
			
    \includegraphics{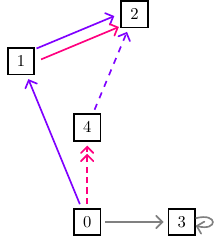}

			\caption{Step 3.}
			\label{subfig-bool-mat-hlt-3}
		\end{subfigure}
		\vspace{\baselineskip}

		\begin{subfigure}{0.33\textwidth}
			\centering
			
    \includegraphics{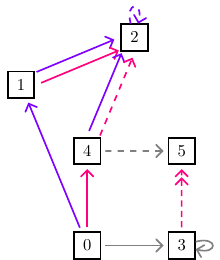}

			\caption{Steps 4 to 6.}
			\label{subfig-bool-mat-hlt-6}
		\end{subfigure}
		\begin{subfigure}{0.33\textwidth}
			\centering
			
    \includegraphics{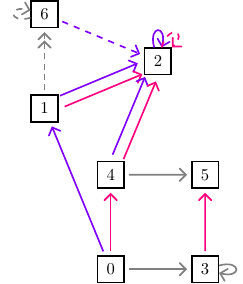}

			\caption{Steps 7 to 9.}
			\label{subfig-bool-mat-hlt-9}
		\end{subfigure}
		\begin{subfigure}{0.33\textwidth}
			\centering
			
    \includegraphics{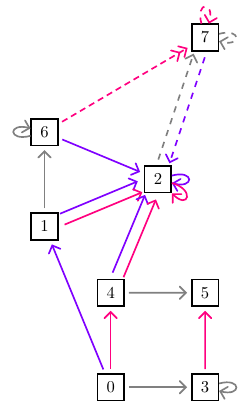}

			\caption{Steps 10 to 13.}
			\label{subfig-bool-mat-hlt-13}
		\end{subfigure}

		\begin{subfigure}{0.49\textwidth}
			\centering
			
    \includegraphics{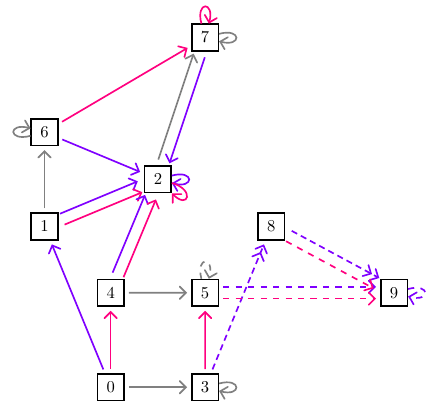}

			\caption{Steps 14 to 18.}
			\label{subfig-bool-mat-hlt-18}
		\end{subfigure}
		\begin{subfigure}{0.49\textwidth}
			\centering
			
    \includegraphics{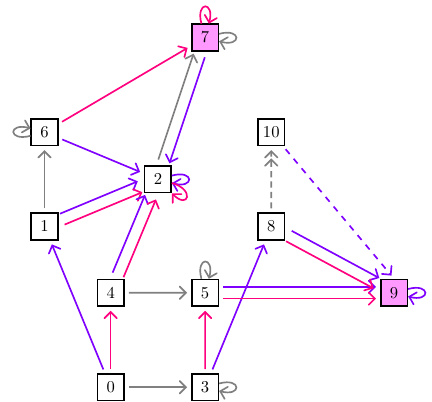}

			\caption{Steps 19 \& 20.}
			\label{subfig-bool-mat-hlt-20}
		\end{subfigure}
		\caption{The output $(\Gamma_{i}, \kappa_{i})$ for $i = 1, \ldots,  20$ in
			\cref{ex-bool-mat-presentation-hlt}. Purple arrows correspond to $a$,
			gray to $b$, pink to $c$, shaded nodes of the same colour belong to
			$\kappa_{i}$, and unshaded nodes belong to singleton classes. A dashed
			edge with a double arrowhead indicates the edge being defined in
			\textbf{TC1}, a dashed edge with a single arrowhead denotes an edge that
			is obtained from \textbf{TC2} or \textbf{TC3}, solid edges correspond to
			edges that existed at the previous step.}
		\label{fig-ex-bool-mat-presentation-hlt-1}
	\end{figure}

	\begin{figure}
		\begin{subfigure}{0.49\textwidth}
			\centering
			
    \includegraphics{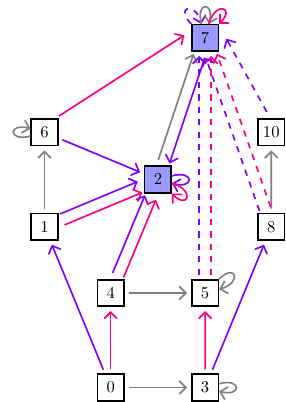}

			\caption{Step 21.}
			\label{subfig-bool-mat-hlt-21}
		\end{subfigure}
		\begin{subfigure}{0.49\textwidth}
			\centering
			
    \includegraphics{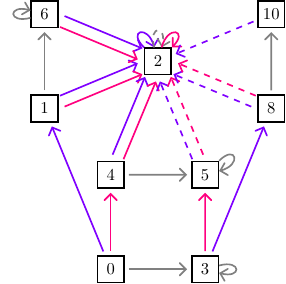}

			\caption{Step 22.}
			\label{subfig-bool-mat-hlt-22}
		\end{subfigure}
		\vspace{\baselineskip}

		\begin{subfigure}{0.98\textwidth}
			\centering
			
    \includegraphics{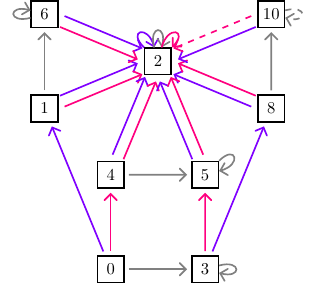}

			\caption{Steps 23 \& 24.}
			\label{subfig-bool-mat-hlt-24}
		\end{subfigure}
		\caption{The output $(\Gamma_{i}, \kappa_{i})$ for $i = 21, \ldots, 24$ of each step in
			\cref{ex-bool-mat-presentation-hlt}. Purple arrows correspond to $1$,
			gray to $3$, pink to $4$, shaded nodes of the same colour belong to
			$\kappa_{i}$, and unshaded nodes belong to singleton classes.          A
			dashed edge with a double arrowhead indicates the edge being defined in
			\textbf{TC1}, a dashed edge with a single arrowhead denotes an edge that is
			obtained from \textbf{TC2} or \textbf{TC3}, solid edges correspond to edges
			that existed at the previous step.}
		\label{fig-ex-bool-mat-presentation-hlt-2}
	\end{figure}
\end{ex}

\begin{ex}
	\label{ex-bool-mat-presentation}
	In this example, we perform the Felsch strategy to find the right Cayley graph of the monoid defined by the presentation:
	\[\langle a,b,c \; | \;
		ac = a^2,\  b^2= b,\  ca= a^2,\  cb =bc,\
		c^2=  a^2,\  a^3= a^2,\  aba= a^2
		\rangle.
	\]

	\begin{description}
		\item [Step 1:] At this step we apply \textbf{F1} to the node
		      $0$. The node $1$ and the edge $(0, a, 1)$ are defined; see
		      \cref{subfig-bool-mat-felsch-1}.

		\item [Step 2:] We apply \textbf{F1} to the node
		      $0$. The node $2$ and the edge $(0, b, 2)$ are defined. Applying \textbf{F2} to $0$ and the relation $(b^2, b)$ leads to the definition of the edge $(2, b, 2)$; see
		      \cref{subfig-bool-mat-felsch-2}.

		\item [Step 3:] We apply \textbf{F1} to $0$. The node $3$ and the edge $(0, c, 3)$ are defined; see \cref{subfig-bool-mat-felsch-3}.

		\item [Step 4:] \textbf{F1} is applied to $1$ and hence the node $4$ and the edge $(1, a, 4)$ are defined. We apply \textbf{F2} and the following applications of \textbf{TC2} yield new information; applying \textbf{TC2}(b) to $0$ and $(a^2, ac)$ yields the edge $(1, c, 4)$, applying \textbf{TC2}(b) to $0$ and $(a^2, ca)$ yields the edge $(3, a, 4)$, the application of \textbf{TC2}(b) to $0$ and $(a^2, c^2)$ yields the edge $(3, c, 4)$, the application of \textbf{TC2}(a) to $0$ and $(a^3, a^2)$ yields the edge $(4, a, 4)$ and finally the application of \textbf{TC2}(b) to $1$ and $(a^2, ac)$ yields the edge $(4, c, 4)$; see \cref{subfig-bool-mat-felsch-4}.

		\item [Step 5:] We apply \textbf{F1} to $1$ which leads to the definition of the node $5$ and the edge $(1, b, 5)$. We apply \textbf{F2} and the following applications of \textbf{TC2} yield new information; the application of \textbf{TC2}(b) to $0$ and $(a^2, aba)$ yields the edge $(5, a, 4)$ and the application of \textbf{TC2}(a) to $1$ and $(b^2, b)$ yields the edge $(5, b, 5)$; see \cref{subfig-bool-mat-felsch-5}.

		\item [Step 6:] We apply \textbf{F1} to $2$ which leads to the definition of the node $6$ and the edge $(2, a, 6)$; see \cref{subfig-bool-mat-felsch-6}.

		\item [Step 7:] We apply \textbf{F1} to $2$ which leads to the definition of the node $7$ and the edge $(2, c, 7)$. We apply \textbf{F2} and the following applications of \textbf{TC2} yield new information; applying \textbf{TC2}(b) to $0$ and $(bc, cb)$ yields the edge $(3, b, 7)$ and applying \textbf{TC2}(b) to $3$ and $(bc, cb)$ yields the edge $(7, b, 7)$; see \cref{subfig-bool-mat-felsch-7}.

		\item [Step 8:] We apply \textbf{F1} and the node $8$ and the edge $(4, b, 8)$ are defined. We apply \textbf{F2} and the following applications of \textbf{TC2} yield new information; applying \textbf{TC2}(a) to $1$ and $(bc, cb)$ yields the edge $(5, c, 8)$, applying \textbf{TC2}(b) to $1$ and $(a^2, aba)$ yields the edge $(8, a, 4)$, applying \textbf{TC2}(a) to $3$ and $(bc, cb)$ yields the edge $(7, c, 8)$, applying \textbf{TC2}(a) to $4$ and $(b^2, b)$  yields the edge $(8, b, 8)$ and applying \textbf{TC2}(a) to $4$ and $(bc, cb)$ yields the edge $(8, c, 8)$. In addition, applying \textbf{TC2}(c) to $5$ and $(a^2, c^2)$ as well as applying \textbf{TC2}(c) to $8$ and $(a^2, c^2)$ indicates $\kappa_{9}$ is the least equivalence containing $(8, 4)$ and $\Delta_{N_9}$; see \cref{subfig-bool-mat-felsch-8}.

		\item [Step 9:] This step is an application of \textbf{TC3} to
		      produce $\Gamma_{10} := \Gamma_{9}/ \kappa_{9}$ resulting in
		      the new edges $(5, c, 4), (4, b, 4)$; see \cref{subfig-bool-mat-felsch-9}.

		\item [Step 10:] We apply \textbf{F1}. The node $9$ and the edge $(6, a, 9)$ are defined and an application of \textbf{F2} follows. Applying \textbf{TC2}(b) to $2$ and $(a^2, ac)$ yields the edge $(6, c, 9)$, applying \textbf{TC2}(b) to $2$ and $(a^2, ca)$ yields the edge $(7, a, 9)$, applying \textbf{TC2}(a) to $2$ and $(a^3, a^2)$  yields the edge $(9, a, 9)$ and applying \textbf{TC2}(b) to $6$ and $(a^2, ac)$ yields the edge $(9, c, 9)$. In addition, applying \textbf{TC2}(c) to $2$ and $(a^2, c^2)$ indicates $\kappa_{11}$ is the least equivalence containing $(4, 9)$ and $\Delta_{N_{11}}$; see \cref{subfig-bool-mat-felsch-10}.

		\item [Step 11:] This step is an application of \textbf{TC3} to
		      produce $\Gamma_{12} := \Gamma_{11}/ \kappa_{11}$ resulting in
		      the new edges $(7, a, 4), (6, a, 4), (6, c, 4) $; see \cref{subfig-bool-mat-felsch-11}.

		\item [Step 12:] We apply \textbf{F1} and the node $10$ and the edge $(6, b, 10)$ are defined. We apply \textbf{F2} and the following applications of \textbf{TC2} yield new information; applying \textbf{TC2}(b) to $2$ and $(a^2, aba)$ yields the edge $(10, a, 4)$, applying \textbf{TC2}(a) to $6$ and $(b^2, b)$ yields the edge $(10, b, 10)$ and applying \textbf{TC2}(a) to $6$ and $(bc, cb)$ yields the edge $(10, c, 4)$; see \cref{subfig-bool-mat-felsch-12}.
	\end{description}

	After Step 12 the graph $\Gamma_{13}$ is complete and compatible with the
	relations in the presentation, and $\kappa_{13}$ is trivial, the enumeration
	terminates, and we see that the semigroup defined by the presentation is
	isomorphic to that defined in \cref{ex-cayley-digraph}.

	\begin{table}\centering
		\begin{tabular}{l|l|l|l|l|l|l|l|l|l|l}
			0             & 1   & 2   & 3   & 4     & 5    & 6    & 7    & 8      & 9      & 10    \\ \hline
			$\varepsilon$ & $a$ & $b$ & $c$ & $a^2$ & $ab$ & $ba$ & $bc$ & $a^2b$ & $ba^2$ & $bab$
		\end{tabular}
		\caption{A word labelling a path from $0$ to each node in the right Cayley graph on the monoid $M$ from \cref{ex-bool-mat-presentation}; see \cref{fig-ex-bool-mat-presentation-felsch-A} and \cref{fig-ex-bool-mat-presentation-felsch-B}.}
		\label{table-ex-bool-mat-presentation}
	\end{table}

	\begin{figure}
		\centering
		\begin{subfigure}{0.30\textwidth}
			\centering
			
    \includegraphics{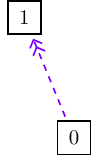}

			\caption{Step 1.}
			\label{subfig-bool-mat-felsch-1}
		\end{subfigure}
		\begin{subfigure}{0.30\textwidth}
			\centering
			
    \includegraphics{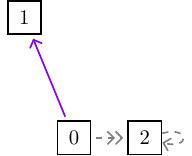}

			\caption{Step 2.}
			\label{subfig-bool-mat-felsch-2}
		\end{subfigure}
		\begin{subfigure}{0.30\textwidth}
			\centering
			
    \includegraphics{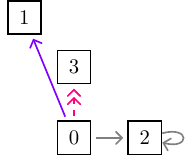}

			\caption{Step 3.}
			\label{subfig-bool-mat-felsch-3}
		\end{subfigure}
		\vspace{\baselineskip}

		\begin{subfigure}{0.30\textwidth}
			\centering
			
    \includegraphics{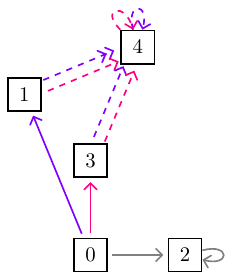}

			\caption{Step 4.}
			\label{subfig-bool-mat-felsch-4}
		\end{subfigure}
		\begin{subfigure}{0.30\textwidth}
			\centering
			
    \includegraphics{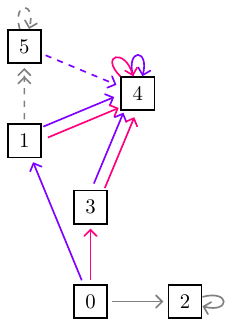}

			\caption{Step 5.}
			\label{subfig-bool-mat-felsch-5}
		\end{subfigure}
		\begin{subfigure}{0.30\textwidth}
			\centering
			
    \includegraphics{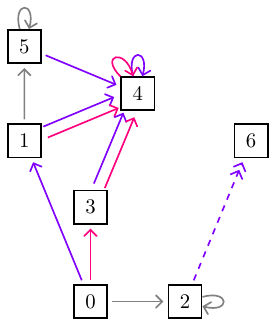}

			\caption{Step 6.}
			\label{subfig-bool-mat-felsch-6}
		\end{subfigure}
		\caption{The output $(\Gamma_{i}, \kappa_{i})$ of each step in
			\cref{ex-bool-mat-presentation}. Purple arrows correspond to $a$, grey to
			$b$, pink to $c$, shaded nodes of the same colour belong to $\kappa_{i}$,
			and unshaded nodes belong to singleton classes.}
		\label{fig-ex-bool-mat-presentation-felsch-A}
	\end{figure}

	\begin{figure}
		\centering
		\begin{subfigure}{0.30\textwidth}
			\centering
			
    \includegraphics{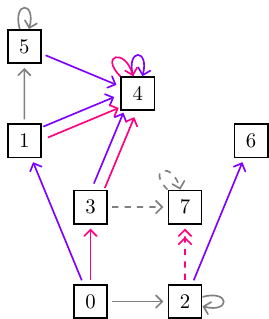}

			\caption{Step 7.}
			\label{subfig-bool-mat-felsch-7}
		\end{subfigure}
		\begin{subfigure}{0.30\textwidth}
			\centering
			
    \includegraphics{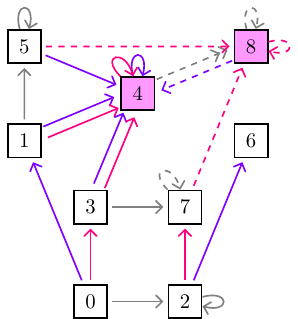}

			\caption{Step 8.}
			\label{subfig-bool-mat-felsch-8}
		\end{subfigure}
		\begin{subfigure}{0.30\textwidth}
			\centering
			
    \includegraphics{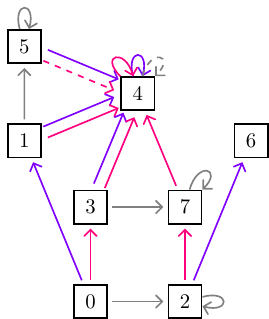}

			\caption{Step 9.}
			\label{subfig-bool-mat-felsch-9}
		\end{subfigure}
		\vspace{\baselineskip}
		\begin{subfigure}{0.30\textwidth}
			\centering
			
    \includegraphics{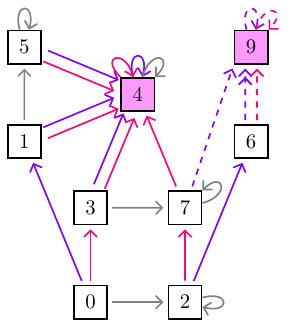}

			\caption{Step 10.}
			\label{subfig-bool-mat-felsch-10}
		\end{subfigure}
		\begin{subfigure}{0.30\textwidth}
			\centering
			
    \includegraphics{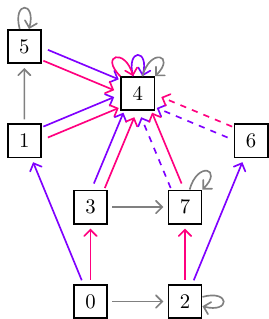}

			\caption{Step 11.}
			\label{subfig-bool-mat-felsch-11}
		\end{subfigure}
		\begin{subfigure}{0.30\textwidth}
			\centering
			
    \includegraphics{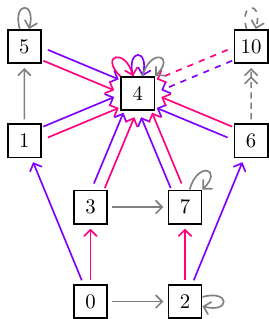}

			\caption{Step 12.}
			\label{subfig-bool-mat-felsch-12}
		\end{subfigure}
		\caption{The output $(\Gamma_{i}, \kappa_{i})$ of each step in
			\cref{ex-bool-mat-presentation}. Purple arrows correspond to $a$, grey to
			$b$, pink to $c$, shaded nodes of the same colour belong to $\kappa_{i}$,
			and unshaded nodes belong to singleton classes.}
		\label{fig-ex-bool-mat-presentation-felsch-B}
	\end{figure}

\end{ex}

\section{Performance comparison}
\label{appendix-benchmarks}
In this appendix the performance of the \libsemigroups implementations of the
HLT and Felsch strategies (as described in this paper) is compared to the
performance of GAP~\cite{GAP4} on a number of examples. Unless otherwise
indicated the times for \libsemigroups are the mean time of 100 runs and the
indicated times for GAP~\cite{GAP4} are the mean of
10 runs. All computations were performed on a 2021 Mac M1 computer with 16GB of
   RAM.

\subsection{Monoids of transformations}

In this section we present some tables of timings for some presentations from
the literature of transformation monoids; see Tables~\ref{table-orient}
and~\ref{table-orient-reverse}.

\subsection{Monoids of partitions}

In this section we present some tables of timings for some presentations from
the literature of monoids of partitions; see
Tables~\ref{table-partition},~\ref{table-singular-brauer},~\ref{table-dual-sym-inv},~\ref{table-uniform}, and~\ref{table-temperley-lieb}.

\subsection{Further finitely presented semigroups and monoids}

In this section we present some tables of timings for some presentations from
the literature of finitely presented semigroups and monoids; see Tables~\ref{table-stellar},~\ref{table-stylic}, and~\ref{table-walker}.

The presentations referred to in \cref{table-walker} are:
\begin{align}\label{eq-walker-1}
  S = \langle a, b, c \mid       & \  a^{14} = a, b^{14} = b, c^{14} = c, a^4ba = b^3, b^4ab = a^3,
  a^4ca = c^3, c^4ac = a^3, b^4cb = c^3, c^4bc = b^3\rangle                                         \\
  \label{eq-walker-2}
  S = \langle a, b \mid          & \ a^{32} = a, b^3 = b, ababa = b, a^{16}ba^4ba^{16}ba^4 \rangle  \\
  \label{eq-walker-3}
  S = \langle a, b \mid          & \ a^{16} = a, b^{16} = b, ab^2 = ba^2 \rangle                    \\
  \label{eq-walker-4}
  S = \langle a, b \mid          & \  a^3 = a, b^6 = b, (abab^4)^7ab^2a = b^2\rangle                \\
  \label{eq-walker-5}
  S = \langle a, b \mid          & \ a^3 = a, b^6 = b, (abab^4)^7ab^2ab^5a^2= b^2\rangle            \\
  \label{eq-walker-6}
  S = \langle a, b \mid          & \ a^3 = a, b^9 = b, abab^7abab^7ab^2ab^8 = b^2\rangle            \\
  \label{eq-walker-7}
  S = \langle a, b, c, d, e \mid & \ a^3 = a, b^3 = b, c^3 = c, d^3 = d,
  e^3 = 3, (ab)^3 = a^2, (bc)^3 = b^2, (cd)^3 = c^2, (de)^3 = d^2,                                  \\
                            \nonumber     & \ ac = ca, ad = da, ae = ea, bd = db, be = eb, ce = ec \rangle   \\
  \label{eq-walker-8}
  S = \langle a, b \mid          & \ a^3 = a, b^{23} = b, ab^{11}ab^2 = b^2a \rangle.
\end{align}

\subsection{Finitely presented groups}

In this section we provide some comparison of the performance of
\libsemigroups, GAP~\cite{GAP4}, and ACE~\cite{Havas1999aa} when applied to
group presentations.  It might be worth noting that \libsemigroups contains no
optimizations whatsoever for group presentations, and that when run in
\libsemigroups each of the presentations given in this section has: an explicit
generator $a ^ {-1}$ for the inverse of each generator $a$ that is not of order
$2$; the relations $a a ^{-1} = a ^ {-1} a = e$ for every generator $a$. The
version of Todd-Coxeter in GAP~\cite{GAP4} is largely written in C, and is
specific to group presentations; similarly, ACE~\cite{Havas1999aa} is written
in C and is specific to groups also. As such it is not surprising that the
performance of \libsemigroups is generally worse that both
ACE~\cite{Havas1999aa} and GAP~\cite{GAP4} when applied to a group
presentation. There are still some cases where \libsemigroups is faster than
GAP~\cite{GAP4}; see \cref{table-ace}.

\begin{align}
   & \begin{aligned}
       \label{eq-ace-2p17-2p14}
       G & =  \langle a, b, c \mid ab^{-1}c^{-1}bac, ba^{-1}c^{-1}ba^2ca^{-1},
       ac^2a^{-1}a^{-1}b^{-1}ab \rangle                                        \\
       H & =  \genset{bc}                                                      \\
     \end{aligned}                                                  \\
   & \begin{aligned}
       \label{eq-ace-2p17-2p3}
       G & = \langle a, b, c\mid ab^{-1}c^{-1}bac, ba^{-1}c^{-1}ba^2ca^{-1},
       ac^2a^{-1}a^{-1}b^{-1}ab \rangle                                      \\
       H & = \genset{bc, a^{-1}b^{-1}(a^{-1})^2bcabc^{-1}}                   \\
     \end{aligned}                                                    \\
   & \begin{aligned}
       \label{eq-ace-2p17-1}
       G & = \langle a, b, c\mid ab^{-1}c^{-1}bac, ba^{-1}c^{-1}ba^2ca^{-1},
       ac^2a^{-1}a^{-1}b^{-1}ab \rangle                                          \\
       H & = \genset{ab ^ {-1}c ^ {-1}bac, ba ^ {-1}c ^ {-1}baaca ^ {-1}, acca ^
       {-1}a ^ {-1}b ^ {-1}ab}                                                   \\
     \end{aligned}                                                \\
   & \begin{aligned}
       \label{eq-ace-2p17-1a}
       G & = \langle a, b, c\mid ab^{-1}c^{-1}bac, ba^{-1}c^{-1}ba^2ca^{-1},
       ac^2a^{-1}a^{-1}b^{-1}ab \rangle                                      \\
       H & = \genset{bc, a^{-1}b^{-1}a^{-1}a^{-1}bcabc^{-1},
       a^{-1}c^3acb^{-1}ca^{-1}}                                             \\
     \end{aligned}                                                    \\
   & \begin{aligned}
       \label{eq-ace-2p17-id}
       G & = \langle a, b, c\mid ab^{-1}c^{-1}bac, ba^{-1}c^{-1}ba^2ca^{-1},
       ac^2a^{-1}a^{-1}b^{-1}ab \rangle                                      \\
       H & = \genset{\varepsilon}                                            \\
     \end{aligned}                                                    \\
   & \begin{aligned}
       \label{eq-ace-2p18}
       G & = \langle a, b, c, x \mid ab^{-1}c^{-1}bac, ba^{-1}c^{-1}ba^2ca^{-1},
       ac^2a^{-1}a^{-1}b^{-1}ab, x^2, a^{-1}xax, b^{-1}xbx, c^{-1}xcx \rangle    \\
       H & = \genset{ab^{-1}c^{-1}bac, ba^{-1}c^{-1}ba^2ca^{-1},
       ac^2a^{-1}a^{-1}b^{-1}ab}                                                 \\
     \end{aligned}                                                \\
   & \begin{aligned}
       \label{eq-ace-F27}
       G & = F_{2, 7}= \langle a, b, c, d, x, y, z \mid abc^{-1}, bcd^{-1},
       cdx^{-1}, dxy^{-1}, xyz^{-1}, yza^{-1}, zab^{-1} \rangle             \\
       H & = \genset{\varepsilon}                                           \\
     \end{aligned}                                                     \\
   & \begin{aligned}
       \label{eq-ace-M12}
       G & = M_{12} = \langle a, b, c \mid a^{11}, b^2, c^2, (ab)^3, (ac)^3,
       (bc)^{10}, (cb)^2abcbc(a^{-1})^5\rangle                               \\
       H & = \genset{\varepsilon}                                            \\
     \end{aligned}                                                    \\
   & \begin{aligned}
       \label{eq-ace-SL219}
       G & = SL(2, 19) = \langle a, b \mid ab^{-1}a^{-1}b^{-1}a^{-1}b^{-1},
       b^{-1}a^{-1}a^{-1}baa, ab^4ab^{10}ab^4ab^{29}a^{12} \rangle          \\
       H & =\genset{b}                                                      \\
     \end{aligned}                                                     \\
   & \begin{aligned}
       \label{eq-ace-big-hard}
         & \hspace{-1em}\begin{aligned}
                       G & = \langle a, b, c, x, y \ \mid &  & x^2, y^3, ab^{-1}c^{-1}bac,
                       ba^{-1}c^{-1}baaca^{-1}, acca^{-1}a^{-1}b^{-1}ab, a^{-1}xax, b^{-1}xbx,
                       c^{-1}xcx,                                                                                                \\
                         &                                &  & a^{-1}y^{-1}ay, b^{-1}y^{-1}by, c^{-1}y^{-1}cy, xy^{-1}xy \rangle
                     \end{aligned}
       \\
       H & = \genset{ab^{-1}c^{-1}bac, ba^{-1}c^{-1}ba^2ca^{-1}, ac^2a^{-2}b^{-1}ab}
         &                                                                                                              &  &      \\
     \end{aligned}
\end{align}

\input{anc/orientation-preserving-monoid.table}
\input{anc/orientation-reversing-monoid.table}

\input{anc/partition-monoid.table}
\input{anc/singular-brauer-monoid.table}
\begin{landscape}
\input{anc/dual-symmetric-inverse-monoid.table}
\input{anc/uniform-block-bijection-monoid.table}
\end{landscape}
\input{anc/temperley-lieb-monoid.table}

\input{anc/stellar-monoid.table}
\input{anc/stylic-monoid.table}
\input{anc/walker.table}

\input{anc/groups.table}

\end{document}

%% file: anc/orientation-preserving-monoid.table
\begin{table}[h]
\centering
\begingroup
\renewcommand*{\arraystretch}{1.2}
\begin{tabular}{l||r|r|r|r|c|c|c|c|c|c}
\multirow{2}{*}{$n$}
& \multirow{2}{*}{$|OP_n|$}
& \multirow{2}{*}{$|A|$}
& \multirow{2}{*}{$|R|$}
& \multirow{2}{*}{$\sum |uv|$}
& \multicolumn{2}{c|}{HLT}
& \multicolumn{2}{c|}{Felsch}
& \multicolumn{2}{c}{GAP}
\\\cline{6-11}
&
&
&
&
& mean
& s.d.
& mean
& s.d.
& mean
& s.d.
\\\hline\hline
    3&24&2&5&43&$8.4\times 10^{-6}$&$8\times 10^{-7}$&$1.84\times 10^{-5}$&$1.8\times 10^{-6}$&$5.3\times 10^{-4}$&$7\times 10^{-5}$\\
4&128&2&6&86&$4.0\times 10^{-5}$&$2\times 10^{-6}$&$2.77\times 10^{-4}$&$5\times 10^{-6}$&$1.92\times 10^{-3}$&$7\times 10^{-5}$\\
5&610&2&7&145&$4.51\times 10^{-4}$&$4\times 10^{-6}$&$4.565\times 10^{-3}$&$1.5\times 10^{-5}$&$1.492\times 10^{-2}$&$1.3\times 10^{-4}$\\
6&2,742&2&8&220&$4.64\times 10^{-3}$&$7\times 10^{-5}$&$4.153\times 10^{-2}$&$1.2\times 10^{-4}$&$1.174\times 10^{-1}$&$3\times 10^{-4}$\\
7&11,970&2&9&311&$4.97\times 10^{-2}$&$4\times 10^{-4}$&$4.0\times 10^{-1}$&$5\times 10^{-2}$&$8.081\times 10^{-1}$&$1.1\times 10^{-3}$\\
8&51,424&2&10&418&$7.11\times 10^{-1}$&$5\times 10^{-3}$&$3.07\times 10^{0}$&$4\times 10^{-2}$&$4.488\times 10^{0}$&$2\times 10^{-3}$\\
9&218,718&2&11&541&$5.40\times 10^{0}$&$3\times 10^{-2}$&$2.98\times 10^{1}$&$2\times 10^{-1}$&-&-\\
10&923,690&2&12&680&$2.503\times 10^{1}$&-&-&-&-&-\\
11&3,879,766&2&13&835&$2.680\times 10^{2}$&-&-&-&-&-\\
12&16,224,804&2&14&1,006&$2.351\times 10^{3}$&-&-&-&-&-\\
13&67,603,744&2&15&1,193&$3.160\times 10^{4}$&-&-&-&-&-\\
\end{tabular}
\endgroup
   \caption{Comparison of \libsemigroups and the implementation in the GAP~\cite{GAP4} library (\texttt{CosetTableOfFpSemigroup}) for the presentations of the monoids $OP_n$ of orientation preserving transformations of a chain from \cite{Arthur2000aa}. Values with no standard deviation indicated were only run once. All times are in seconds.}
   \label{table-orient}
\end{table}

%% file: anc/orientation-reversing-monoid.table
\begin{table}[h]
\centering
\begingroup
\renewcommand*{\arraystretch}{1.2}
\begin{tabular}{l||r|r|r|r|c|c|c|c|c|c}
\multirow{2}{*}{$n$}
& \multirow{2}{*}{$|OR_n|$}
& \multirow{2}{*}{$|A|$}
& \multirow{2}{*}{$|R|$}
& \multirow{2}{*}{$\sum |uv|$}
& \multicolumn{2}{c|}{HLT}
& \multicolumn{2}{c|}{Felsch}
& \multicolumn{2}{c}{GAP}
\\\cline{6-11}
&
&
&
&
& mean
& s.d.
& mean
& s.d.
& mean
& s.d.
\\\hline\hline
    3&27&3&9&65&$1.388\times 10^{-5}$&$9\times 10^{-8}$&$3.75\times 10^{-5}$&$3\times 10^{-7}$&$6.78\times 10^{-4}$&$1.2\times 10^{-5}$\\
4&180&3&10&122&$8.45\times 10^{-5}$&$6\times 10^{-7}$&$5.784\times 10^{-4}$&$1.7\times 10^{-6}$&$4.18\times 10^{-3}$&$1.5\times 10^{-4}$\\
5&1,015&3&11&199&$1.120\times 10^{-3}$&$1.0\times 10^{-5}$&$1.034\times 10^{-2}$&$5\times 10^{-5}$&$3.140\times 10^{-2}$&$8\times 10^{-5}$\\
6&5,028&3&12&296&$1.257\times 10^{-2}$&$6\times 10^{-5}$&$1.459\times 10^{-1}$&$3\times 10^{-4}$&$2.466\times 10^{-1}$&$2\times 10^{-4}$\\
7&23,051&3&13&413&$2.404\times 10^{-1}$&$1.2\times 10^{-3}$&$1.369\times 10^{0}$&$5\times 10^{-3}$&$1.8565\times 10^{0}$&$1.1\times 10^{-3}$\\
8&101,272&3&14&550&$4.94\times 10^{0}$&$3\times 10^{-2}$&$1.437\times 10^{1}$&$4\times 10^{-2}$&$1.302\times 10^{1}$&$4\times 10^{-2}$\\
9&434,835&3&15&707&$7.035\times 10^{0}$&-&-&-&-&-\\
10&1,843,320&3&16&884&$8.509\times 10^{1}$&-&-&-&-&-\\
11&7,753,471&3&17&1,081&$8.460\times 10^{2}$&-&-&-&-&-\\
12&32,440,884&3&18&1,298&$6.178\times 10^{3}$&-&-&-&-&-\\
\end{tabular}
\endgroup
   \caption{Comparison of \libsemigroups and the implementation in the GAP~\cite{GAP4} library (\texttt{CosetTableOfFpSemigroup}) for the presentations of the monoids $OR_n$ of orientation preserving and reversing transformations of a chain from \cite{Arthur2000aa}. Values with no standard deviation indicated were only run once. All times are seconds.}
   \label{table-orient-reverse}
\end{table}

%% file: anc/partition-monoid.table
\begin{table}[h]
\centering
\begingroup
\renewcommand*{\arraystretch}{1.2}
\begin{tabular}{l||r|r|r|r|c|c|c|c|c|c}
\multirow{2}{*}{$n$}
& \multirow{2}{*}{$|P_n|$}
& \multirow{2}{*}{$|A|$}
& \multirow{2}{*}{$|R|$}
& \multirow{2}{*}{$\sum |uv|$}
& \multicolumn{2}{c|}{HLT}
& \multicolumn{2}{c|}{Felsch}
& \multicolumn{2}{c}{GAP}
\\\cline{6-11}
&
&
&
&
& mean
& s.d.
& mean
& s.d.
& mean
& s.d.
\\\hline\hline
    4&4,140&4&19&144&$7.43\times 10^{-3}$&$1.2\times 10^{-4}$&$1.918\times 10^{-2}$&$5\times 10^{-5}$&$1.2919\times 10^{-1}$&$1.1\times 10^{-4}$\\
5&115,975&4&19&163&$5.69\times 10^{-1}$&$2\times 10^{-3}$&$1.082\times 10^{0}$&$7\times 10^{-3}$&$5.377\times 10^{0}$&$6\times 10^{-3}$\\
6&4,213,597&4&20&198&$3.24\times 10^{1}$&$3\times 10^{-1}$&$1.460\times 10^{2}$&$5\times 10^{-1}$&-&-\\
7&190,899,322&4&20&219&$2.240\times 10^{3}$&-&-&-&-&-\\
\end{tabular}
\endgroup
   \caption{Comparison of \libsemigroups and the implementation in the GAP~\cite{GAP4} library (\texttt{CosetTableOfFpSemigroup}) for the presentations of the partition monoids $P_n$ from \cite[Theorem 41]{East2011aa}. Values with no standard deviation indicated were only run once. All times are in seconds.}
   \label{table-partition}
\end{table}

%% file: anc/singular-brauer-monoid.table
\begin{table}[h]
\centering
\begingroup
\renewcommand*{\arraystretch}{1.2}
\begin{tabular}{l||r|r|r|r|c|c|c|c|c|c}
\multirow{2}{*}{$n$}
& \multirow{2}{*}{$|B_n\setminus S_n|$}
& \multirow{2}{*}{$|A|$}
& \multirow{2}{*}{$|R|$}
& \multirow{2}{*}{$\sum |uv|$}
& \multicolumn{2}{c|}{HLT}
& \multicolumn{2}{c|}{Felsch}
& \multicolumn{2}{c}{GAP}
\\\cline{6-11}
&
&
&
&
& mean
& s.d.
& mean
& s.d.
& mean
& s.d.
\\\hline\hline
    3&9&6&21&78&$1.72\times 10^{-5}$&$2\times 10^{-7}$&$2.4\times 10^{-5}$&$4\times 10^{-6}$&$4.6\times 10^{-4}$&$4\times 10^{-5}$\\
4&81&12&126&600&$1.955\times 10^{-4}$&$8\times 10^{-7}$&$6.27\times 10^{-4}$&$1.8\times 10^{-5}$&$5.50\times 10^{-3}$&$1.2\times 10^{-4}$\\
5&825&20&450&2,300&$5.33\times 10^{-3}$&$5\times 10^{-5}$&$1.968\times 10^{-2}$&$1.8\times 10^{-4}$&$1.60\times 10^{-1}$&$2\times 10^{-3}$\\
6&9,675&30&1,185&6,240&$1.359\times 10^{-1}$&$2\times 10^{-4}$&$6.59\times 10^{-1}$&$6\times 10^{-3}$&$4.81\times 10^{0}$&$4\times 10^{-2}$\\
7&130,095&42&2,583&13,818&$3.89\times 10^{0}$&$2\times 10^{-2}$&$2.190\times 10^{1}$&$1.2\times 10^{-1}$&-&-\\
8&1,986,705&56&4,956&26,768&$3.357\times 10^{2}$&-&-&-&-&-\\
\end{tabular}
\endgroup
   \caption{Comparison of \libsemigroups and the implementation in the GAP~\cite{GAP4} library (\texttt{CosetTableOfFpSemigroup}) on the presentations for the singular Brauer monoids $B_n \setminus S_n$ from \cite{Maltcev2007aa}. Values with no standard deviation indicated were only run once. All times are in seconds.}
   \label{table-singular-brauer}
\end{table}

%% file: anc/dual-symmetric-inverse-monoid.table
\begin{table}[h]
\centering
\begingroup
\renewcommand*{\arraystretch}{1.2}
\begin{tabular}{l||r|r|r|r|c|c|c|c|c|c|c|c}
\multirow{2}{*}{$n$}
& \multirow{2}{*}{$|I_n^*|$}
& \multirow{2}{*}{$|A|$}
& \multirow{2}{*}{$|R|$}
& \multirow{2}{*}{$\sum |uv|$}
& \multicolumn{2}{c|}{HLT}
& \multicolumn{2}{c|}{Felsch}
& \multicolumn{2}{c|}{Rc}
& \multicolumn{2}{c}{GAP}
\\\cline{6-13}
&
&
&
&
& mean
& s.d.
& mean
& s.d.
& mean
& s.d.
& mean
& s.d.
\\\hline\hline
    3&25&3&11&54&$1.5\times 10^{-5}$&$5\times 10^{-6}$&$1.92\times 10^{-5}$&$5\times 10^{-7}$&$1.850\times 10^{-5}$&$8\times 10^{-8}$&$4.94\times 10^{-4}$&$1.0\times 10^{-5}$\\
4&339&4&19&116&$4.17\times 10^{-4}$&$4\times 10^{-6}$&$6.05\times 10^{-4}$&$4\times 10^{-6}$&$1.104\times 10^{-3}$&$9\times 10^{-6}$&$8.74\times 10^{-3}$&$2\times 10^{-5}$\\
5&6,721&5&28&177&$3.15\times 10^{-2}$&$3\times 10^{-4}$&$2.780\times 10^{-2}$&$7\times 10^{-5}$&$2.926\times 10^{-2}$&$6\times 10^{-5}$&$3.213\times 10^{-1}$&$5\times 10^{-4}$\\
6&179,643&6&39&276&$3.639\times 10^{0}$&$9\times 10^{-3}$&$1.399\times 10^{0}$&$1.2\times 10^{-2}$&$4.044\times 10^{0}$&$1.3\times 10^{-2}$&$3.858\times 10^{1}$&$4\times 10^{-2}$\\
7&6,166,105&7&52&422&-&-&$1.115\times 10^{2}$&-&-&-&-&-\\
\end{tabular}
\endgroup
   \caption{Comparison of \libsemigroups and the implementation in the GAP~\cite{GAP4} library (\texttt{CosetTableOfFpSemigroup}) the presentations for the dual symmetric inverse monoids $I_n ^ *$ from \cite{Easdown2008aa}.This monoid is sometimes called the \textit{uniform block bijection monoid}. Values with no standard deviation indicated were only run once. All times are in seconds.}
   \label{table-dual-sym-inv}
\end{table}

%% file: anc/uniform-block-bijection-monoid.table
\begin{table}[h]
\centering
\begingroup
\renewcommand*{\arraystretch}{1.2}
\begin{tabular}{l||r|r|r|r|c|c|c|c|c|c|c|c}
\multirow{2}{*}{$n$}
& \multirow{2}{*}{$|FI_n^*|$}
& \multirow{2}{*}{$|A|$}
& \multirow{2}{*}{$|R|$}
& \multirow{2}{*}{$\sum |uv|$}
& \multicolumn{2}{c|}{HLT}
& \multicolumn{2}{c|}{Felsch}
& \multicolumn{2}{c|}{Rc}
& \multicolumn{2}{c}{GAP}
\\\cline{6-13}
&
&
&
&
& mean
& s.d.
& mean
& s.d.
& mean
& s.d.
& mean
& s.d.
\\\hline\hline
    3&16&3&8&47&$1.1\times 10^{-5}$&$4\times 10^{-6}$&$2.6\times 10^{-5}$&$4\times 10^{-6}$&$1.57\times 10^{-5}$&$1.4\times 10^{-6}$&$4.48\times 10^{-4}$&$1.7\times 10^{-5}$\\
4&131&4&13&67&$8.8\times 10^{-5}$&$4\times 10^{-6}$&$1.54\times 10^{-4}$&$1.1\times 10^{-5}$&$1.02\times 10^{-4}$&$4\times 10^{-6}$&$2.302\times 10^{-3}$&$1.3\times 10^{-5}$\\
5&1,496&5&20&95&$2.64\times 10^{-3}$&$5\times 10^{-5}$&$2.243\times 10^{-3}$&$2.0\times 10^{-5}$&$3.35\times 10^{-3}$&$1.0\times 10^{-4}$&$3.612\times 10^{-2}$&$5\times 10^{-5}$\\
6&22,482&6&29&131&$7.98\times 10^{-2}$&$4\times 10^{-4}$&$5.39\times 10^{-2}$&$4\times 10^{-4}$&$5.67\times 10^{-2}$&$2\times 10^{-4}$&$8.186\times 10^{-1}$&$1.3\times 10^{-3}$\\
7&426,833&7&40&175&$4.78\times 10^{0}$&$4\times 10^{-2}$&$1.608\times 10^{0}$&$1.3\times 10^{-2}$&$4.162\times 10^{0}$&$1.1\times 10^{-2}$&$1.466\times 10^{2}$&$1.0\times 10^{0}$\\
8&9,934,563&8&53&227&$8.030\times 10^{1}$&-&$5.572\times 10^{1}$&-&$8.193\times 10^{1}$&-&-&-\\
\end{tabular}
\endgroup
   \caption{Comparison of \libsemigroups and the implementation in the GAP~\cite{GAP4} library (\texttt{CosetTableOfFpSemigroup}) the presentations for the factorisable dual symmetric inverse monoids $FI_n ^ *$ from \cite{fitzgerald_2003}. This monoid is sometimes called the \textit{uniform block bijection monoid}. Values with no standard deviation indicated were only run once. All times are in seconds.}
   \label{table-uniform}
\end{table}

%% file: anc/temperley-lieb-monoid.table
\begin{table}[h]
\centering
\begingroup
\renewcommand*{\arraystretch}{1.2}
\begin{tabular}{l||r|r|r|r|c|c|c|c|c|c}
\multirow{2}{*}{$n$}
& \multirow{2}{*}{$|J_n|$}
& \multirow{2}{*}{$|A|$}
& \multirow{2}{*}{$|R|$}
& \multirow{2}{*}{$\sum |uv|$}
& \multicolumn{2}{c|}{HLT}
& \multicolumn{2}{c|}{Felsch}
& \multicolumn{2}{c}{GAP}
\\\cline{6-11}
&
&
&
&
& mean
& s.d.
& mean
& s.d.
& mean
& s.d.
\\\hline\hline
    3&5&2&4&14&$4.87\times 10^{-6}$&$6\times 10^{-8}$&$5.9\times 10^{-6}$&$4\times 10^{-7}$&$2.69\times 10^{-4}$&$1.2\times 10^{-5}$\\
4&14&3&8&29&$8.34\times 10^{-6}$&$8\times 10^{-8}$&$1.132\times 10^{-5}$&$1.0\times 10^{-7}$&$3.26\times 10^{-4}$&$6\times 10^{-6}$\\
5&42&4&13&48&$1.7\times 10^{-5}$&$2\times 10^{-6}$&$2.686\times 10^{-5}$&$1.5\times 10^{-7}$&$5.39\times 10^{-4}$&$1.1\times 10^{-5}$\\
6&132&5&19&71&$4.18\times 10^{-5}$&$6\times 10^{-7}$&$9.51\times 10^{-5}$&$1.0\times 10^{-6}$&$1.464\times 10^{-3}$&$1.1\times 10^{-5}$\\
7&429&6&26&98&$1.463\times 10^{-4}$&$1.4\times 10^{-6}$&$4.50\times 10^{-4}$&$4\times 10^{-6}$&$5.35\times 10^{-3}$&$6\times 10^{-5}$\\
8&1,430&7&34&129&$5.93\times 10^{-4}$&$4\times 10^{-6}$&$2.248\times 10^{-3}$&$1.0\times 10^{-5}$&$2.274\times 10^{-2}$&$5\times 10^{-5}$\\
9&4,862&8&43&164&$2.456\times 10^{-3}$&$9\times 10^{-6}$&$1.034\times 10^{-2}$&$5\times 10^{-5}$&$9.753\times 10^{-2}$&$1.4\times 10^{-4}$\\
10&16,796&9&53&203&$1.052\times 10^{-2}$&$4\times 10^{-5}$&$4.57\times 10^{-2}$&$4\times 10^{-4}$&$4.15\times 10^{-1}$&$3\times 10^{-3}$\\
11&58,786&10&64&246&$4.238\times 10^{-2}$&$1.4\times 10^{-4}$&$1.966\times 10^{-1}$&$1.4\times 10^{-3}$&$1.789\times 10^{0}$&$9\times 10^{-3}$\\
12&208,012&11&76&293&$1.805\times 10^{-1}$&$1.6\times 10^{-3}$&$8.52\times 10^{-1}$&$8\times 10^{-3}$&$8.73\times 10^{0}$&$6\times 10^{-2}$\\
13&742,900&12&89&344&$7.791\times 10^{-1}$&$7\times 10^{-4}$&$3.63\times 10^{0}$&$3\times 10^{-2}$&$4.40\times 10^{1}$&$4\times 10^{-1}$\\
14&2,674,440&13&103&399&$3.235\times 10^{0}$&$9\times 10^{-3}$&$1.517\times 10^{1}$&$1.1\times 10^{-1}$&-&-\\
15&9,694,845&14&118&458&$1.424\times 10^{1}$&-&-&-&-&-\\
16&35,357,670&15&134&521&$6.419\times 10^{1}$&-&-&-&-&-\\
\end{tabular}
\endgroup
   \caption{Comparison of \libsemigroups and the implementation in the GAP~\cite{GAP4} library (\texttt{CosetTableOfFpSemigroup}) on the presentations for the Temperley-Lieb monoids $J_n$ from \cite[Theorem 2.2]{East2021aa}; the Temperley-Lieb monoid is also sometimes referred to as the \textit{Jones monoid} in the literature. Values with no standard deviation indicated were only run once. All times are in seconds.}
   \label{table-temperley-lieb}
\end{table}

%% file: anc/stellar-monoid.table
\begin{table}[h]
\centering
\begingroup
\renewcommand*{\arraystretch}{1.2}
\begin{tabular}{l||r|r|r|r|c|c|c|c|c|c}
\multirow{2}{*}{$n$}
& \multirow{2}{*}{$|\operatorname{Stellar}(n)|$}
& \multirow{2}{*}{$|A|$}
& \multirow{2}{*}{$|R|$}
& \multirow{2}{*}{$\sum |uv|$}
& \multicolumn{2}{c|}{HLT}
& \multicolumn{2}{c|}{Felsch}
& \multicolumn{2}{c}{GAP}
\\\cline{6-11}
&
&
&
&
& mean
& s.d.
& mean
& s.d.
& mean
& s.d.
\\\hline\hline
    3&16&3&9&45&$9.9\times 10^{-6}$&$1.2\times 10^{-6}$&$1.43\times 10^{-5}$&$3\times 10^{-7}$&$4.10\times 10^{-4}$&$1.2\times 10^{-5}$\\
4&65&4&14&71&$2.239\times 10^{-5}$&$1.8\times 10^{-7}$&$5.35\times 10^{-5}$&$4\times 10^{-7}$&$9.69\times 10^{-4}$&$1.4\times 10^{-5}$\\
5&326&5&20&103&$1.017\times 10^{-4}$&$8\times 10^{-7}$&$4.30\times 10^{-4}$&$5\times 10^{-6}$&$4.80\times 10^{-3}$&$2\times 10^{-5}$\\
6&1,957&6&27&141&$8.29\times 10^{-4}$&$4\times 10^{-6}$&$4.10\times 10^{-3}$&$3\times 10^{-5}$&$3.660\times 10^{-2}$&$5\times 10^{-5}$\\
7&13,700&7&35&185&$9.27\times 10^{-3}$&$5\times 10^{-5}$&$4.192\times 10^{-2}$&$1.2\times 10^{-4}$&$3.310\times 10^{-1}$&$3\times 10^{-4}$\\
8&109,601&8&44&235&$1.554\times 10^{-1}$&$6\times 10^{-4}$&$4.603\times 10^{-1}$&$1.6\times 10^{-3}$&$3.473\times 10^{0}$&$3\times 10^{-3}$\\
9&986,410&9&54&291&$9.67\times 10^{0}$&$1.4\times 10^{-1}$&$5.62\times 10^{0}$&$5\times 10^{-2}$&-&-\\
10&9,864,101&10&65&353&-&-&$7.337\times 10^{1}$&-&-&-\\
11&108,505,112&11&77&421&-&-&$1.018\times 10^{3}$&-&-&-\\
\end{tabular}
\endgroup
   \caption{Comparison of \libsemigroups and the implementation in the GAP~\cite{GAP4} library (\texttt{CosetTableOfFpSemigroup}) for the presentations of the stellar monoids from \cite{Gay2019aa}. Values with no standard deviation indicated were only run once. All times are in seconds.}
   \label{table-stellar}
\end{table}

%% file: anc/stylic-monoid.table
\begin{table}[h]
\centering
\begingroup
\renewcommand*{\arraystretch}{1.2}
\begin{tabular}{l||r|r|r|r|c|c|c|c|c|c}
\multirow{2}{*}{$n$}
& \multirow{2}{*}{$|\operatorname{Stylic}(n)|$}
& \multirow{2}{*}{$|A|$}
& \multirow{2}{*}{$|R|$}
& \multirow{2}{*}{$\sum |uv|$}
& \multicolumn{2}{c|}{HLT}
& \multicolumn{2}{c|}{Felsch}
& \multicolumn{2}{c}{GAP}
\\\cline{6-11}
&
&
&
&
& mean
& s.d.
& mean
& s.d.
& mean
& s.d.
\\\hline\hline
    3&15&3&11&57&$1.07\times 10^{-5}$&$1.1\times 10^{-6}$&$1.767\times 10^{-5}$&$9\times 10^{-8}$&$4.7\times 10^{-4}$&$8\times 10^{-5}$\\
4&52&4&24&132&$2.97\times 10^{-5}$&$2\times 10^{-7}$&$8.64\times 10^{-5}$&$7\times 10^{-7}$&$1.07\times 10^{-3}$&$7\times 10^{-5}$\\
5&203&5&45&255&$1.309\times 10^{-4}$&$8\times 10^{-7}$&$6.69\times 10^{-4}$&$5\times 10^{-6}$&$4.39\times 10^{-3}$&$7\times 10^{-5}$\\
6&877&6&76&438&$7.98\times 10^{-4}$&$4\times 10^{-6}$&$4.926\times 10^{-3}$&$1.2\times 10^{-5}$&$2.722\times 10^{-2}$&$4\times 10^{-5}$\\
7&4,140&7&119&693&$5.52\times 10^{-3}$&$3\times 10^{-5}$&$3.594\times 10^{-2}$&$1.8\times 10^{-4}$&$1.899\times 10^{-1}$&$3\times 10^{-4}$\\
8&21,147&8&176&1,032&$4.033\times 10^{-2}$&$8\times 10^{-5}$&$2.673\times 10^{-1}$&$8\times 10^{-4}$&$1.4019\times 10^{0}$&$1.9\times 10^{-3}$\\
9&115,975&9&249&1,467&$3.048\times 10^{-1}$&$4\times 10^{-4}$&$2.041\times 10^{0}$&$8\times 10^{-3}$&$1.0803\times 10^{1}$&$1.5\times 10^{-2}$\\
10&678,570&10&340&2,010&$2.5358\times 10^{0}$&$1.4\times 10^{-3}$&$1.625\times 10^{1}$&$9\times 10^{-2}$&-&-\\
11&4,213,597&11&451&2,673&$2.155\times 10^{1}$&-&-&-&-&-\\
12&27,644,437&12&584&3,468&$2.503\times 10^{2}$&-&-&-&-&-\\
\end{tabular}
\endgroup
   \caption{Comparison of \libsemigroups and the implementation in the GAP~\cite{GAP4} library (\texttt{CosetTableOfFpSemigroup}) for the presentations of the stylic monoids from \cite{Abram2021aa}. Values with no standard deviation indicated were only run once. All times are in seconds.}
   \label{table-stylic}
\end{table}

%% file: anc/walker.table
\begin{table}[h]
\centering
\begingroup
\renewcommand*{\arraystretch}{1.2}
\begin{tabular}{l||r|r|r|r|c|c|c|c|c|c}
\multirow{2}{*}{$n$}
& \multirow{2}{*}{$|S|$}
& \multirow{2}{*}{$|A|$}
& \multirow{2}{*}{$|R|$}
& \multirow{2}{*}{$\sum |uv|$}
& \multicolumn{2}{c|}{HLT}
& \multicolumn{2}{c|}{Felsch}
& \multicolumn{2}{c}{GAP}
\\\cline{6-11}
&
&
&
&
& mean
& s.d.
& mean
& s.d.
& mean
& s.d.
\\\hline\hline
    \eqref{eq-walker-1}&1&3&9&99&$2.14\times 10^{-3}$&$1\times 10^{-4}$&$5.68\times 10^{-2}$&$3\times 10^{-4}$&$1.62\times 10^{-1}$&$3\times 10^{-3}$\\
\eqref{eq-walker-2}&14,911&4&6&31&$5.32\times 10^{-2}$&$3\times 10^{-4}$&$3.344\times 10^{-2}$&$1.6\times 10^{-4}$&$6.93\times 10^{0}$&$5\times 10^{-2}$\\
\eqref{eq-walker-3}&20,490&2&3&40&$2.454\times 10^{-1}$&$5\times 10^{-4}$&$2.382\times 10^{-2}$&$1.0\times 10^{-4}$&$5.19\times 10^{0}$&$3\times 10^{-2}$\\
\eqref{eq-walker-4}&36,412&5&6&31&$2.903\times 10^{-1}$&$9\times 10^{-4}$&$8.875\times 10^{-1}$&$7\times 10^{-4}$&$7.530\times 10^{0}$&$1.6\times 10^{-2}$\\
\eqref{eq-walker-5}&72,822&5&6&34&$4.708\times 10^{-1}$&$1.9\times 10^{-3}$&-&-&$2.401\times 10^{1}$&$9\times 10^{-2}$\\
\eqref{eq-walker-6}&78,722&6&7&32&$1.327\times 10^{-1}$&$5\times 10^{-4}$&$3.297\times 10^{0}$&$1.8\times 10^{-2}$&$6.922\times 10^{0}$&$9\times 10^{-3}$\\
\eqref{eq-walker-7}&153,500&5&15&76&$1.799\times 10^{-1}$&$6\times 10^{-4}$&$1.819\times 10^{-1}$&$1.7\times 10^{-3}$&$2.496\times 10^{0}$&$9\times 10^{-3}$\\
\eqref{eq-walker-8}&270,272&2&3&46&$8.002\times 10^{-1}$&$1.9\times 10^{-3}$&$2.67\times 10^{0}$&$3\times 10^{-2}$&$1.234\times 10^{1}$&$2\times 10^{-2}$\\
\end{tabular}
\endgroup
   \caption{Comparison of \libsemigroups and the implementation in the GAP~\cite{GAP4} library (\texttt{CosetTableOfFpSemigroup}) on examples from Walker~\cite{Walker1992aa}. All times are in seconds.}
   \label{table-walker}
\end{table}

%% file: anc/groups.table
\begin{table}[h]
\centering
\begingroup
\renewcommand*{\arraystretch}{1.2}
\begin{tabular}{l||r|r|r|r|c|c|c|c|c|c}
\multirow{2}{*}{$n$}
& \multirow{2}{*}{$|G:H|$}
& \multirow{2}{*}{$|A|$}
& \multirow{2}{*}{$|R|$}
& \multirow{2}{*}{$\sum |uv|$}
& \multicolumn{2}{c|}{libsemigroups}
& \multicolumn{2}{c|}{ACE}
& \multicolumn{2}{c}{GAP}
\\\cline{6-11}
&
&
&
&
& mean
& s.d.
& mean
& s.d.
& mean
& s.d.
\\\hline\hline
    \eqref{eq-ace-2p17-2p14}&16,384&6&9&34&$2.450\times 10^{-1}$&$1.2\times 10^{-3}$&$3.62\times 10^{-2}$&$1.1\times 10^{-3}$&$1.99\times 10^{-1}$&$2\times 10^{-3}$\\
\eqref{eq-ace-2p17-2p3}&8&6&9&34&$2.547\times 10^{-1}$&$5\times 10^{-4}$&$2.14\times 10^{-2}$&$8\times 10^{-4}$&$1.444\times 10^{-1}$&$4\times 10^{-4}$\\
\eqref{eq-ace-2p17-1}&131,072&6&9&34&$1.392\times 10^{0}$&$6\times 10^{-3}$&$1.88\times 10^{-1}$&$5\times 10^{-3}$&$1.89\times 10^{1}$&$2\times 10^{-1}$\\
\eqref{eq-ace-2p17-1a}&1&6&9&34&$1.037\times 10^{-1}$&$6\times 10^{-4}$&$1.25\times 10^{-2}$&$8\times 10^{-4}$&$3.08\times 10^{-2}$&$2\times 10^{-4}$\\
\eqref{eq-ace-2p17-id}&131,072&6&9&34&$1.80\times 10^{0}$&$2\times 10^{-2}$&$2.02\times 10^{-1}$&$2\times 10^{-3}$&$1.880\times 10^{1}$&$1.5\times 10^{-1}$\\
\eqref{eq-ace-2p18}&262,144&7&14&50&$3.467\times 10^{0}$&$4\times 10^{-3}$&$6.08\times 10^{-1}$&$3\times 10^{-3}$&$3.97\times 10^{1}$&$3\times 10^{-1}$\\
\eqref{eq-ace-F27}&29&14&21&49&$7.76\times 10^{-2}$&$2\times 10^{-4}$&$2.44\times 10^{-2}$&$1.6\times 10^{-3}$&$2.145\times 10^{-2}$&$4\times 10^{-5}$\\
\eqref{eq-ace-M12}&95,040&7&14&60&$2.364\times 10^{-1}$&$1.0\times 10^{-3}$&$5.63\times 10^{-2}$&$8\times 10^{-4}$&$1.0025\times 10^{-1}$&$1.8\times 10^{-4}$\\
\eqref{eq-ace-SL219}&180&4&7&83&$4.63\times 10^{-1}$&$3\times 10^{-3}$&$1.953\times 10^{-1}$&$6\times 10^{-4}$&$7.66\times 10^{0}$&$3\times 10^{-2}$\\
\eqref{eq-ace-big-hard}&786,432&9&21&73&$1.252\times 10^{1}$&$9\times 10^{-2}$&$2.305\times 10^{0}$&$1.5\times 10^{-2}$&$6.752\times 10^{1}$&$1.6\times 10^{-1}$\\
\end{tabular}
\endgroup
   \caption{Comparison of \libsemigroups, ACE~\cite{Havas1999aa}, and the implementation in the GAP~\cite{GAP4} library (\texttt{CosetTableFromGensAndRels}). All times are in seconds.}
   \label{table-ace}
\end{table}